\documentclass[review,onefignum,onetabnum]{siamart190516}
\usepackage{enumerate}
\usepackage{bm}
\usepackage{setspace}
\usepackage{makecell}
\usepackage{subfigure}
\usepackage{multirow}
\usepackage{booktabs}
\usepackage{amssymb,extarrows,amsmath}
\usepackage[linesnumbered,ruled,lined]{algorithm2e}
\usepackage{relsize}

\DeclareMathOperator*{\argmin}{arg\,min}
\DeclareMathOperator{\diver}{\textbf{div}} %
\DeclareMathOperator{\diverl}{\textbf{\emph{div}}} 

\DeclareMathOperator{\rot}{\textbf{rot}}

\makeatletter

\newcommand{\Rmnum}[1]{\expandafter\@slowromancap\romannumeral #1@}
\newtheorem{remark}{Remark}
\newtheorem{assumption}{Assumption}[section]
\makeatother
\graphicspath{{figs/}}

\title{Diffeomorphic Image Registration with An Optimal Control Relaxation and Its Implementation\thanks{Manuscript Version 1, \textbf{\color{blue}\sc \today}. \funding{This work was supported by the National Natural Science Foundation of China (NSFC Project number 11771369), also partly by grants from the outstanding young scholars of Education Bureau of Hunan Province, P. R. China (number 17B257) and Natural Science Foundation of Hunan Province, P. R. China (number 2018JJ2375, 2017SK2014, 2018XK2304).
}}}

\author{Jianping Zhang\thanks{School of Mathematics and Computational Science, Hunan National Center for Applied Mathematics, and Hunan Key Laboratory for Computation and Simulation in Science and Engineering, Xiangtan University, Xiangtan, Hunan 411105,
P. R. China (Corresponding author, \email{jpzhang@xtu.edu.cn}).}\and Yanyan Li\thanks{School of Mathematics and Computational Science, Xiangtan University, Xiangtan, Hunan 411105,
P. R. China (\email{liyanyan@smail.xtu.edu.cn}).}
}

\begin{document}
\maketitle
% REQUIRED
\begin{abstract}
Image registration has played an important role in image processing problems, especially in medical imaging applications. It is well known that when the deformation is large, many variational models cannot ensure diffeomorphism. In this paper, we propose a new registration model based on an optimal control relaxation constraint for large deformation images, which can theoretically guarantee that the registration mapping is diffeomorphic. We present an analysis of optimal control relaxation for indirectly seeking the diffeomorphic transformation of Jacobian determinant equation and its registration applications, including the construction of diffeomorphic transformation as a special space. We also provide an existence result for the control increment optimization problem in the proposed diffeomorphic image registration model with an optimal control relaxation. Furthermore, a fast iterative scheme based on the augmented Lagrangian multipliers method (ALMM) is analyzed to solve the control increment optimization problem, and a convergence analysis is followed. Finally, a grid unfolding indicator is given, and a robust solving algorithm for using the deformation correction and backtrack
strategy is proposed to guarantee that the solution is diffeomorphic. Numerical experiments show that the registration model we proposed can not only get a diffeomorphic mapping when the deformation is large, but also achieves the state-of-the-art performance in quantitative evaluations in comparing with other classical models.
\end{abstract}

%% REQUIRED

% REQUIRED

\pagestyle{myheadings}
\thispagestyle{plain}
\markboth{Optimal Control Relaxation of Diffeomorphic Image Registration and its Solution}
{Jianping Zhang and Yanyan Li}

\section{Introduction}\label{sec:Section1}
Image registration refers to the process of finding a spatial transformation between two or more images according to a certain registration measure, so that the corresponding points of two images can be spatially consistent. Image registration is one of the most basic problems in computer vision and also is a difficult problem in medical image analysis. It has been widely used in target tracking, geological exploration, disease detection, angiography and other fields \cite{CChen2019,HFLi2020,Liu2011Total,Modersitzki2004Numerical,JZhang2016}. Therefore, medical image registration has a vital role in promoting the research of clinical medicine \cite{Behan20062,hill2001medical,MansiiLogDemons,SotirasDeformable}.

Non-rigid registration is suitable for the non-rigid object deformation, and can achieve local deformation more accurately. The commonly used models of non-rigid registration include: total variation (TV) model \cite{10,Frohn}, diffusion model \cite{fischer2002fast}, curvature model \cite{fischer2003curvature}, elastic model \cite{broit1981optimal,Fischler}, viscous fluid model \cite{christensen1997volumetric} and optical flow field model \cite{Modersitzki2004Numerical} etc. 
 Fischler and Elschlager \cite{Fischler} proposed an embedding elastic metric to constrain relative movement. 
%the linear elastic smoother was introduced to image registration by Broit \cite{broit1981optimal} and Bajcsy and Kovacic \cite{bajcsy1989multiresolution}.
% The elastic model was first proposed by Broit \cite{broit1981optimal} , 
Bajcsy and Kovacic \cite{bajcsy1989multiresolution} used affine transformation to correct large global differences before elastic registration. 
%, and implemented it with a multi-resolution method \cite{bajcsy1989multiresolution}.
 But the elastic model does not allow for large image deformations due to its local nature. Christensen et al. \cite{christensen1997volumetric,christensen1996deformable} proposed a viscous fluid regularization term to replace the elastic energy. Such model can deal with large deformation, but it can not generate diffeomorphism. In addition, high computational complexity seriously hampers its application. The Demons registration model was first proposed by Thirion \cite{thirion1998image}. It is a non-parametric non-rigid registration model that is improved based on the optical flow equation, which is more efficient than other non-rigid registration methods. As pointed out by Fischer and Modersitzki \cite{fischer2002fast}, implementation for diffusion registration is based on finite difference approximation of a diffusion-like equation. One of the most competitive features of this method is its fast speed, which makes such scheme very attractive for high-resolution applications. Curvature-based image registration is a fully automatic and non-rigid registration method proposed by Modersitzki \cite{fischer2003curvature}, which relies on a curvature based regularization term and provides a smooth solution. In contrast to many other nonlinear registration techniques, the curvature model does not require an additional affine linear pre-registration step. Zhang and Chen \cite{JPZhang2015} presented a novel variational framework of the total fractional-order variation to supervise nonlocal deformation in non-rigid image registration. These registration models can be effective for relatively small deformation, but not all models are effective for large deformation, and in particular few models can be guaranteed to generate diffeomorphic mapping.

%-----------------------------------------
 Vercauteren \cite{Vercauteren2007Non} proposed a non-parametric diffeomorphic demons image registration algorithm. Christensen et al. \cite{christensen1996deformable} developed a large deformation model for computing transformations, which overcomes the limitations of the small deformation model and ensures that the transformations are diffeomorphic. Later, large deformation diffeomorphic metric mapping (LDDMM) framework that can generate the diffeomorphic transformation of image registration was proposed \cite{beg2005computing,MangA,TrouvDiffeomorphisms}. Lam and Lui \cite{KCLam2014} introduced a novel model involving a Beltrami coefficient term to obtain diffeomorphic image or surface registrations via quasi-conformal maps, which can deal with large deformations. In \cite{YTLee2016}, a notion of conformality distortion of a diffeomorphism in the $n$-dimensional Euclidean space has also been formulated.
%--------------------------------

In the past few years, more and more scholars have paid attention to diffeomorphic image registration. The folding degree of deformed mesh in diffeomorphic image registration, measured by the local quantity ${\det (\nabla{\bm{\varphi}} (\bm{x}))}$, should be reduced or avoided, where ${\bm{\varphi}} (\bm{x})$ is the transformation in image registration, and ${\det (\nabla{\bm{\varphi}} (\bm{x}))}$ is the Jacobian determinant of ${\bm{\varphi}} (\bm{x})$, especially ${\det (\nabla{\bm{\varphi}} (\bm{x}))} >0$ restricts that the transformation ${\bm{\varphi}} (\bm{x})$ is an one-to-one mapping.

Haber and Modersitzki \cite{HaberNumerical} in 2004 proposed a volume-preserving registration model by limiting that ${\det (\nabla{\bm{\varphi}} (\bm{x}))}$ equals to $1$. Although such  incompressibility has important applications in some fields, it is not necessary or reasonable in others. Karacal et al. \cite{karacal2004} modified the deformation constraint to satisfy ${\det (\nabla{\bm{\varphi}} (\bm{x}))}>0$, which ensures the topology preservation of the deformation field ${\bm{\varphi}} (\bm{x})$. Yanovsky et al. \cite{Yanovsky2007Log} applied the symmetric Kullback-Leibler distance to quantify ${\det (\nabla{\bm{\varphi}} (\bm{x}))}$ for seeking a diffeomorphic mapping. Chen and Oktem \cite{CChen2018} developed a variational framework for indirect image registration, where a template is registered against a target that is known only through indirect noisy observations. Zhang and Chen \cite{Daoping2018} proposed a novel diffeomorphic model for image registration that establishes a link between the Beltrami coefficient of the transformation ${\bm{\varphi}} (\bm{x})$ and the quantity ${\det (\nabla{\bm{\varphi}} (\bm{x}))}$.

Liao \cite{liao2008optimal} in 2008 proposed the optimal control method for data alignment. The div-curl system was used as a constraint term to ensure the diffeomorphism of the deformation field. Hsiao \cite{hsiao2014new} proposed a new registration method based on the div-curl system, and obtained a Laplace equation about deformation field that satisfies the homogeneous Dirichlet boundary condition. A gradient descent method is given to optimize the transformation $\bm{\varphi}$.

In this paper, we aim to reformulate an optimal control relaxation dealing with the partial differential equation of Jacobian determinant ${\det(\nabla{\bm{\varphi}}(\bm{x}))}= f(\bm{\varphi}(\bm{x}))$
proposed by Dacorogna and Moser \cite{dacorogna1990partial}, to model diffeomorphic image registration problem. Our contributions are summarized as follows.
\begin{itemize} %enumerate}[1)]
\item[-] We propose a novel optimal control relaxation technique solving the diffeomorphic image registration model (denoted as OCRDIR) for large deformation images, that establishes a link between the registration transformation and the Jacobian equation, which can theoretically guarantee that the registration mapping is smooth, invertible and also diffeomorphic.

\item[-] We present an analysis of optimal control relaxation for indirectly seeking the diffeomorphic transformation of Jacobian determinant equation. We also provide an existence result for the minimum of the control increment non-convex optimization in our OCRDIR model.

\item[-] We propose a fast ALMM iterative scheme to solve the control increment optimization problem. The main novelty is the applicability of a general framework of constraint
optimization tools, which we can apply to reformulate the minimization problem. We also analyze the convergence of the proposed scheme. 
\item[-] We give a grid unfolding indicator for cell central-difference discretization, then a robust ALMM-based OCRDIR algorithm for using the deformation correction and backtrack
strategy is proposed to guarantee that the solution is diffeomorphic. Numerical experiments show that the  proposed method can not only get a diffeomorphic mapping when the deformation is large, but also achieves the state-of-the-art performance in quantitative evaluations in comparing with other classical models.
\end{itemize}

The rest of the paper is organized as follows. In Section \ref{sec:Section2} we review some methods
of image registration modeling. In Section
\ref{sec:Section3}, we build some optimal control relaxation theories to indirectly seek the diffeomorphic transformation of Jacobian determinant equation. Our new ALMM-based OCRDIR method is proposed and the convergence analysis is followed in Section \ref{sec:Section4}. We show
experimental results in Section \ref{sec:Section5}. Finally we conclude this work in Section \ref{sec:conclusions}.

\section{Reviews}\label{sec:Section2} The task of finding a suitable transformation $\bm{\varphi}: \Omega\subset \mathbb{R}^d\rightarrow \mathbb{R}^d (d=2\text{ or }3)$ such that a transformed version $T\big(\bm{\varphi}(\bm{x})\big)$ of the \emph{template} image $T(\bm{x})$ is similar to the \emph{reference}
image $R(\bm{x})$, is to solve a general minimization problem
\begin{equation}\label{similarity_measure}
\min\limits_{\bm{\varphi}} \mathcal{D}[R,T;\bm{\varphi}], 
\end{equation}
where $\Omega$ is a bounded closed set and $\bm{u}(\bm{x})=\bm{\varphi}(\bm{x})-\bm{x}$ is the increment field (i.e., $T\big(\bm{\varphi}(\bm{x})\big)=T\big({\bm{x}} + {\bm{u}(\bm{x})}\big)$), $\mathcal{D}[R,T;\bm{\varphi}]$ is the similarity measure. It is well-known that the problem (\ref{similarity_measure}) is ill-posed \cite{Modersitzki2004Numerical}, therefore, the classical mathematical model of image registration can be given by
\begin{equation*}
\min\limits_{\bm{\varphi}}\left\{\mathcal{E}\big(\bm{\varphi}\big) := \mathcal{D}[R,T;\bm{\varphi}]+ \tau \mathcal{S}[\bm{\varphi}]\right\},
\end{equation*}
where $\mathcal{S}[\bm{\varphi}]$ is a regularization term, and $\tau$ is a regularization parameter for managing the trade-off between $\mathcal{D}[R,T;\bm{\varphi}]$ and $\mathcal{S}[\bm{\varphi}]$. Throughout the paper, we choose the so called sum of squared differences (SSD) \cite{Modersitzki2004Numerical} defined by
\begin{equation*}\label{eq:eqssd}
\begin{split}
\mathcal{D}_{_\text{SSD}}[R,T;{\bm{\varphi}}] =& \frac{1}{2}\int_\Omega  {{{(T({\bm{\varphi}}({\bm{x}})) - R({\bm{x}}))}^2}d{\bm{x}}}
= \frac{1}{2}\int_\Omega  {{{(T({\bm{x}} + {\bm{u}}({\bm{x}})) - R({\bm{x}}))}^2}d{\bm{x}}},
\end{split}
\end{equation*}
 as similarity measure between $T$ and $R$. We also refer readers to literatures \cite{hermosillo2002variational,roche,MooreQuality} for many other similarity measures in mono-modality or multi-modality image registration.

\subsection{Image registration modeling} 
 As remarked, various ways of modeling non-rigid
image registration have been proposed. Below we review a few of them.
\subsubsection*{Active demons} Demons algorithm is a well-known non-parametric image registration method proposed by Thirion \cite{thirion1998image}, it combines the optical flow image matching idea with the thermodynamic molecular diffusion theory proposed by Maxwell to image registration. For any point $\bm{x}\in\Omega$ in the \emph{reference} image $R$, the matching point $\bm{\varphi}({\bm{x}})=\bm{x}+\bm{u}(\bm{x})\in\Omega$ is found one-to-one in the \emph{template} image $T$ such that $\mathcal{D}[R,T;\bm{\varphi}({\bm{x}})]=\frac{1}{2}{\int_\Omega}\big[T\big(\bm{x}+\bm{u}(\bm{x})\big)-R(\bm{x})\big]^2d{\bm{x}}\approx 0$, the second-order Taylor expansion of $\mathcal{D}[R,T;\bm{\varphi}({\bm{x}})]$ at the origin $\bm{u}(\bm{x})=\bm{0}$ is obtained
\begin{equation*}
\begin{split}
\mathcal{D}[R,T;\bm{\varphi}({\bm{x}})]=&\cfrac{1}{2}\mathlarger{\int_\Omega}\big[T\big(\bm{x}+\bm{u}(\bm{x})\big)-R(\bm{x})\big]^2d{\bm{x}}\\
\approx &\cfrac{1}{2}\mathlarger{\int_\Omega}\Big[\big(T(\bm{x})-R(\bm{x})\big)^2+2\big(T(\bm{x})-R(\bm{x})\big)\nabla T(\bm{x})\cdot\bm{u}(\bm{x})\\
&+\big(\nabla T(\bm{x})\cdot\bm{u}(\bm{x})\big)^2+\big(T(\bm{x})-R(\bm{x})\big)\big({\bm{u}(\bm{x})}^T\nabla^2 T(\bm{x})\bm{u}(\bm{x})\big)\Big]d{\bm{x}}.
\end{split}
\end{equation*}
The second-order derivative term $\big(T(\bm{x})-R(\bm{x})\big)\big({\bm{u}(\bm{x})}^T\nabla^2 T(\bm{x})\bm{u}(\bm{x})\big)$ is omitted from the above formula, then using the first-order variation and \textbf{Sherman-Morrison} formula, one has the minimum point
\[\bm{u}(\bm{x}) = \cfrac{(T(\bm{x}) - R(\bm{x}))\nabla T(\bm{x})}{|\nabla T(\bm{x})|^2 }.\]
Since the \emph{template} image $T(\bm{x})$ in the flat region will make $|\nabla T(\bm{x})|$ small or equal to zero, resulting in the instability of the algorithm, Thirion adds a constraint term $(T(\bm{x})- R(\bm{x}))^2$ in the denominator to ensure the stability of the increment field, that is
\begin{equation}\label{eq:eq36}
\begin{array}{*{20}{c}}
\bm{u}(\bm{x}) = \cfrac{(T(\bm{x}) - R(\bm{x}))\nabla T(\bm{x})}{|\nabla T(\bm{x})|^2 + \alpha(T(\bm{x}) - R(\bm{x}))^2},
\end{array}
\end{equation}
where ${\nabla T(\bm{x})}$ is the force of the internal boundary, and $T(\bm{x}) - R(\bm{x})$ is the force of the external images.

The traditional demons only uses the gradient information of \emph{template} image $T(\bm{x})$  to drive image deformation, when $\nabla T(\bm{x})$ is insufficient, it may cause the undesired registration errors. In addition, when the object deformation is large, the demons algorithm may not complete the registration requirements, even the convergence speed will be very slow. To deal with these difficulties, Wang et al. used the ideas of Newton force and reaction force to propose the active demons algorithm \cite{wang2005validation}, which adds the gradient information of the \emph{reference} image to obtain following formula
\begin{equation}\label{eq:eq37}
\begin{array}{*{20}{c}}
{\bm{u}(\bm{x})} = \cfrac{{(T - R)\nabla R}}{{{{(\nabla R)}^2} + {\tau ^2}{{(T - R)}^2}}} + \cfrac{{(T - R)\nabla T}}{(\nabla T)^2 + {\tau ^2}{{(T - R)}^2}},
\end{array}
\end{equation}
where the normalization factor $\tau$ is proposed by Cachier et al. \cite{cachier1999fast} to adjust the force strength.
%-------------------------------------
\subsubsection*{LDDMM framework}
The framework of large deformation diffeomorphic metric mapping (LDDMM) can generate the diffeomorphic transformation for image registration \cite{beg2005computing,TrouvDiffeomorphisms,SotirasDeformable}. The variational minimization of LDDMM framework can be defined as follow \cite{MangA}:
\begin{equation}\label{lddmm_minimization}
\left\{\begin{array}{l}
\min\limits_{\bm{v},\mathcal{U}( \cdot ,1)} \left\{ \mathcal{J}(\bm{v},\mathcal{U}( \cdot ,1)): = \mathcal D[\mathcal{U}( \cdot ,1),R] + \tau \mathcal{S}[\bm{v}]\right\},  \\
\text{ s.t. }\mathcal{C}(\bm{v}( \bm{x} ,t),\mathcal{U}( \bm{x} ,t)) = 0,\;\text{ on } \Omega \times [0,1].
\end{array} \right.
\end{equation}
where $\mathcal D$ is a similarity measure, $\mathcal{U}(\cdot,\cdot): \Omega\times [0,1]\to \mathbb{R}$ is a time series of images, $\bm{v}(\cdot,\cdot): \Omega\times [0,1]\to \mathbb{R}^d$ is the velocity field. The regularization of the velocity field $\bm{v}(\bm{x} ,t)$ can be defined by diffusion regularizer or curvature regularizer as follows
\begin{equation*}
{\mathcal{S}^\emph{diff}}[\bm{v}]=\cfrac{1}{2}{\mathlarger\int_0^1} {\mathlarger\int_\Omega }\sum\limits_{l = 1}^d   {{{|\nabla {v_l}|}^2}}d{\bm{x}}dt,
\quad
{\mathcal{S}^\emph{curv}}[\bm{v}]=\cfrac{1}{2}{\mathlarger\int_0^1} {\mathlarger\int_\Omega }\sum\limits_{l = 1}^d   {{{|\Delta {v_l}|}^2}}d{\bm{x}}dt.
\end{equation*}
 The constraints $\mathcal C(\bm{v}(\bm{x} ,t),\mathcal{U}( \bm{x} ,t))=0$ is given by the transport equation
\begin{equation*}
\mathcal C(\bm{v}(\bm{x} ,t),\mathcal{U}( \bm{x} ,t))=0\quad \Longleftrightarrow\quad \left\{
\begin{array}{l}
{\partial _t}\mathcal{U}( \bm{x},t) + \bm{v}(\bm{x},t) \cdot \nabla \mathcal{U}( \bm{x} ,t) = 0 ,   \\
\mathcal{U}( \bm{x} ,0) = T(\bm{x}).        
\end{array} \right.
\end{equation*}
The above intensity $\mathcal{U}(\cdot,\cdot)$ is constant along the characteristic curve
$\bm{\varphi}(\cdot,t)$ ($t\in [0,1]$), i.e., $\mathcal{U}(\bm{\varphi}(\bm{x},t),t)= T(\bm{x})$ for all $\bm{x}\in \Omega$. Hence the deformation function $\bm{\varphi}(\bm{x},t)$ satisfies                              
\begin{equation}\label{eq:curve_evolving}
\partial_{t} \bm{\varphi}(\bm{x},t)=\bm{v}(\bm{x}, t):=\bm{\bar{v}}\left(\bm{\varphi}(\bm{x},t), t\right) \quad \text { and } \quad \bm{\varphi}(\bm{x},0)=\bm{\varphi}_0(\bm{x})=\bm{x}.
\end{equation}
The characteristic curve equation (\ref{eq:curve_evolving}) can be computed by the fourth-order Runge-Kutta (RK4) scheme of ODE \eqref{eq:curve_evolving} as follow
\begin{equation}\label{eq_RK4}
\bm{\varphi}_{n+1}=\textbf{RK4}(\bm{\varphi}_n,\bm{\bar {v}},\Delta t)\Leftrightarrow\left\{\begin{array}{l}
\bm{\varphi}_{n+1}=\bm{\varphi}_{n}+\frac{\Delta t}{6}\left(k_{1}+2 k_{2}+2 k_{3}+k_{4}\right), \\
k_{1}=\bm{\bar {v}}\left(\bm{\varphi}_n, t_n\right), \\
k_{2}=\bm{\bar {v}}\left(\bm{\varphi}_{n}+\frac{\Delta t}{2}k_{1}, t_{n}+\frac{\Delta t}{2} \right), \\
k_{3}=\bm{\bar {v}}\left(\bm{\varphi}_{n}+\frac{\Delta t}{2}k_{2}, t_{n}+\frac{\Delta t}{2} \right), \\
k_{4}=\bm{\bar {v}}\left(\bm{\varphi}_{n}+\Delta tk_{3}, t_{n}+\Delta t \right),
\end{array}\right.
\end{equation}
and the variational minimization (\ref{lddmm_minimization}) is iteratively solved to obtain the final state ${\bm\varphi}_{1}(\bm{x}):={\bm\varphi}(\bm{x},1)$, and the other interpolation methods are also provided in an image registration toolbox FAIR \cite{Fair} to obtain the deformed image $\mathcal{U}(\bm{\varphi}(\bm{x},1),1)= T({\bm\varphi}_{1})$ \cite{MangA}.

\subsubsection*{H. Y. Hsiao registration}
A new non-rigid registration algorithm was proposed by H. Y. Hsiao et al. \cite{hsiao2014new} which is based on the deformation method of numerical grid generation for computational fluid dynamics. The registration transformation is generated by iteratively minimizing
\begin{equation}\label{eq:eq44}
\begin{array}{*{20}{c}}
\min\limits_{{\bm{\varphi}}}\mathcal{D}_{_\text{SSD}}[R,T;{\bm{\varphi}}]= \cfrac{1}{2}\mathlarger{\int_\Omega}  {\Big(T\big(\bm{\varphi} ({\bm{x}})\big)}  - R({\bm{x}})\Big)^2d{\bm{x}},
\end{array}
\end{equation}
an iterative scheme of transformation $\bm{\varphi}(\bm{x})$ is updated by
\begin{equation}\label{eq:eq45}
\begin{array}{*{20}{c}}
\bm{\varphi}(\bm{x}) = {\bm{\varphi} _{\text{old}}}(\bm{x}) +  \frac{{{\bm{u}(\bm{x})}}}{{1 + \diver\;  {\bm{u}(\bm{x})}}}\Delta t ,
\end{array}
\end{equation}
the displacement field $\bm{u}(\bm{x})$ satisfies partial differential
equation or div-curl systems
\begin{equation}\label{eq:eq46}
\left\{ {\begin{array}{rll}
  \diver {\bm{u}(\bm{x})}  &= f(\bm{x}) - 1,& \text{ in }\Omega  \\
  \rot{\bm{u}(\bm{x})}  &= g ,&\text{ in }\Omega \\
  ~~~{\bm{u}(\bm{x})} &={\bm{0}},&\text{ on }\partial\Omega.
  \end{array}} \right.
\quad \Longrightarrow\quad 
  \left\{ {\begin{array}{rll}
\Delta \bm{u}&=\mathbf{F},&\text{ in }\Omega\\
{\bm{u}(\bm{x})} &={\bm{0}},&\text{ on }\partial\Omega.
\end{array}} \right.
\end{equation}

The another main technical achievement of this algorithm is the derivation of the gradient of $ \mathcal{D}_{_\text{SSD}}$ with respect to $\mathbf{F}=(f_{x1}-g_{x2}, f_{x2}-g_{x1})$. The derivation of $\partial \mathcal{D}_{_\text{SSD}} / \partial \mathbf{F}$ is used  as gradient to update $\mathbf{F}$ by  $\mathbf{F}=\mathbf{F}_{\text {old }}+(\partial \mathcal{D}_{_\text{SSD}} / \partial \mathbf{F}) \Delta s$, where $\Delta s$ is the step size (also \emph{refer to} \cite{hsiao2014new} for more details).

\section{Diffeomorphic transformation}\label{sec:Section3}
The main use of Jacobian differential equation, which is the determinant of the Jacobian matrix of a vectorial function $\bm{\varphi} (\bm{x})\in \mathbb{R}^2$, is found in the transformation of coordinates. While the diffeomorphism is an important topic in differential geometry,  which is developed to deal with the transformation between two manifolds, in general, two different coordinate systems. 

\subsection{Analysis} Dacorogna and Moser in 1990 proved the existence and regularity of diffeomorphisms based on Jacobian equation \cite{dacorogna1990partial}. To proceed, we briefly review as follows:
\begin{theorem}[Diffeomorphic mapping \cite{dacorogna1990partial}]\label{thm01} 
Let $\ell\geq 0$ be an integer, $0<\epsilon<1$, $d\geq 2$, and a bounded closed set $\Omega\subset \mathbb{R}^d$ have a $\mathcal{C}^{\ell+3,\;\epsilon}$ boundary $\partial \Omega$ ($\mathcal{C}^{\ell,\;\epsilon}$ denotes the usual H{\"o}lder spaces). $g,\; g_{_0} \in \mathcal{C}^{\ell,\; \epsilon}(\Omega)$ with $g$, $g_{_0}>0 $ in $\Omega_{\text{in}}:=\Omega/ \partial\Omega $. Then there exists at least one diffeomorphism $\bm{\varphi}$ with
$\bm{\varphi}$, $\bm{\varphi}^{-1} \in \mathcal{C}^{\ell+1,\; \epsilon}\left(\Omega;\;\mathbb{R}^d\right)$ and satisfying
\begin{equation}\label{eq_diffeo}
\left\{\begin{array}{rll}
g(\bm{\varphi} (\bm{x}))\det \big(\nabla \bm{\varphi} (\bm{x})\big) &=\lambda g_{_0}(\bm{x}),\qquad &\bm{x} \in \Omega_{\text{in}} \\
\bm{\varphi} (\bm{x}) &= \bm{x},  \qquad   &\bm{x} \in \partial \Omega,
\end{array} \right.
\end{equation}
where $\lambda=\cfrac{\int_{\Omega} g(\bm{\varphi} (\bm{x})) d\bm{x}}{\int_{\Omega}g_{_0}(\bm{x}) d\bm{x}}$.
\end{theorem}

Note that the above Jacobian system \eqref{eq_diffeo} may have several smooth solutions \cite{dacorogna1990partial}. For more general measures, for instance for those with a mass density, the above problem is extended to the special setting of two mass measures $g_0(\cdot)$ and $g(\cdot)$ on two manifolds $\mathcal{X}$ and $\mathcal{Y}$ as finding two smooth one-to-one maps $\bm{\varphi}:\mathcal{X}\rightarrow \mathcal{Y}$, $\bm{\varphi}^{-1}:\mathcal{Y}\rightarrow \mathcal{X}$  such that
\[\int_{\mathcal{Z}} g_0(\bm{x})d\bm{x}=\int_{\mathcal{Z}} g(\bm{\varphi}(\bm{x}))\det \big(\nabla \bm{\varphi} (\bm{x})\big) d\bm{x},\;\;\text{for any}\; \mathcal{Z}\subset {\mathcal{X}}.\]
 Further, the constraint \eqref{eq_diffeo} can be used to perform image registration 
because it does operate as an image warping method with a diffeomorphic transformation $\bm{\varphi}(\bm{x})$. Intuitively, a diffeomorphic mapping $\bm{\varphi}:\mathcal{X}\rightarrow \mathcal{Y}$, can be interpreted as a function moving a single point $\bm{x}\in\mathcal{X}$ from a coordinate system to another. It is also important to realize that the mass $g_0(\bm{x})$ defined on $\mathcal{X}$ may be a known prior but the mass $g(\bm{\varphi}(\bm{x}))$ defined on $\mathcal{Y}$ should be unknown.

\begin{remark}[Variants of Jacobian system]\label{rmk1} In the above theorem, the system \eqref{eq_diffeo} seeking a diffeomorphic mapping $\bm{\varphi}$ can be rewritten as the following variants:   
\begin{description}
\item[Case 1.] assume that $g(\bm{\varphi} (\bm{x}))=1$ for all $\bm{x}\in\Omega_{\text{in}}$, $ \int_\Omega  {g_0(\bm{x})\;d\bm{x} = |\Omega |} $, and $\lambda=1$. Thus the Jacobian equation (\ref{eq_diffeo}) can be written as:
\begin{equation}\label{eq:eqb11}
\left\{\begin{array}{rll}
\det \big(\nabla \bm{\varphi} (\bm{x})\big) &= g_0(\bm{x}),\qquad &\bm{x} \in \Omega_{\text{in}} \\
\bm{\varphi} (\bm{x}) &= \bm{x},  \qquad   &\bm{x} \in \partial \Omega .
\end{array} \right.
\end{equation}
\item[Case 2.] assume that $\lambda=1$, $g_{_0}(\bm{x})=1$ in (\ref{eq_diffeo}) for all $\bm{x}\in\Omega$, and $g\big(\bm{\varphi}(\bm{x})\big)>0$, and $ \int_\Omega  g(\bm{\varphi}(\bm{x}))d\bm{x} = |\Omega |$. Thus there exists at least one transformation ${\bm{\varphi}(\bm{x})}$: $\Omega\rightarrow\Omega$ (one-to-one and onto mapping) such that
\begin{equation*}%\label{eq:eqjb11b}
\left\{ {\begin{array}{rll}
{g\big(\bm{\varphi}(\bm{x})\big)\det \big(\nabla \bm{\varphi} (\bm{x})\big)}& = 1,\qquad &\bm{x} \in \Omega_{\text{in}} \\
{\bm{\varphi} (\bm{x}) }&= \bm{x}, \qquad    &\bm{x} \in \partial \Omega.
\end{array}} \right.
\end{equation*}

\item[Case 3.] assume further that $f\big(\bm{\varphi}(\bm{x})\big):=1/g\big(\bm{\varphi}(\bm{x})\big)$, the above system can be rewritten as 
\begin{equation}\label{eq:eqjb11c}
\left\{ {\begin{array}{rll}
{\det \big(\nabla \bm{\varphi} (\bm{x})\big)}& = f\big(\bm{\varphi}(\bm{x})\big),\qquad &\bm{x} \in \Omega_{\text{in}} \\
{\bm{\varphi} (\bm{x}) }&= \bm{x}, \qquad    &\bm{x} \in \partial \Omega.
\end{array}} \right.
\end{equation}
\end{description}
\end{remark}

 Since such systems are not easy to solve directly, construction of diffeomorphic mapping ${\bm{\varphi}(\bm{x})}: \Omega\rightarrow\Omega$ of the Jacobian equation \eqref{eq:eqb11} is a very meaningful work for many image problems. Dacorogna and Moser designed a constructing method to indirectly obtain a diffeomorphic mapping ${\bm{\varphi}(\bm{x})}$ of \eqref{eq:eqb11} (\emph{see} \cite{dacorogna1990partial} for more details).

To proceed, there is a need to define a suitable nonempty diffeomorphic transformation set $\textbf{Diff}_{\bm{\varphi}}^{f(\bm{\varphi})}(\Omega)$ by the solution of the differential system (\ref{eq:eqjb11c}) as
\begin{equation}\label{eq:eqjb11d}
\textbf{Diff}_{\bm{\varphi}}^{f(\bm{\varphi})}(\Omega):=\left\{\bm{\varphi}(\bm{x}) \in \mathrm{C}^{\ell+1,\; \epsilon}(\Omega;\;\mathbb{R}^2)\left|\left\{ {\begin{array}{rll}
{\det \big(\nabla \bm{\varphi} (\bm{x})\big)}& = f\big(\bm{\varphi}(\bm{x})\big),\; &\bm{x} \in \Omega_{\text{in}} \\
{\bm{\varphi} (\bm{x}) }&= \bm{x}, \;  &\bm{x} \in \partial \Omega.
\end{array}} \right.\right.\right\}.
\end{equation}

Without loss of generality, we restrict our attention to $\ell=1$ in the following part.

\subsection{Construction}\label{sect3_construct} Before seeking a diffeomorphic transformation $\bm{\varphi} (\bm{x})$ in (\ref{eq:eqjb11c}), we first analyze a constructing method of the \emph{time-dependent} diffeomorphic solution $\bm{\phi} (\bm{x},t)$ (corresponding to
the coordinate system at time $t$) that satisfies the \emph{time-sequence} Jacobian equation defined by
\begin{equation}\label{eq:zj2c}
\left\{ {\begin{array}{rll}
g\big(\bm{\phi} (\bm{x},t)\big)\det \nabla \bm{\phi} (\bm{x},t)&= g\big(\bm{\phi} (\bm{x},0)\big)\det \nabla \bm{\phi} (\bm{x},0)=g_0(\bm{x}), &\bm{x} \in \Omega_{\text{in}} \\
 \bm{\phi} (\bm{x},t) &= \bm{x},  &\bm{x} \in \partial \Omega\\
 \bm{\phi} (\bm{x},0)&=\bm{x},&\bm{x} \in \Omega
\end{array}} \right.
\end{equation}
for any $t\in [0,1]$, where $\cfrac{\int_{\Omega} g(\bm{\phi} (\bm{x},t)) d\bm{x}}{\int_{\Omega}g_0(\bm{x}) d\bm{x}}=1$. Especially if taking $g\big(\bm{\phi} (\bm{x},t)\big)=1/f\big(\bm{\phi} (\bm{x},t)\big)>0$, one easily has
\begin{equation}\label{eq:zj2d}
\left\{ {\begin{array}{rll}
\det \nabla \bm{\phi} (\bm{x},t)&= \frac{f(\bm{\phi} (\bm{x},t))}{f(\bm{x})}, &\bm{x} \in \Omega_{\text{in}} \\
 \bm{\phi} (\bm{x},t) &= \bm{x},  &\bm{x} \in \partial \Omega\\
  \bm{\phi} (\bm{x},0)&=\bm{x},&\bm{x} \in \Omega.
\end{array}} \right.
\end{equation}

The following technical Lemma is useful for understanding the construction method solving \emph{time-dependent} diffeomorphism $\bm{\phi} (\bm{x},t)$ in (\ref{eq:zj2d}):
\begin{lemma}[Abel-Jacobi-Liouville Identity \cite{Liu66}] \label{lemma1}
Let $\varPhi(t)$ be a $n \times n $ matrix function where each element is differentiable on $t\in [0,1]$. If $\frac{d\varPhi(t)}{dt} = A(t)\varPhi(t)$ where $A(t)$ is a $n\times n $ matrix, then one has $\frac{d}{dt}\det \varPhi(t) = \textbf{Tr}A(t) \det\varPhi(t)$, where $\textbf{Tr}A(t)$ is the trace of $A(t)$.
\end{lemma}

\subsubsection{Dynamical system} Inspired by the knowledge in \cite{dacorogna1990partial} and the work in \cite{generation1992},
the previous literature (\emph{see} Theorem \ref{thm01}) shows that there exists a diffeomorphic transformation $\bm{\phi} (\bm{x},t)\in \mathcal{C}^{\ell+1,\; \epsilon}(\Omega, [0,1])$ of system (\ref{eq:zj2d}) with $f(\bm{x})=1$, for which we can find $\bm{\varphi} (\bm{x})=\bm{\phi} (\bm{x},1)$ satisfying (\ref{eq:eqjb11c}). To construct such diffeomorphic transformation $\bm{\phi} (\bm{x},t)$, we consider the following continuous evolution system
\begin{eqnarray}\label{eq:eqzj14a}
 \begin{array}{rl}
    \cfrac{d}{dt}\bm{\phi} (\bm{x},t)&=\bm{v}\big(\bm{\phi} (\bm{x},t), t\big).  \\
  \end{array}
  \end{eqnarray}
The state $\bm{\phi} (\bm{x},t)$ of the above dynamical system evolves through continuous time $t\in[0,1]$, such \emph{time sequence} 
$\bm{\phi} (\bm{x},t)$ can be seen as a coordinate system that moves through the state space. The evolution rule will specify how this coordinate system $\bm{\phi} (\bm{x},t)$ moves by designing its velocity field $\bm{v}\big(\bm{\phi} (\bm{x},t), t\big)$. In this case, starting with an initial state (coordinate system) 
\(\bm{\phi} (\bm{x},0)=\bm{x}\)
 at time $t=0$, the trajectory $\bm{\phi} (\bm{x},t)$ of all future times will be a new state through state space.

We also notice that the velocity field $\bm{v}\big(\bm{\phi} (\bm{x},t), t\big)$ from interactions between different variants (components) of $\bm{\phi} (\bm{x},t)$ cannot be
explained in terms of the individual properties of these variants alone. Although $\bm{v}\big(\bm{\phi} (\bm{x},t), t\big)$ cannot be measured in an experiment directly, it can in principle be obtained from the prior knowledge of control increment $\bm{u}\big(\bm{\phi} (\bm{x},t)\big)$ of the state $\bm{\phi} (\bm{x},t)$ in the process of creating a dynamical system.

To illustrate the idea of dynamical system solving (\ref{eq:eqjb11c}) or (\ref{eq:zj2d}), the velocity field can be analytically designed as
\begin{equation}\label{eq:eqzj14b}
\bm{v}\big(\bm{\phi} (\bm{x},t), t\big):= \cfrac{\bm{u}(\bm{\phi}(\bm{x},t))}{h(\bm{\phi} (\bm{x},t),t)},\end{equation}
where a homotopy composite function $h(\bm{\phi} (\bm{x},t),t)$ will be constructed in Section \ref{Sect_choice}. 

\subsubsection{Control increment system} One of main ingredients for constructing the diffeomorphic solution $\bm{\varphi} (\bm{x})=\bm{\phi} (\bm{x},1)$ of system (\ref{eq:eqjb11c}) is to determine control increment $\bm{u}(\bm{\phi}(\bm{x},t))$ by manually choosing the composite function $h(\bm{\phi} (\bm{x},t),t)$. 
Now let us look at how we may apply (\ref{eq:eqzj14a}) and (\ref{eq:eqzj14b}) to exploit $\bm{u}(\bm{\phi}(\bm{x},t))$. To express this discussion in mathematical terms, we define
\begin{equation}\label{eq:homotopy_function}
H(\bm{\phi} (\bm{x},t),t) =\det \nabla \bm{\phi} (\bm{x},t)\;h(\bm{\phi} (\bm{x},t),t).
\end{equation}

\begin{lemma}\label{lem_control_increment}
Let $h(\bm{\phi} (\bm{x},t),t)>0$ be a homotopy composite function with $h(\bm{\phi} (\bm{x},0),0)=h(\bm{x},0)$. If the control increment $\bm{u}(\bm{\phi} (\bm{x},t))$ in (\ref{eq:eqzj14b}) is a solution of the following PDE system
\begin{equation}\label{eq:zj2c_div}
\left\{ {\begin{array}{rll}
\diverl\Big(\bm{u}(\bm{\phi} (\bm{x},t))\Big)+\cfrac{\partial h(\bm{\phi} (\bm{x},t),t)}{\partial t}&=0,\qquad &\bm{x} \in \Omega_{\text{in}}, \\
 \bm{u}(\bm{\phi} (\bm{x},t)) &= \bm{0},  \qquad   &\bm{x} \in \partial \Omega,
\end{array}} \right.
\end{equation}
then the solution $\bm{\phi} (\bm{x},t)$ of the equation (\ref{eq:eqzj14a}) with $\bm{v}\big(\bm{\phi} (\bm{x},t), t\big):= \cfrac{\bm{u}(\bm{\phi}(\bm{x},t))}{h(\bm{\phi} (\bm{x},t),t)}$ satisfies
\begin{equation}\label{eq:zj2b}
\det \nabla \bm{\phi} (\bm{x},t)=\frac{h(\bm{x},0)}{h(\bm{\phi} (\bm{x},t),t)},\; \text{ for all } t\in (0,1]. 
%f\big(\bm{\phi} (\bm{x},t)\big),\; \text{ for all } t\in [0,1].
\end{equation}
\end{lemma}

\begin{proof} Let us consider the gradient form of the equation (\ref{eq:eqzj14a}) as
\begin{equation}\label{gradient1a}
\nabla \left(\frac{d}{{dt}}\bm{\phi} (\bm{x},t)\right) =\frac{d}{{dt}}\nabla \bm{\phi} (\bm{x},t) = \nabla \bm{v}\big(\bm{\phi} (\bm{x},t),t\big)\nabla \bm{\phi} (\bm{x},t),
\end{equation}
 which can be written by using Lemma \ref{lemma1} as
\begin{equation*}
\begin{split}
\frac{d}{{dt}}\det \nabla \bm{\phi} (\bm{x},t) = \textbf{Tr}\Big(\nabla \bm{v}\big(\bm{\phi} (\bm{x},t),t\big)\Big)\det \nabla \bm{\phi} (\bm{x},t)
=\diver \Big(\bm{v}\big(\bm{\phi} (\bm{x},t),t\big)\Big)\det \nabla \bm{\phi} (\bm{x},t).
\end{split}
\end{equation*}
Further from (\ref{eq:eqzj14b}), we easily obtain
\begin{equation}\label{eq:eqzj14c}
 \bm{u}\big(\bm{\phi} (\bm{x},t)) = \bm{v}\big(\bm{\phi} (\bm{x},t),t\big)h(\bm{\phi} (\bm{x},t),t)
\end{equation}
and
\begin{equation*}
\diver\Big(\bm{v}\big(\bm{\phi} (\bm{x},t),t\big)\Big) h(\bm{\phi} (\bm{x},t),t) = \diver\Big(\bm{u}(\bm{\phi} (\bm{x},t))\Big)
- \bm{v}\big(\bm{\phi} (\bm{x},t),t\big)\cdot\nabla h(\bm{\phi} (\bm{x},t),t).
\end{equation*}
Therefore, by rearranging this expression, we obtain the following derivative formula for estimating the properties of the function $H(\bm{\phi} (\bm{x},t),t)$:
\begin{equation*}
\begin{split}
\frac{d }{d t}H(\bm{\phi} (\bm{x},t),t)=&h(\bm{\phi} (\bm{x},t),t)\frac{d }{{d t}}\det \nabla \bm{\phi} (\bm{x},t)\\
&\qquad\qquad+ \det \nabla \bm{\phi} (\bm{x},t)\left(\cfrac{\partial h(\bm{\phi} (\bm{x},t),t)}{\partial t}+\nabla h(\bm{\phi} (\bm{x},t),t)\cdot\frac{d}{{dt}}\bm{\phi} (\bm{x},t)\right)\\
=& \det \nabla \bm{\phi} (\bm{x},t)\Bigg(\diver\Big(\bm{v}\big(\bm{\phi} (\bm{x},t),t\big)\Big)h(\bm{\phi} (\bm{x},t),t)\\
&\qquad\qquad+\cfrac{\partial h(\bm{\phi} (\bm{x},t),t)}{\partial t}+\nabla h(\bm{\phi} (\bm{x},t),t)\cdot\frac{d}{{dt}}\bm{\phi} (\bm{x},t)\Bigg)\\
=& \det \nabla \bm{\phi} (\bm{x},t)\left(\diver\Big(\bm{u}(\bm{\phi} (\bm{x},t))\Big)+\cfrac{\partial h(\bm{\phi} (\bm{x},t),t)}{\partial t}\right).
\end{split}
\end{equation*}
If letting 
\begin{equation*}
\diver\Big(\bm{u}(\bm{\phi} (\bm{x},t))\Big)+\cfrac{\partial h(\bm{\phi} (\bm{x},t),t)}{\partial t}=0, 
\end{equation*}
it is easy to check that $\frac{d }{{d t}}H(\bm{\phi} (\bm{x},t),t) = 0$ for all $t\in (0,1]$, hence we obtain
\begin{equation*}
\begin{split}
 h(\bm{x},0)=\det \nabla \bm{\phi} (\bm{x},0) h(\bm{x},0) = H(\bm{x},0)=H(\bm{\phi} (\bm{x},t),t) = \det \nabla \bm{\phi} (\bm{x},t)h(\bm{\phi} (\bm{x},t),t),
 \end{split}
\end{equation*}
i.e.,
\begin{equation*}%\label{eq:zj2b}
\det \nabla \bm{\phi} (\bm{x},t)=\frac{h(\bm{x},0)}{h(\bm{\phi} (\bm{x},t),t)},\; \text{ for all } t\in (0,1]. 
\end{equation*}

According to the above construction, if we consider the boundary condition $\bm{u}(\bm{\phi} (\bm{x},t))= \bm{0}$ for all $\bm{x} \in \partial \Omega$, thus the control increment $\bm{u}(\bm{\phi} (\bm{x},t))$ of the state $\bm{\phi} (\bm{x},t)$ in (\ref{eq:eqzj14b}) is defined as a non-uniqueness solution of the following PDE system:
\begin{equation*}%\label{eq:zj2c}
\left\{ {\begin{array}{rll}
\diver\Big(\bm{u}(\bm{\phi} (\bm{x},t))\Big)+\cfrac{\partial h(\bm{\phi} (\bm{x},t),t)}{\partial t}&=0,\qquad &\bm{x} \in \Omega_{\text{in}}, \\
 \bm{u}(\bm{\phi} (\bm{x},t)) &= \bm{0},  \qquad   &\bm{x} \in \partial \Omega,
\end{array}} \right.
\end{equation*}
which proves the assertion.
\end{proof}

We remark that similar construction of the control increment $\bm{u}(\bm{\phi}(\bm{x},t))$ can be found in \cite{dacorogna1990partial,generation1992}.

\subsubsection{Boundary condition} Next we analyze that the solution $\bm{\phi}(\bm{x},1)$ in (\ref{eq:eqzj14a}) satisfies the boundary condition of (\ref{eq:eqjb11c}), i.e., $\bm{\phi}(\bm{x},1)=\bm{x}$ when ${\bm{x} \in \partial \Omega} $. Since $\bm{u}(\bm{\phi} (\bm{x},t))=\bm{0}$ for all ${\bm{x} \in \partial \Omega} $ and $t\in [0,1]$, (\ref{eq:eqzj14a}) can be simplified as:
  \begin{equation*}
  {\frac{d}{{dt}}\bm{\phi} (\bm{x},t)}={\bm{v}\big(\bm{\phi}(\bm{x},t),t\big)}=\bm{0},\quad \bm{x} \in \partial \Omega. 
  \end{equation*}
Hence $\bm{\phi}(\bm{x},t)$ is independent of $t$ on the boundary~$\partial \Omega$, it means that $\bm{\phi}(\bm{x},1)=\bm{\phi}(\bm{x},t)=\bm{\phi}(\bm{x},0)=\bm{x}$ for any ${\bm{x} \in \partial \Omega} $. Furthermore, combining with (\ref{eq:zj2b}), one has 
\begin{equation*}%\label{eq:zj2c}
\left\{ {\begin{array}{rll}
\det \nabla \bm{\phi} (\bm{x},t)&= \frac{h(\bm{x},0)}{h(\bm{\phi} (\bm{x},t),t)},\qquad &\bm{x} \in \Omega_{\text{in}}\\%=\frac{f(\bm{\phi} (\bm{x},t))}{f(\bm{x})},\qquad &\bm{x} \in \Omega \\
 \bm{\phi} (\bm{x},t) &= \bm{x},  \qquad   &\bm{x} \in \partial \Omega,
\end{array}} \right.
\end{equation*}
for any $t\in (0,1]$. 

\begin{remark}
Especially if one takes $t=1$, then the solution $\bm{\varphi}(\bm{x})=\bm{\phi}(\bm{x},1)$ obtained by the above construction method satisfies the differential system
\begin{equation}\label{eq:zj2a}
\left\{ {\begin{array}{rll}
\det \big(\nabla \bm{\varphi}(\bm{x})\big)& =\frac{h(\bm{x},0)}{h(\bm{\varphi} (\bm{x}),1)}  %f(\bm{\varphi}(\bm{x}))
,\qquad &\bm{x}\in \Omega_{\text{in}}, \\
 \bm{\varphi} (\bm{x}) &=\bm{x},  \qquad   &\bm{x} \in \partial \Omega .
\end{array}} \right.
\end{equation}
\end{remark}

\subsubsection{Choice of path following composite function}\label{Sect_choice} Path following method, also called imbedding and homotopy continuation method, was first applied to the solution of separation models involving large numbers of nonlinear equations \cite{SALGOVIC1981}, which provides an useful approach to find the zeros of nonlinear function $f: \mathbb{R}^n\rightarrow \mathbb{R}^n$ in a globally convergent way. 

As has been discussed, $g_{_0}(\bm{x})$ and $g(\bm{\varphi}(\bm{x}))$ are two mass measures on two manifolds $\mathcal{X}$ and $\mathcal{Y}$, especially in $\bm{x}$- and $\bm{\varphi}(\bm{x})$-coordinate systems.
The purpose of considering the path following in this work is to build a smooth path between the beginning solution $\bm{\phi}(\bm{x},0)=\bm{x}$ of \eqref{eq:zj2d} at time $t=0$ and the ending target solution $\bm{\phi}(\bm{x},1)=\bm{\varphi}(\bm{x})$ of \eqref{eq:zj2d} at time $t=1$ (\emph{see} blue curve in Fig.\ref{fig_pathfollowing}). This means that it gradually deforms this initial transformation $\bm{x}$ (coordinate system at time $t=0$) into one helpful deformation field $\bm{\phi}(\bm{x},t)$ of (\ref{eq:zj2d}) as time $t\in [0,1]$, and eventually ends with a perfect registration $T(\bm{\varphi}(\bm{x}))\approx R(\bm{x})$ between the \emph{template} $T(\bm{x})$ and the \emph{reference} $R(\bm{x})$. 

To obtain system \eqref{eq:zj2b}, the design of $H(\bm{\phi} (\bm{x},t),t)=\det \nabla \bm{\phi} (\bm{x},t)\;h(\bm{\phi} (\bm{x},t),t)$ is one of the most important steps, where the function $H(\bm{\phi} (\bm{x},t),t)$ satisfies $\frac{d }{{d t}}H(\bm{\phi} (\bm{x},t),t) = 0$ for all $t\in (0,1]$, i.e.,
\[g_{_0}(\bm{x}):=h(\bm{x},0)= H(\bm{x},0)=H(\bm{\phi} (\bm{x},t),t) = \det \nabla \bm{\phi} (\bm{x},t)h(\bm{\phi} (\bm{x},t),t)\]
if \eqref{eq:zj2c_div} holds (\emph{see} Fig.\ref{fig_pathfollowing}(a)).

 As has been seen from (\ref{eq:eqzj14a}) and (\ref{eq:zj2c_div}), the function $h(\bm{\phi}(\bm{x},t),t)$ selectively determines the control increment $ \bm{u}(\bm{\phi}(\bm{x},t))$ and the velocity field $\bm{v}\big(\bm{\phi} (\bm{x},t), t\big)$ of the evolving deformation $\bm{\phi} (\bm{x},t)$. Next we follow the
discussion in \cite{EAllgower1993Continuation} to construct a composite function 
$h(\bm{\phi}(\bm{x},t),t)$ that allows to use different homotopy variants to serve the objective of continuously deforming an origin mass function $g_{_0}(\bm{x})$ at $t=0$ into a target mass function $g(\bm{\varphi}(\bm{x}))$ at $t=1$ (\emph{see} Fig. \ref{fig_pathfollowing}(b)). Here the set of admissible path following composite functions is defined by
\begin{equation}
\begin{split}
\mathcal{\bm{H}}_{g,\bm{\phi}}(\Omega,[0,1]):=\Big\{h(\bm{\phi}(\bm{x},t),t)\in\mathcal{C}^\ell(\Omega, &[0,1])\Big|\;\bm{\phi}(\bm{x},t)\in\mathcal{C}^{\ell+1,\; \epsilon}(\Omega, [0,1]),\\ 
h(\bm{\phi}(\bm{x},t),t)>0\;\text{ and }&\; \frac{\int_\Omega\;h(\bm{\phi}(\bm{x},t),t)\;d\bm{x}}{\int_\Omega\;g_{_0}(\bm{x})\;d\bm{x}}=1,\; \forall\; t\in[0,1],\\h(\bm{\phi}(\bm{x},0),0)=&g_{_0}(\bm{x})\; \text{ and }\; h(\bm{\phi}(\bm{x},1),1)=g(\bm{\varphi}(\bm{x})) \text{ on } \Omega\Big\},
\end{split}
\end{equation}
then we list several composite functions that allows to be used in Section \ref{sect3_construct} as follows
\begin{description}
  \item[P1:] $h(\bm{\phi}(\bm{x},t),t)=(1-t)g_{_0}(\bm{x}) + t\;g(\bm{\phi}(\bm{x},t)),\; \text{ for all } t\in [0,1]$;
  \item[P2:] $h(\bm{\phi}(\bm{x},t),t)=g_{_0}(\bm{x})e^{t\ln \frac{g(\bm{\phi}(\bm{x},t))}{g_{_0}(\bm{x})}},\; \text{ for all } t\in [0,1]$;
  \item[P3:] $h(\bm{\phi}(\bm{x},t),t)=(1-t)\det \big(\nabla \bm{\phi}(\bm{x},t)\big)+ t\;g(\bm{\phi}(\bm{x},t)),\; \text{ for all } t\in [0,1]$;
    \item[P4:] $h(\bm{\phi}(\bm{x},t),t)=g(\bm{\phi}(\bm{x},t)),\; \text{ for all } t\in [0,1]$.
\end{description}

\begin{figure}
\begin{center}
    \includegraphics[width=0.95\textwidth]{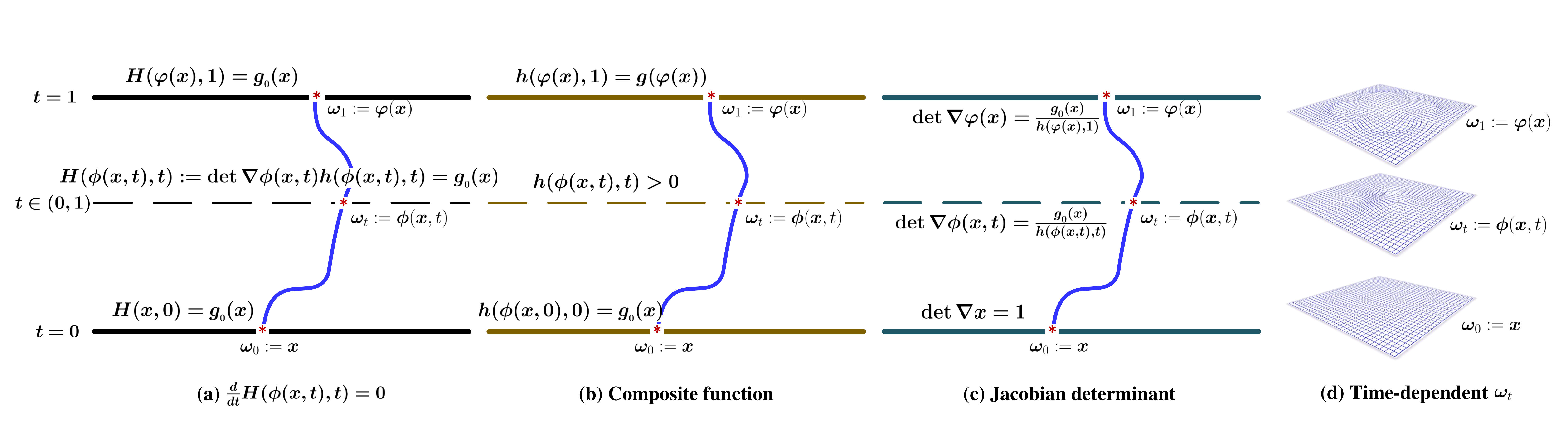}     
\end{center}
\caption{Time-sequence transformation $\bm{\omega}_t$, composition function and time-dependent Jacobian determinant.
}\label{fig_pathfollowing}
\end{figure}

For discrete setting, the mapping $\bm{\phi}(\bm{x},t)$ consists simply in moving the positions $\bm{x}_i=(x_1^i,x_2^i)$ of all the grid points to new positions $\bm{\phi}(\bm{x}_i,t)$, for instance for those with a grid quality measure $g(\bm{\phi}(\bm{x},t))$, the notion of $\bm{\phi}(\bm{x},t)$ plays a fundamental to describe spatial modifications for image registration.
The mass measure $g(\bm{\phi}(\bm{x},t))$ guarantees that $\bm{\phi}(\bm{x},t)$ is diffeomorphic, unfortunately it is unknown and depends on the moving image $T(\bm{\phi}(\bm{x},t))$ and the \emph{reference} $R(\bm{x})$, also depends on $\bm{\phi}(\bm{x},t)$ that be partly driven by SSD for image registration. Hence we consider the dirac delta approximation defined by
\begin{equation}\label{eq_g_func_approx}
\begin{split}
g(\bm{\phi}(\bm{x},t))=\int_{\mathbb{R}^2}g(\bm{y})\delta(\bm{\phi}(\bm{x},t)-\bm{y})d\bm{y}%=\int_{\mathbb{R}^2}g(\bm{y})\delta(\bm{\phi}(\bm{x},t)-\bm{y})d\bm{y}
=&\lim\limits_{\sigma\rightarrow 0}\frac{1}{2\pi\sigma^2}\int_{\mathbb{R}^2}g(\bm{y})\exp\left(-\frac{\|\bm{y}-\bm{\phi}(\bm{x},t)\|^2}{2\sigma^2}\right)d\bm{y}\\
\approx&\frac{1}{2\pi\sigma_\epsilon^2}\int_\Omega g(\bm{y})\exp\left(-\frac{\|\bm{y}-\bm{\phi}(\bm{x},t)\|^2}{2\sigma_\epsilon^2}\right)d\bm{y},
\end{split}
\end{equation}
where $g(\bm{y}):=g_{_0}(\bm{x})=1$ is an uniform mass density of the point $\bm{y}$ in origin $\bm{y}$-coordinate system for our registration problem at time $t=0$.

 In our work, an appropriate composite function $h(\bm{\phi}(\bm{x},t),t)$ will be chose to exploit the control increment $ \bm{u}(\bm{\phi}(\bm{x},t))$ of the diffusion equation (\ref{eq:zj2c_div}), and such $ \bm{u}(\bm{\phi}(\bm{x},t))$ is also required to minimize registration energy. Hence $\bm{\phi} (\bm{x},t)$ satisfies the diffeomorphic system (\ref{eq:zj2b}) % (\ref{eq:zj2c}) 
for any time $t\in [0,1]$.

\subsection{Implementation} We can deduce the implementation, following reasoning similar to that in \cite{XChen2016Numerical}. Then the overall constructing process based on a dynamical system may be summarized as follows.
\begin{description}
  \item[Step 1.] Designing manually a homotopy composite function $h(\bm{\phi} (\bm{x},t),t)\in \mathcal{H}_{g,\bm{\phi}}(\Omega,[0,1])$ such that
\[h(\bm{\phi} (\bm{x},0),0)=g_{_0}(\bm{x}) \text{ and } h(\bm{\phi} (\bm{x},1),1)=g(\bm{\phi}(\bm{x},1))=g(\bm{\varphi}(\bm{x})) \text{ for all } \bm{x}\in\Omega.\] 

  \item[Step 2.] Solving $ \bm{u}(\bm{\phi}(\bm{x},t)) $: $\Omega\rightarrow \mathbb{R}^2$, such that
  \begin{equation}\label{eq:eqzj12}
  \left\{ {\begin{array}{rll}
    \diver \bm{u}\big(\bm{\phi} (\bm{x},t)\big)&=-\cfrac{\partial h(\bm{\phi} (\bm{x},t),t)}{\partial t}, \quad &\bm{x} \in \Omega_{\text{in}} \\
    {\bm{u}(\bm{\phi} (\bm{x},t))} &=\bm{0}, \quad  &\bm{x} \in \partial \Omega.
    \end{array}} \right.
  \end{equation}
  where $\bm{\phi} (\bm{x},t)\in\Omega$ for all $t\in [0,1]$.
  
  \item[Step 3.] Solving $\bm{\phi} (\bm{x},t)$ from the dynamical system
  \begin{eqnarray}\label{eq:eqzj14}
  \left\{ \begin{array}{rll}
  \cfrac{d}{dt}\bm{\phi} (\bm{x},t)&=\bm{v}\big(\bm{\phi} (\bm{x},t), t\big):= \cfrac{\bm{u}(\bm{\phi}(\bm{x},t))}{h(\bm{\phi} (\bm{x},t),t)},\quad &\bm{x} \in \Omega \\
  \bm{\phi} (\bm{x},0) &=\bm{x},\quad &\bm{x} \in \Omega.  \\
  \end{array} \right.
  \end{eqnarray}
\end{description}
A solution $\bm{\phi}(\bm{x},1)$ of (\ref{eq:eqzj14}) satisfies the differential system (\ref{eq:eqjb11c}), therefore we can solve Step 2 and Step 3 alternately to obtain the solution of the differential system (\ref{eq:eqjb11c}).

The dynamical system (\ref{eq:eqzj14}) is an ODE system that can be integrated
numerically using a high-order difference scheme. In image registration, the transformation $\bm{\phi} (\bm{x},t_{k+1})$ at time step $t_{k+1}$ is updated from the current transformation $\bm{\phi} (\bm{x},t_{k})$ by
solving equation (\ref{eq:eqzj14}) with the \textbf{RK4} difference scheme \eqref{eq_RK4}, namely
\begin{equation*}
\begin{array}{rll}
\bm{\phi} (\bm{x},t_{k+1})=\textbf{RK4}(\bm{\phi} (\bm{x},t_{k}),\bm{v},\Delta t).
\end{array}
\end{equation*}
We also refer readers to literatures \cite{Landmark,MangA,Modersitzki2004Numerical,Fair} for more details. Therefore, the solution $\bm{\varphi} (\bm{x})=\bm{\phi} (\bm{x},1)$ obtained by the above construction method satisfies the differential system
\begin{equation*}\label{eq:eqzj15}
\left\{ {\begin{array}{rll}
\det \big(\nabla \bm{\varphi} (\bm{x})\big)& = \frac{h(\bm{x},0)}{h(\bm{\varphi} (\bm{x}),1)},\qquad &\bm{x} \in \Omega_{\text{in}} \\
 \bm{\varphi} (\bm{x}) &= \bm{x},  \qquad   &\bm{x} \in \partial \Omega .
\end{array}} \right.
\end{equation*}

\subsection{Control increment in diffeomorphic image registration}
Although there are many methods for solving the above control increment PDE system (\ref{eq:eqzj12}) with non-unique solutions, it is impossible to find useful formulas for the close-form solutions of most. To look for numerical approximations $\bm{u}(\bm{\phi} (\bm{x},t))$ for diffeomorphic image registration, it is useful to add an optimization that guarantees suitable increment $\bm{u}(\bm{\phi} (\bm{x},t))$ of the diffeomorphic solution $\bm{\phi} (\bm{x},t)$. Therefore, the \emph{time sequential} minimization problem
\begin{equation}\label{eq_min_sub}
\min\limits_{\bm{u}(\bm{\phi} (\cdot,t))}\left\{\mathcal{E}^t_0(\bm{u}):=\mathcal{D}[R,T;{\bm{\bm{\phi}} (\cdot,t)+\alpha(t)\bm{u}\big(\bm{\phi} (\cdot,t)\big)}]+ \tau \mathcal{S}[{\bm{u}\big(\bm{\phi} (\cdot,t)\big)}]\right\}
\end{equation}
is imposed as a new restriction for $\bm{u}(\bm{\phi} (\bm{x},t))$, together with the control increment system (\ref{eq:eqzj12}) (A feasible set), to ensure the existence of an unique increment $\bm{u}(\bm{\phi} (\bm{x},t))$. Here $\bm{\phi} (\bm{x},t)$ is the state at time $t$, especially if $t$ is fixed, the state $\bm{\phi} (\bm{x},t)$ can be explicitly updated. Moreover, $\bm{u}$ not only satisfies the PDE system \eqref{eq:eqzj12} but also ensures that the deformation field $\bm{\phi} (\bm{x},t)$ has good smoothness and diffeomorphism. Here $\mathcal{S}[{\bm{u}\big(\bm{\phi} (\cdot,t)\big)}] $ is the regularizer, $\tau>0$ is the regularization parameter. Hence $ T(\bm{\phi} (\bm{x},t))$ is spatially aligned with $ R(x)$.

\begin{remark} As we can observe, a well-posed $\bm{u}(\bm{\phi} (\bm{x},t))$ in image registration has been obtained in case of minimizing local $\mathcal{E}^t_0\left(\bm{u}\right)$ (at time $t$). Although the first-order update of $\bm{\phi} (\bm{x},t)$ is employed in energy $\mathcal{E}^t_0\left(\bm{u}\right)$, it essentially admits a solution $\bm{u}(\bm{\phi} (\bm{x},t))$ satisfying the constraint (\ref{eq:eqzj12}). It is well known that the \textbf{RK4} discretization of ODE \eqref{eq:eqzj14} is a stable higher-order accurate difference scheme, which has been widely used in image registration for obtaining a diffeomorphic solution \cite{MangA,Fair}. 

\end{remark}
  
%-------------------------
\section{The proposed optimal control relaxation model for diffeomorphic registration}\label{sec:Section4} In this section, we aim to present an optimal control system modeling image registration, which can help to get a diffeomorphic deformation.

%--------------------------
The diffeomorphic image registration is, in the space of diffeomorphism deformation field $\textbf{Diff}_{\bm{\varphi}}^{f(\bm{\varphi})}(\Omega)$ $\big(f(\bm{\varphi})>0$ and $ \int_\Omega  1/{f(\bm{\varphi}(\bm{x}))}\;d\bm{x} = |\Omega | \big)$, to seek an optimal solution $\bm{\varphi}$ with some regularity $ \mathcal{S}[{\bm{\varphi}}]$ such that $\mathcal{D}[R,T;{\bm{\varphi}}]$ is as small as possible. Here we consider diffeomorphic image registration model as a constrained optimization problem
\begin{equation}\label{model1}
\begin{array}{rll}
\min\limits_{\bm{\varphi}\in\textbf{Diff}_{\bm{\varphi}}^{f(\bm{\varphi})}(\Omega)}\big\{\mathcal{E}({\bm{\varphi}}):=\mathcal{D}[R,T;{\bm{\varphi}}] + \tau \mathcal{S}[{\bm{\varphi}}]\big\},
\end{array}
\end{equation}
where $\textbf{Diff}_{\bm{\varphi}}^{f(\bm{\varphi})}(\Omega)$ is defined by (\ref{eq:eqjb11d}). Usually the regularization term $ \mathcal{S}[{\bm{\varphi}}]$ can be elastic modulus \cite{broit1981optimal}, diffusion modulus \cite{fischer2002fast} or curvature modulus \cite{fischer2003curvature}, etc.

\begin{definition} For a function $\bm{f}:\Omega\rightarrow \mathbb{R}^2$, assume that
\[
\big[\bm{f}\big]^{k,\epsilon}_{\Omega}=\sup\limits_{\bm{x}, \bm{y} \in \Omega, \bm{x} \neq \bm{y}} \frac{|\partial^k\bm{f}(\bm{x})-\partial^k\bm{f}(\bm{y})|}{|\bm{x}-\bm{y}|^{\epsilon}}<+\infty,\;\;  k=0,1,\dots,\ell+1.
\]
A bounded H{\"o}lder subset $\textbf{\emph{HB}}^{\ell+1,\epsilon}$ is defined as 
\begin{equation}\label{eq_space}
\begin{split}
\textbf{\emph{HB}}^{\ell+1,\epsilon}:=\Big\{\bm{f}\in\mathcal{C}_0^{\ell+1,\; \epsilon}(\Omega, [0,1]): \exists \mathcal{L}>0& \text{ such that }\\ \|\bm{f}\|_\infty\leq 1 \text{ and }& \big[\bm{f}\big]^{k,\epsilon}_{\Omega}\leq \mathcal{L}, \forall k=0,\dots,\ell+1 \Big\}.
\end{split}
\end{equation}
\end{definition}

In the previous Section, we have known that the differential system (\ref{eq:eqjb11c}) exists at least one diffeomorphic solution ${\bm{\varphi}}({\bm{x}})$, so the solution set $\textbf{Diff}_{\bm{\varphi}}^{f(\bm{\varphi})}(\Omega)$ is not empty. A method to solve $\bm{\varphi}\in \textbf{Diff}_{\bm{\varphi}}^{f(\bm{\varphi})}(\Omega)$ is constructed, and one of the important steps is to solve (\ref{eq:eqzj12}). We denote the constraint set
$\textbf{DC}_{h(\bm{\phi}_t)}(\Omega) $ as the solution set of differential system
\begin{eqnarray}\label{eq_rot}
 \left\{ {\begin{array}{rll}
   \diver \bm{u}\big(\bm{\phi} (\bm{x},t)\big)=&-\cfrac{\partial h(\bm{\phi} (\bm{x},t),t)}{\partial t}, &\bm{x} \in \Omega_{\text{in}}, \\
   ~~~\bm{u}\big(\bm{\phi} (\bm{x},t)\big) =&\bm{0},~ &\bm{x} \in \partial \Omega,  \\
 \bm{u}(\bm{\phi}(\bm{x},t))&\in\textbf{HB}^{\ell+1,\epsilon}\subset\mathcal{C}_0^{\ell+1,\; \epsilon}(\Omega, [0,1])&\text{ and }  \bm{\phi}(\bm{x},t)\in\mathcal{C}^{\ell+1,\; \epsilon}(\Omega, [0,1]).\\
 \end{array}}\right.
\end{eqnarray}
Further we employ a new coordinate system $\bm{\omega}_t$ to represent the transformed grid $\bm{\phi} (\bm{x},t)$ at time $t$, hence $\textbf{DC}_{h(\bm{\omega}_t)}(\Omega)$ replacing $\textbf{DC}_{h(\bm{\phi}_t)}(\Omega) $ denotes the solution set of differential system
\begin{eqnarray}\label{eq_new_rot}
 \left\{ {\begin{array}{rll}
   \diver \bm{u}(\bm{\omega}_t)=&-\cfrac{\partial h(\bm{\omega}_t,t)}{\partial t}, &\bm{\omega}_t \in \Omega_{\text{in}}\\
   ~~~\bm{u}(\bm{\omega}_t) =&\bm{0},~ &\bm{\omega}_t \in \partial \Omega.  \\
 \end{array}} \right.
\end{eqnarray}
Since $\bm{\phi} (\bm{x},t)$ is a diffeomorphic mapping with respect to $\bm{x}$, thus one has $\bm{x}=\bm{\phi}^{-1}(\bm{\omega}_t)$.
Suppose also that $h(\bm{\omega}_t,t)\in \mathcal{H}_{g,\bm{\phi}}(\Omega,[0,1])$ is the $\ell$-th order continuously differentiable on $\Omega$, from the above differential system, thus we can easily see that $\bm{u}\in\mathcal{C}^{\ell+1}(\Omega,\mathbb{R}^2)$. We furtherly restrict the uniform boundedness of the first derivatives of $\bm{u}$ in (\ref{eq_new_rot}), so that such $\bm{u}$ has the characteristic
of equicontinuity. In order to employ a regularization $\mathcal{S}[\bm{u}]$ of $\bm{u}$ in the proposed registration model, we modify the constraint set $\textbf{DC}_{h(\bm{\omega}_t)}(\Omega)$ as
\begin{equation}\label{eq_new_rot_b}
\begin{split}
\textbf{DCH}_{h_t}(\Omega):=\Big\{ \bm{u}&\in\textbf{DC}_{h(\bm{\omega}_t)}(\Omega)\cap H^1(\Omega)
\Big\}.
\end{split}
\end{equation}
with the norm
\[\|u\|_{\text{DCH}}=\|u\|_{\mathcal{C}^{\ell+1}}+\|u\|_{H^1(\Omega)}.\]

\subsection{Compactness of the space $\textbf{DCH}_{h_t}(\Omega)$}
We expect that this new constraint set $\textbf{DCH}_{h_t}(\Omega)$ will have similar differential geometry properties to a diffeomorphic set $\textbf{DC}_{h(\bm{\phi}_t)}(\Omega)$, but the new set will be better in analysis for image registration problem, depending on the
solution space for $\bm{u}$. 

Proving existence of solutions for a functional such as $\mathcal{E}(\bm{u})$ may proceed in different ways;
below we establish the result by lower semi-continuity of the functional and sequential compactness of
space.

\begin{lemma}[The Bolzano-Weierstrass Theorem \cite{carl2018bolzano}]\label{lemma_4_2}
Assume that 
$\Omega$ is a compact closed set. Then every
continuous function $\bm{u}\in\textbf{\emph{DC}}_{h_t}(\Omega)$ on a compact set $\Omega$ is bounded, i.e., $\textbf{\emph{DC}}_{h_t}(\Omega)$ is a bounded function space in $\mathcal{C}^{\ell+1}$.
\end{lemma}

\begin{definition}
A subset $\mathcal{F}$ of $\mathcal{C}(\Omega,\mathbb{R}^d)$ is equicontinuous if for every $\varepsilon>0$ there exists some $\delta$ such that
$\|\bm{\phi}(\bm{x})-\bm{\phi}(\bm{y})\|<\varepsilon$ for all $\bm{\phi} \in \mathcal{F}$, and $\|\bm{x}-\bm{y}\|<\delta,\; \bm{x}, \bm{y} \in \Omega$.
\end{definition}

\begin{lemma}\label{lemma_4_5}
For any $h(\bm{\omega}_t,t)\in\mathcal{H}_{g,\bm{\phi}}(\Omega,[0,1])$, any bounded sequence $\mathcal{F}=\{\bm{u}^i\}_{i=1}^{+\infty}\subset\textbf{\emph{DC}}_{h_t}(\Omega)$ is equicontinuous in $\mathcal{C}^{\ell+1}$.
\end{lemma}
\begin{proof}
For any $\bm{u}\in\mathcal{F}\subset\textbf{{DC}}_{h_t}(\Omega)\subset\textbf{{HB}}^{\ell+1,\epsilon}$ and $\mu\in(0,1]$, $\exists \mathcal{L}>0$, from \eqref{eq_space} one has 
\begin{equation*}
 \left\{ {\begin{array}{rll}
     &\diver \bm{u}(\bm{x})=-\cfrac{\partial h(\bm{x},t)}{\partial t}, &\bm{x}  \in \Omega_{\text{in}}\\
   &~~~\bm{u}(\bm{x}) =\bm{0},~ &\bm{x} \in \partial \Omega.  \\
  &\big[\bm{u}\big]^{k,\epsilon}_{\Omega}=\sup\limits_{\bm{x}, \bm{y} \in \Omega, \bm{x} \neq \bm{y}} \frac{|\partial^k\bm{u}(\bm{x})-\partial^k\bm{u}(\bm{y})|}{|\bm{x}-\bm{y}|^{\epsilon}}<\mathcal{L}, & k=0,1,\dots,\ell+1.
 \end{array}} \right.
\end{equation*}
Moreover, the bounded
\[\Lambda^k:=\sup\limits_{\bm{u}\in\mathcal{F}}\big[\bm{u}\big]^{k,\epsilon}_{\Omega}<\mathcal{L},\quad \text{for } k=0,1,\dots,\ell+1\]
is defined. For any $\bm{x}, \bm{y}\in\Omega$,  
we have
$$
\|\bm{u}(\bm{y})-\bm{u}(\bm{x})\|_{\mathcal{C}^{\ell+1}}\leq (\ell+1)\mathcal{L}|\bm{y}-\bm{x}|^\mu,
$$
if taking $\delta=\left(\frac{\varepsilon}{(\ell+1)\mathcal{L}}\right)^{1/\epsilon}$, thus we deduce
$$
\|\bm{u}(\bm{y})-\bm{u}(\bm{x})\|_{\mathcal{C}^{\ell+1}}\leq \varepsilon,
$$
which shows the assertion.
\end{proof}

Combining the boundedness of $\textbf{{DC}}_{h_t}(\Omega)$ from Lemma \ref{lemma_4_2} and the sequentially equicontinuity of $\textbf{{DC}}_{h_t}(\Omega)$ in $\mathcal{C}^{\ell+1}$ from Lemma \ref{lemma_4_5}, we deduce that $\textbf{{DC}}_{h_t}(\Omega)$ is sequentially compact from Arzela-Ascoli Theorem.
From the compactness of $H^1(\Omega)$ \cite{EZeidler1985}, it follows

\begin{lemma}\label{lemma_4_6} $\textbf{\emph{DCH}}_{h_t}(\Omega)$ is sequentially compact.
\end{lemma}

\subsection{Optimal control relaxation for the diffeomorphic registration (\ref{model1})} In this part, we first consider the nonlinear dynamical system of the form
\begin{equation}\label{eq:eq13_opti_control}
\left\{\begin{array}{rl}
\cfrac{d}{dt}\;\bm{\omega}_t&=\frac{\bm{u}(\bm{\omega}_t)}{h(\bm{\omega}_t,t)}\\
\bm{\omega}(\bm{x},0)&=\bm{\omega}_0=\bm{x},
\end{array}\right.
\end{equation}
where $\bm{\omega}_t$ is the state function at time $t$ and $\bm{u}(\bm{\omega}_t)$ is control increment function, and $\bm{\omega}_0=\bm{\omega}(\bm{x},0)=\bm{x}$, $\bm{\omega}_1=\bm{\omega}(\bm{x},1)=\bm{\varphi}(\bm{x})$ are the initial state and final state (solution) at $t=0$ and $t=1$, respectively. Moreover, $h(\bm{\omega}_t,t)$ is a real-valued continuously differentiable function with respect to $\bm{\omega}_t$ and time $t$. 

To indirectly solve the diffeomorphic deformation $\bm{\varphi}$ by using the construction, we will use $\textbf{DCH}_{h_t}(\Omega)$ to replace the constraints $\textbf{Diff}_{\bm{\varphi}}^{f(\bm{\varphi})}(\Omega)$ in model (\ref{model1}), and propose a \emph{time-sequence} registration model
to find the optimal control increment $\bm{u}(\bm{\omega}_t)$ as follow
\begin{equation}\label{eq:eq13}
\begin{array}{rl}
&\min\limits_{\bm{u}(\bm{\omega}_t)\in\textbf{DCH}_{h_{t}}(\Omega)} \Bigg\{\mathcal{E}\big(\bm{\omega}_t,\bm{u}\big):=\mathcal{D}^{\alpha(t)}_{T,R}[\bm{u}]+ \tau \mathcal{S}^t[\bm{u}]\Bigg\}\\
\\
&\qquad\textbf{ s.t.}\;\left\{\begin{array}{rl}\cfrac{d}{dt}\;\bm{\omega}_t&=\bm{v}(\bm{\omega}_t,t):=\frac{\bm{u}(\bm{\omega}_t)}{h(\bm{\omega}_t,t)},\\
\bm{\omega}(\bm{x},0)&=\bm{\omega}_0=\bm{x},
\end{array}\right.
\end{array}
\end{equation}
where $\mathcal{D}^{\alpha(t)}_{T,R}[\bm{u}]:= \mathcal{D}[R,T;{\bm{\omega}_t+\alpha(t)\bm{u}(\bm{\omega}_t)}]= \cfrac{1}{2} \mathlarger{\int_\Omega }\Big(T\big({\bm{\omega}_t+\alpha(t)\bm{u}(\bm{\omega}_t}) \big)- R(\bm{\phi}^{-1}(\bm{\omega}_t))\Big)^2 d\bm{\omega}_t$,
$\mathcal{S}^t[\bm{u}] := \cfrac{1}{2} \mathlarger{\int_\Omega} |\nabla {\bm{u}(\bm{\omega}_t)}|^2 d\bm{\omega}_t$, $\tau>0$ and $\alpha(t)=\frac{\Delta t}{h(\bm{\omega}_t,t)}>0$.
Using the temporal \textbf{RK4} discretization in $[0,1]$, an update scheme of the diffeomorphic mapping $ \bm{\omega}_{t_k} $ in (\ref{eq:eq13}) is given as
\begin{equation}\label{eq:eq15}
\left\{\begin{array}{rl}
\bm{\omega}_{t_k}=& \textbf{RK4}(\bm{\omega}_{t_{k-1}},\bm{v},\Delta t),\\
\bm{u}(\bm{\omega}_{t_k})=&\argmin\limits_{\bm{u}(\bm{\omega}_{t_k})\in\textbf{DCH}_{h_{t_k}}(\Omega)} 
\left\{\mathcal{E}^{t_k}\big(\bm{u}\big):=\mathcal{E}\big(\bm{\omega}_{t_k},\bm{u}\big)\right\},\\
\bm{\omega}(\bm{x},0)=&\bm{\omega}_0=\bm{x}.
\end{array}
\right.
\end{equation}

The optimization energy $\mathcal{E}^{t_k}({\bm{u}})$ in (\ref{eq:eq15}) is a non-convex functional, however the constraint ${\bm{u}(\bm{\omega}_{t_k})}\in\textbf{DCH}_{h_{t_k}}(\Omega)$ ensures that the deformation field $\bm{\omega}_{t_k}=\bm{\phi} (\bm{x},t_k)$ is a diffeomorphic solution of (\ref{eq:zj2c}), hence $\bm{\phi} (\bm{x},t_k)$ is reversible and its inverse transform $\bm{x}=\bm{\phi}^{-1}(\bm{\omega}_{t_k})$ is continuous.
%--------------------------
\subsection{ALMM for the optimal control system} In the numerical implementation, we employ the ALMM algorithm \cite{wu2010augmented} to solve \emph{time-sequence} constrained optimization problem
\begin{equation}\label{eq:eq15_subp}
\min\limits_{\bm{u}(\bm{\omega}_{{\iota}})\in\textbf{DCH}_{h_{{\iota}}}(\Omega)} \Big\{\mathcal{E}^{{\iota}}\big(\bm{u}\big):= \mathcal{D}^{\alpha(\iota)}_{T,R}[\bm{u}] + \tau \mathcal{S}[{\bm{u}}] \Big\}\qquad\text{ with } \iota:=t_k,
\end{equation}
where $\bm{\omega}_{{\iota}}=\bm{\omega}_{{t_k}}$ denotes the state (coordinate system) at time $t_k$.

To proceed, we make the following assumptions:
\begin{assumption}\label{assump1} 
For any $\bm{x}\in\Omega$, there exists a constant $\mathcal{M}_0>0$, assume that two images $T(\bm{x})$ and $R(\bm{x})$ satisfies:
\[\max\left\{\|T\|_{L^\infty(\Omega)},\|R\|_{L^\infty(\Omega)},\|\nabla T\|_{L^\infty(\Omega)},\|\nabla^2 T\|_{L^\infty(\Omega)}\right\}<\mathcal{M}_0<+\infty;\] 
hence $\|T-R\|_{L^\infty(\Omega)}<2\mathcal{M}_0$.  
Moreover, assume that the regularization functional $\mathcal{S}[{\bm{u}}]$ in (\ref{eq:eq15_subp}) statisfies that

(\romannumeral1) lower semi-continuity (LSC), i.e., let $\bm{u}_j\in\textbf{\emph{DCH}}_{h_t}(\Omega)$ and $\bm{u}_j   
\xlongrightarrow[\mbox{\emph{L}}^1(\Omega)]{*}\bm{u}$, then 
\[\mathcal{S}[{\bm{u}}]\leq \varliminf\limits_{j\rightarrow{+\infty}}\mathcal{S}[{\bm{u}_j}].\]

(\romannumeral2) lower boundedness, i.e., there exists a constant $L$, $\forall\; \bm{u}\in\textbf{\emph{DCH}}_{h_t}(\Omega)$, one has 
\[\mathcal{S}[{\bm{u}}]\geq L.\]
\end{assumption}

Next, in addition to Assumption \ref{assump1}, let us analyze the properties of the energy functional $\mathcal{E}^{{\iota}}(\boldsymbol{u})$.
\begin{lemma}[Lower semi-continuity of $\mathcal{E}^{{\iota}}(\boldsymbol{u})$]
\label{lem_lsc}
Assume that $T(\boldsymbol{x})$ is differentiable with respect to $\boldsymbol{x}$.
 Then the functional $\mathcal{W}(\boldsymbol{u}):=\mathcal{D}^{\alpha(\iota)}_{T,R}[\bm{u}]$ is LSC and consequently
 $\mathcal{E}^{{\iota}}(\boldsymbol{u})$ from (\ref{eq:eq15_subp}) is also LSC, i.e.
for each $\epsilon>0$ and $\bm{u}(\bm{\omega}_{{\iota}})\in\textbf{\emph{DCH}}_{h_t}(\Omega)$ with
 $|\mathcal{W}(\boldsymbol{u})|<\infty$, there exists a $\delta(\epsilon)>0$ such that for
all $\boldsymbol{v}\in\textbf{\emph{DCH}}_{h_t}(\Omega)$ satisfying
 $\|\boldsymbol{u}-\boldsymbol{v}\|_{L^2(\Omega)}<\delta(\epsilon)$, the inequality holds:
\( \mathcal{W}(\boldsymbol{u})<\mathcal{W}(\boldsymbol{v})+\epsilon.\)
\end{lemma}
\begin{proof}
Since the function $T(\boldsymbol{x})$ is differentiable, there exists a real number $\theta\in (0,\;1)$
such that
$$T(\boldsymbol{x}+\boldsymbol{u})=T(\boldsymbol{x}+\boldsymbol{v})+
\nabla T(\boldsymbol{s})\cdot \boldsymbol{h}$$
 where $\boldsymbol{s}=\boldsymbol{x}+\theta \boldsymbol{u}+(1-\theta)
 \boldsymbol{v}$ and $\boldsymbol{h}=\boldsymbol{u}-\boldsymbol{v}$. Hence we have
\[\mathcal{W}(\boldsymbol{u})=\mathcal{W}(\boldsymbol{v})+
\int_\Omega\;\left((T(\boldsymbol{x}+\boldsymbol{v})-R(\boldsymbol{x}))
(\nabla T(\boldsymbol{s})\cdot \boldsymbol{h})
+\boldsymbol{h}^T(\nabla T(\boldsymbol{s})\nabla T(\boldsymbol{s})^T)\boldsymbol{h}\right)d\Omega,\]
From Assumption \ref{assump1}, the above equation leads to
\[|\mathcal{W}(\boldsymbol{u})-\mathcal{W}(\boldsymbol{v})|\leq c_1\|\boldsymbol{u}-\boldsymbol{v}\|_{L^2(\Omega)}+c_2 \|\boldsymbol{u}-\boldsymbol{v}\|^2_{L^2(\Omega)}, c_1\geq 0, c_2\geq 0.\]
In this case, we have, if taking
            $\delta(\epsilon)=(-c_1+\sqrt{c_1^2+4\epsilon c_2})/(2c_2)$,
\(\mathcal{W}(\boldsymbol{u})<\mathcal{W}(\boldsymbol{v})+\epsilon.\)
 Consequently combined with Assumption \ref{assump1} for term (ii) i.e. LSC of
  $\mathcal{S}[{\bm{u}}]$,
 the functional $\mathcal{E}^{\iota}(\boldsymbol{u})$ from (\ref{eq:eq15_subp})
  is  lower semi-continuous.
This proves the Lemma.
\end{proof}

\begin{theorem}
The minimization problem (\ref{eq:eq15_subp}) has at least one solution for any ${{\iota}}\in[0,1]$, $\alpha({{\iota}})$ and $\tau>0$.
\end{theorem}
\begin{proof} The above result (Lemma \ref{lem_lsc}), combined with the compactness of $\mathcal{E}^{\iota}(\boldsymbol{u})$ (Lemma \ref{lemma_4_6}) and
boundedness of $\textbf{DCH}_{h_\iota}(\Omega)$
 (i.e, $\boldsymbol{u}$ is
bounded in $\textbf{DCH}_{h_\iota}(\Omega)$), completes the proof
(see the general framework in \cite[Prop. 38.12(d)]{EZeidler1985}).
\end{proof}

To solve the control increment $\bm{u}(\bm{\omega}_{{\iota}})$ of proposed optimal control problem (\ref{eq:eq15}), we first write (\ref{eq:eq15_subp}) as a saddle point problem defined by
\begin{equation}\label{eq_eq20_a}
\begin{split}
\min\limits_{\bm{u}_\iota}\max\limits_{\lambda_\iota}\left\{\mathcal{L}^{\iota}\big (\bm{u},\lambda):=\mathcal{E}^{{\iota}}\big(\bm{u}\big)- \mathlarger\int_{\Omega}{\lambda (\bm{\omega}_{{\iota}})}\Big(\diver{\bm{u}(\bm{\omega}_{{\iota}})}+\cfrac{\partial h(\bm{\omega}_{{\iota}},t)}{\partial t}\Big|_{t={\iota}}\Big)d\bm{\omega}_{\iota}\right\},
\end{split}
\end{equation}
where ${\lambda}_\iota={\lambda}(\bm{\omega}_{{\iota}})$ is a multiplier function. Obviously, the maximum w.r.t. $\lambda_\iota$ will be $+\infty$, unless $\diver{\bm{u}(\bm{\omega}_{{\iota}})}+\cfrac{\partial h(\bm{\omega}_{{\iota}},t)}{\partial t}\Big|_{t={\iota}}=0$, so this is equivalent to the original problem (\ref{eq:eq15_subp}). But this equivalence is not very useful and computationally. To deal easily with it, we add two 'proximal point' terms
\begin{equation}\label{eq_eq20_b}
\begin{split}
\min\limits_{\bm{u}_\iota}\max\limits_{\lambda_\iota}\Big\{\mathcal{L}^{{\iota}}_{{\beta}}\big (\bm{u},\lambda;\bm{u}^i,\lambda^i\big):=\mathcal{L}^{\iota}\big (\bm{u},\lambda)&+\frac{1}{2 \gamma} \int \big\| \bm{u}(\bm{\omega}_{{\iota}})-\bm{u}^{i}(\bm{\omega}_{{\iota}})\big\|^{2}d \bm{\omega}_{\iota}\\&-\frac{1}{2 \beta} \int \big\| \lambda(\bm{\omega}_{{\iota}})-\lambda^i(\bm{\omega}_{{\iota}})\big\|^{2}d \bm{\omega}_{\iota}\Big\},
\end{split}
\end{equation}
where $\gamma$ and $\beta >0$ are two penalty parameters. Maximization w.r.t. $\lambda$ is now trivial (a concave quadratic), yielding
\begin{equation}\label{eq:eq_multiplier_update}
\lambda(\bm{\omega}_{\iota})= {\lambda }^{i}(\bm{\omega}_{{\iota}}) - \beta\Big(\diver\bm{u}(\bm{\omega}_{{\iota}}) + \cfrac{\partial h(\bm{\omega}_{{\iota}},t)}{\partial t}\Big|_{t={\iota}}\Big).
\end{equation}
Inserting the above multiplier update $\lambda(\bm{\omega}_{\iota})$ into (\ref{eq_eq20_b}) leads to 
\begin{equation}\label{eq:eq20}
\begin{split}
\min\limits_{\bm{u}}\Bigg\{\mathcal{E}_{\bm{u}^i,\lambda^i}&(\bm{u}):= \mathcal{E}^{{\iota}}\big(\bm{u}\big)- \mathlarger\int_{\Omega}{\lambda^i(\bm{\omega}_{{\iota}})}\Big(\diver{\bm{u}(\bm{\omega}_{{\iota}})}+\cfrac{\partial h(\bm{\omega}_{{\iota}},t)}{\partial t}\Big|_{t={\iota}}\Big)d\bm{\omega}_{\iota}\\
+&\frac{1}{2 \gamma} \int \big\| \bm{u}(\bm{\omega}_{{\iota}})-\bm{u}^{i}(\bm{\omega}_{{\iota}})\big\|^{2}d \bm{\omega}_{\iota}+\cfrac{{{\beta }}}{2} \mathlarger\int_{\Omega} \Big(\diver{\bm{u}(\bm{\omega}_{{\iota}})}+\cfrac{\partial h(\bm{\omega}_{{\iota}},t)}{\partial t}\Big|_{t={\iota}}\Big)^2d\bm{\omega}_{\iota}\Bigg\}.
\end{split}
\end{equation}
Thus the iterative scheme of ALMM can be modified as
\begin{equation}\label{eq:eq21}
\left\{ {\begin{array}{rl}
  {{\bm{u}^{i + 1}(\bm{\omega}_{{\iota}})}}   &= \arg \min\limits_{\bm{u}} \mathcal{E}_{\bm{u}^i,\lambda^i}(\bm{u}), \\
  {{\lambda^{i + 1}(\bm{\omega}_{{\iota}})}}   
    &= {{\lambda }^{i}(\bm{\omega}_{{\iota}})} - {\beta}\Big(\diver{\bm{u}^{i + 1}(\bm{\omega}_{{\iota}})} + \cfrac{\partial h(\bm{\omega}_{{\iota}},t)}{\partial t}\Big|_{t={\iota}}\Big).
  \end{array}} \right.
\end{equation}
Hence a new iteration $\big(\bm{u}^{i + 1}(\bm{\omega}_{{\iota}}), {\lambda}^{i + 1}(\bm{x})\big)^T$ is solved.
 
\subsection{Convergence analysis} We now study the convergence of the proposed algorithm based on (\ref{eq:eq21}). The next two Lemmas and Assumption can apply not just to the saddle point problem (\ref{eq_eq20_a}) but also to the alternation iterations (\ref{eq:eq21}).
\begin{lemma}\label{constraint_opti_viq}
Let $\mathcal{X}$ be a closed convex set, $\mathcal{F}_1(\bm{v})$ and $\mathcal{F}_2(\bm{v})$ be convex functions on $\mathcal{X}$
and the first-order variation of $\mathcal{F}_2(\bm{v})$ exists. Assume that the solution set of the $\min\{\mathcal{F}_1(\bm{v})+\mathcal{F}_2(\bm{v})|\bm{v}\in\mathcal{X}\}$ is nonempty. Then
\begin{equation}\label{ineq_convegence_analysis0}
\bm{v}^*=\arg\min\{\mathcal{F}(\bm{v}):=\mathcal{F}_1(\bm{v})+\mathcal{F}_2(\bm{v})|\bm{v}\in\mathcal{X}\}
\end{equation}
if and only if
\begin{equation}\label{ineq_convegence_analysis}
\mathcal{F}_1(\bm{v})-\mathcal{F}_1(\bm{v}^*)+\mathcal{F}_2^\prime(\bm{v}^*)(\bm{v}-\bm{v}^*)\geq 0,\;\exists \bm{v}^*\in\mathcal{X} \text{ and }\forall \;\bm{v}\in\mathcal{X}. 
\end{equation}
\end{lemma}
\begin{proof}
\textbf{(\ref{ineq_convegence_analysis0}) $\Longrightarrow$ (\ref{ineq_convegence_analysis}):} since $\bm{v}^*$ is a minimizer of the convex functional $\mathcal{F}(\bm{v})$ on $\mathcal{X}$, then for any $\bm{v}\in\mathcal{X}$ and $\epsilon\in(0,1]$, we have
\[\mathcal{F}_1(\bm{v}^*+\epsilon(\bm{v}-\bm{v}^*))\leq\epsilon\mathcal{F}_1(\bm{v})+(1-\epsilon)\mathcal{F}_1(\bm{v}^*),\]
\[\frac{\mathcal{F}_1(\bm{v}^*+\epsilon(\bm{v}-\bm{v}^*))-\mathcal{F}_1(\bm{v}^*)}{\epsilon}+\frac{\mathcal{F}_2(\bm{v}^*+\epsilon(\bm{v}-\bm{v}^*))-\mathcal{F}_2(\bm{v}^*)}{\epsilon}\geq 0,\]
which yields
\[\mathcal{F}_1(\bm{v})-\mathcal{F}_1(\bm{v}^*)+\frac{\mathcal{F}_2(\bm{v}^*+\epsilon(\bm{v}-\bm{v}^*))-\mathcal{F}_2(\bm{v}^*)}{\epsilon}\geq 0.\]
Letting $\epsilon\rightarrow 0^+$, the above inequality can be written as 
\[\mathcal{F}_1(\bm{v})-\mathcal{F}_1(\bm{v}^*)+\mathcal{F}_2^\prime(\bm{v}^*)(\bm{v}-\bm{v}^*)\geq 0,\;\forall \;\bm{v}\in\mathcal{X}. \]

\textbf{(\ref{ineq_convegence_analysis}) $\Longrightarrow$ (\ref{ineq_convegence_analysis0}):}  since $\mathcal{F}_2(\bm{v})$ is convex, i.e., $\mathcal{F}_2(\bm{v}^*+\epsilon(\bm{v}-\bm{v}^*))\leq\epsilon\mathcal{F}_2(\bm{v})+(1-\epsilon)\mathcal{F}_2(\bm{v}^*)$, then
\[\mathcal{F}_2(\bm{v})-\mathcal{F}_2(\bm{v}^*)-\frac{\mathcal{F}_2(\bm{v}^*+\epsilon(\bm{v}-\bm{v}^*))-\mathcal{F}_2(\bm{v}^*)}{\epsilon}\geq 0.\] 
Letting $\epsilon\rightarrow 0^+$, the above inequality can be written as
\[\mathcal{F}_2(\bm{v})-\mathcal{F}_2(\bm{v}^*)-\mathcal{F}_2^\prime(\bm{v}^*)(\bm{v}-\bm{v}^*)\geq 0,\;\forall \;\bm{v}\in\mathcal{X}.\]
By substituting it into (\ref{ineq_convegence_analysis}), we obtain
\[\mathcal{F}(\bm{v})-\mathcal{F}(\bm{v}^*)=\mathcal{F}_1(\bm{v})-\mathcal{F}_1(\bm{v}^*)+\mathcal{F}_2(\bm{v})-\mathcal{F}_2(\bm{v}^*)\geq 0,\;\forall \;\bm{v}\in\mathcal{X},\]
which implies that $\bm{v}^*$ is a solution of $\min\limits_{\bm{v}\in\mathcal{X}} \mathcal{F}(\bm{v})$. The proof is complete.
\end{proof}

The above variational equivalence perspective is particularly valid for convex problems, where we can expect to solve the subproblems $\min\limits_{\bm{u}} \mathcal{E}_{\bm{u}^i,\lambda^i}(\bm{u})$ in an alternatively iteration sense. This is reflected in the following assumption, which we require throughout this section.
\begin{assumption}\label{assump_1} 
For any saddle point $(\bm{u}_{\iota}^*,\lambda_{\iota}^*)$ of problem (\ref{eq_eq20_a}), assume that there exists a positive constant $\delta_{\iota}$ and a closed box defined by
\[\mathcal{N}\left((\bm{u}_{\iota}^*,\lambda_{\iota}^*),\delta_{\iota}\right)=\left\{(\bm{u},\lambda)\Big| \|\bm{u}-\bm{u}^*\|^2_{L^2(\Omega)}\leq\delta_{\iota},\;\|\lambda-\lambda^*\|^2_{L^2(\Omega)}\leq\delta_{\iota}\right\},\]
 such that 
\begin{equation}\label{eq_assume}
\begin{split}
{\mathcal{E}^{\iota}}^{\prime\prime}\big(\bm{u}_{\iota}\big)\bm{\xi}\bm{\xi}=\alpha^2({{\iota}})\int_\Omega \Big((T(\bm{\omega}_t+\bm{u}_{\iota})-&R(\bm{\phi}^{-1}(\bm{\omega}_t)))
(\bm{\xi}^T\nabla^2 T(\bm{\omega}_t+\bm{u}_{\iota})\bm{\xi})
\\&+\bm{\xi}^T(\nabla T(\bm{\omega}_t+\bm{u}_{\iota})\nabla T(\bm{\omega}_t+\bm{u}_{\iota})^T)\bm{\xi}\Big) d\bm{\omega}_{\iota}\geq 0
\end{split}
\end{equation}
for any $(\bm{u}_{\iota},\lambda_{\iota})\in \mathcal{N}\left((\bm{u}_{\iota}^*,\lambda_{\iota}^*),\delta_{\iota}\right)$ and $\bm{0}\neq\bm{\xi}\in\textbf{\emph{DCH}}_{h_t}(\Omega)$.
\end{assumption}

One of the most important consequences of the above Assumption condition is that the $\mathcal{E}^{{\iota}}\big(\bm{u}_{\iota}\big)$ is a convex functional on $\mathcal{N}\left((\bm{u}_{\iota}^*,\lambda_{\iota}^*),\delta_{\iota}\right)$. Now, we are ready to reformulate the saddle-point models (\ref{eq_eq20_a}) as special cases of mixed variational inequality for establishing the convergence for the scheme (\ref{eq:eq21}).
\begin{lemma}\label{constraint_opti_viq1}
Let $\mathcal{E}^{{\iota}}\big(\bm{u}_{\iota}\big)$ satisfy Assumption \ref{assump_1} and $(\bm{u}_{\iota}^*,\lambda_{\iota}^*)$ be the saddle point of problem (\ref{eq_eq20_a}), i.e.,
\[\mathcal{L}^{\iota}\big (\bm{u}_{\iota}^*,\lambda_{\iota})\leq\mathcal{L}^{\iota}\big (\bm{u}_{\iota}^*,\lambda_{\iota}^*)\leq \mathcal{L}^{\iota}\big (\bm{u}_{\iota},\lambda_{\iota}^*),\quad \forall\; (\bm{u}_{\iota},\lambda_{\iota})\in \mathcal{N}\left((\bm{u}_{\iota}^*,\lambda_{\iota}^*),\delta_{\iota}\right),\]
if and only if
\begin{equation}\label{eq_viq_opti_cond}
\left\{\begin{array}{rl}
&\mathcal{E}^{{\iota}}\big(\bm{u}_{\iota}\big)-\mathcal{E}^{{\iota}}\big(\bm{u}_{\iota}^*\big)+\int_\Omega(\bm{u}_{\iota}-\bm{u}_{\iota}^*)\nabla\left(\lambda_{\iota}^*\right)d\bm{\omega}_{\iota}\geq 0,\\
&\qquad\int_\Omega(\lambda_{\iota}-\lambda_{\iota}^*)(\diver\bm{u}_{\iota}^*+\cfrac{\partial h(\bm{\omega}_{{\iota}},t)}{\partial t}\Big|_{t={\iota}})d\bm{\omega}_{\iota}\geq 0,
  \end{array}\right. \forall\; (\bm{u}_{\iota},\lambda_{\iota})\in \mathcal{N}\left((\bm{u}_{\iota}^*,\lambda_{\iota}^*),\delta_{\iota}\right).
\end{equation}
Especially, by substituting, we obtain
\[\mathcal{E}^{{\iota}}\big(\bm{u}_{\iota}\big)-\mathcal{E}^{{\iota}}\big(\bm{u}_{\iota}^*\big)+\int_\Omega(\bm{u}_{\iota}-\bm{u}_{\iota}^*)\nabla\left(\lambda_{\iota}^*\right)+(\lambda_{\iota}-\lambda_{\iota}^*)(\diver\bm{u}_{\iota}^*+\cfrac{\partial h(\bm{\omega}_{{\iota}},t)}{\partial t}\Big|_{t={\iota}})d\bm{\omega}_{\iota}\geq 0\]
for any $(\bm{u}_{\iota},\lambda_{\iota})\in \mathcal{N}\left((\bm{u}_{\iota}^*,\lambda_{\iota}^*),\delta_{\iota}\right)$.
\end{lemma}

The mixed variational inequality serves as an unified mathematical model was presented in \cite{Gu2014}, for discrete convex optimization, we generalize it to infinite-dimensional saddle-point problem.
From this point of view, it is natural to expect that the sequence (\ref{eq:eq21}) based on ALMM algorithm is convergent, provided $\{(\bm{u}_{\iota}^{k},\lambda_{\iota}^{k})\}$ satisfies the Fej{\'e}r monotone condition \cite{Bauschke2011}. We are now ready to show this important result which can be seen as a infinite-dimensional variant of \cite{Gu2014}.
\begin{theorem}\label{constraint_opti_viq2}
For any $(\bm{u}^0_{\iota},\lambda^0_{\iota})\in \mathcal{N}\left((\bm{u}_{\iota}^*,\lambda_{\iota}^*),\delta_{\iota}\right)$, let $\mathcal{E}^{{\iota}}\big(\bm{u}_{\iota}\big)$ satisfy Assumption \ref{assump_1} and $(\bm{u}_{\iota}^{i+1},\lambda_{\iota}^{i+1})$ be the solution of the ALMM algorithm (\ref{eq:eq21}), and there exists a saddle point $(\overline{\bm{u}_{\iota}^*},\overline{\lambda_{\iota}^*})$ of problem (\ref{eq_eq20_a}) such that 
\begin{equation}\label{eq_sequency_convegence}
\bm{u}_{\iota}^i\xrightarrow{L^2(\Omega)}\overline{\bm{u}_{\iota}^*},\quad \lambda_{\iota}^i\xrightarrow{L^2(\Omega)}\overline{\lambda_{\iota}^*},\quad\text{ as } i\rightarrow+\infty.
\end{equation}
\end{theorem}

\begin{proof}
\textbf{Step 1.} We should prove that $(\bm{u}_{\iota}^{i},\lambda_{\iota}^{i})$ is a convergent sequence on $\mathcal{N}\left((\bm{u}_{\iota}^*,\lambda_{\iota}^*),\delta_{\iota}\right)$.

Since $\mathcal{E}^{{\iota}}\big(\bm{u}\big)$ is a convex functional on $\mathcal{N}\left((\bm{u}_{\iota}^*,\lambda_{\iota}^*),\delta_{\iota}\right)$, for any $(\bm{u}_{\iota},\lambda_{\iota})\in \mathcal{N}\left((\bm{u}_{\iota}^*,\lambda_{\iota}^*),\delta_{\iota}\right)$, there exists the solution
$(\bm{u}_{\iota}^{i+1},\lambda_{\iota}^{i+1})\in \mathcal{N}\left((\bm{u}_{\iota}^*,\lambda_{\iota}^*),\delta_{\iota}\right)$ of the ALMM algorithm (\ref{eq:eq21}) satisfying that
\begin{equation}\label{eq_viq_opti_2}
\left\{ {\begin{array}{rl}
&\mathcal{E}^{{\iota}}\big(\bm{u}_{\iota}\big)-\mathcal{E}^{{\iota}}\big(\bm{u}_{\iota}^{i+1}\big)+\frac{1}{\gamma}\int_\Omega(\bm{u}_{\iota}-\bm{u}_{\iota}^{i+1})\left(\bm{u}_{\iota}^{i+1}-\bm{u}_{\iota}^{i}\right)d\bm{\omega}_{\iota}
\\&\quad\qquad+\int_\Omega(\bm{u}_{\iota}-\bm{u}_{\iota}^{i+1})\nabla\left(\lambda_{\iota}^{i}-\beta(\diver\bm{u}_{\iota}^{i+1}+\cfrac{\partial h(\bm{\omega}_{{\iota}},t)}{\partial t}\Big|_{t={\iota}})\right)d\bm{\omega}_{\iota}\geq 0,\\
&\quad\qquad\int_\Omega(\lambda_{\iota}-\lambda_{\iota}^{i+1})\left(\lambda_{\iota}^{i+1}-\lambda_{\iota}^{i}+\beta(\diver\bm{u}_{\iota}^{i+1}+\cfrac{\partial h(\bm{\omega}_{{\iota}},t)}{\partial t}\Big|_{t={\iota}})\right)d\bm{\omega}_{\iota}\geq 0.
  \end{array}} \right.
\end{equation}
By taking the multiplier update ${{\lambda }_{\iota}^{i+1}}={{\lambda }_{\iota}^{i}} - {\beta}\Big(\diver{\bm{u}_{\iota}^{i + 1}} + \cfrac{\partial h(\bm{\omega}_{{\iota}},t)}{\partial t}\Big|_{t={\iota}}\Big)$, we have
\begin{equation}\label{eq_viq_opti}
\left\{ {\begin{array}{rl}
\mathcal{E}^{{\iota}}\big(\bm{u}_{\iota}\big)&-\mathcal{E}^{{\iota}}\big(\bm{u}_{\iota}^{i+1}\big)+\int_\Omega(\bm{u}_{\iota}-\bm{u}_{\iota}^{i+1})\left(\nabla\left(\lambda_{\iota}^{i+1}\right)+\frac{1}{\gamma}(\bm{u}_{\iota}^{i+1}-\bm{u}_{\iota}^{i})\right)d\bm{\omega}_{\iota}\geq 0,\\
&\int_\Omega(\lambda_{\iota}-\lambda_{\iota}^{i+1})\left((\diver\bm{u}_{\iota}^{i+1}+\cfrac{\partial h(\bm{\omega}_{{\iota}},t)}{\partial t}\Big|_{t={\iota}})+\frac{1}{\beta}(\lambda_{\iota}^{i+1}-\lambda_{\iota}^{i})\right)d\bm{\omega}_{\iota}\geq 0.
  \end{array}} \right.
\end{equation}
Let us recall that the Gauss formula and the divergence theorem can be employed to obtain
\begin{equation}\label{eq_viq_diver}
\int_\Omega(\bm{u}_{\iota}-\bm{u}_{\iota}^{i+1})\left(\nabla\lambda_{\iota}-\nabla\lambda_{\iota}^{i+1}\right)+(\diver\bm{u}_{\iota}-\diver\bm{u}_{\iota}^{i+1})\left(\lambda_{\iota}-\lambda_{\iota}^{i+1}\right)d\bm{\omega}_{\iota}=0.
\end{equation}
Combining (\ref{eq_viq_opti}) and (\ref{eq_viq_diver}), we derive
\begin{equation}\label{eq_viq_opti3}
\begin{split}
\frac{1}{\gamma}&\int_\Omega(\bm{u}_{\iota}-\bm{u}_{\iota}^{i+1})(\bm{u}_{\iota}^{i+1}-\bm{u}_{\iota}^{i})d\bm{\omega}_{\iota}+\frac{1}{\beta}\int_\Omega(\lambda_{\iota}-\lambda_{\iota}^{i+1})(\lambda_{\iota}^{i+1}-\lambda_{\iota}^{i})d\bm{\omega}_{\iota}\\&\geq\mathcal{E}^{{\iota}}\big(\bm{u}_{\iota}^{i+1}\big)-\mathcal{E}^{{\iota}}\big(\bm{u}_{\iota}\big)+\int_\Omega(\bm{u}_{\iota}^{i+1}-\bm{u}_{\iota})\nabla\left(\lambda_{\iota}\right)+(\lambda_{\iota}^{i+1}-\lambda_{\iota})(\diver\bm{u}_{\iota}+\cfrac{\partial h(\bm{\omega}_{{\iota}},t)}{\partial t}\Big|_{t={\iota}})d\bm{\omega}_{\iota}.
\end{split}
\end{equation}
If taking $(\bm{u}_{\iota},\lambda_{\iota})=(\bm{u}_{\iota}^*,\lambda_{\iota}^*)$ and following Lemma \ref{constraint_opti_viq1}, the inequality (\ref{eq_viq_opti3}) can be rewritten as 
\begin{equation}\label{eq_viq_ineq}
\begin{split}
\frac{1}{\gamma}&\int_\Omega(\bm{u}_{\iota}^*-\bm{u}_{\iota}^{i+1})(\bm{u}_{\iota}^{i+1}-\bm{u}_{\iota}^{i})d\bm{\omega}_{\iota}+\frac{1}{\beta}\int_\Omega(\lambda_{\iota}^*-\lambda_{\iota}^{i+1})(\lambda_{\iota}^{i+1}-\lambda_{\iota}^{i})d\bm{\omega}_{\iota}\\&\geq\mathcal{E}^{{\iota}}\big(\bm{u}_{\iota}^{i+1}\big)-\mathcal{E}^{{\iota}}\big(\bm{u}_{\iota}^*\big)+\int_\Omega(\bm{u}_{\iota}^{i+1}-\bm{u}_{\iota}^*)\nabla\left(\lambda_{\iota}^*\right)+(\lambda_{\iota}^{i+1}-\lambda_{\iota}^*)(\diver\bm{u}_{\iota}^*+\cfrac{\partial h(\bm{\omega}_{{\iota}},t)}{\partial t}\Big|_{t={\iota}})d\bm{\omega}_{\iota}\geq 0.
\end{split}
\end{equation}
Note that if $\int_\Omega \bm{f}^T(\bm{g}-\bm{f})d\bm{\omega}_{\iota}\geq 0$, then 
\[\|\bm{f}\|^2_{L^2(\Omega)}=\int_\Omega\bm{f}^T\bm{f}d\bm{\omega}_{\iota}\leq \int_\Omega\bm{g}^T\bm{g}d\bm{\omega}_{\iota}-\int_\Omega(\bm{g}-\bm{f})^T(\bm{g}-\bm{f})d\bm{\omega}_{\iota}=\|\bm{g}\|^2_{L^2(\Omega)}-\|\bm{g}-\bm{f}\|^2_{L^2(\Omega)}.\]
The fact that the left side in (\ref{eq_viq_ineq}) satisfies $\int_\Omega \bm{f}^T(\bm{g}-\bm{f})d\bm{\omega}_{\iota}\geq 0$ implies that
\begin{equation}\label{eq_viq_opti4}
\begin{split}
\varpi^{i+1}+\vartheta^{i+1}\leq\varpi^{i},
\end{split}
\end{equation}
where $\varpi^{i+1}:=\frac{\beta}{\gamma}\|\bm{u}_{\iota}^{i+1}-\bm{u}_{\iota}^*\|^2_{L^2(\Omega)}+\|\lambda_{\iota}^{i+1}-\lambda_{\iota}^*\|^2_{L^2(\Omega)}$ and $\vartheta^{i+1}:=\frac{\beta}{\gamma}\|\bm{u}_{\iota}^{i+1}-\bm{u}_{\iota}^i\|^2_{L^2(\Omega)}+\|\lambda_{\iota}^{i+1}-\lambda_{\iota}^i\|^2_{L^2(\Omega)}$. Since the sequence $\{\varpi^{i+1}\}$ is positive, bounded and nonincreasing, thus it is convergent and so $\{\vartheta^{i+1}\}$ tends to zero, hence there exists $(\overline{\bm{u}_{\iota}^*}, \overline{\lambda_{\iota}^*})\in \mathcal{N}\left((\bm{u}_{\iota}^*,\lambda_{\iota}^*),\delta_{\iota}\right)$, we deduce from the preceding inequalities that
\[(\bm{u}_{\iota}^i,\lambda_{\iota}^i)\xrightarrow{L^2(\Omega)} (\overline{\bm{u}_{\iota}^*}, \overline{\lambda_{\iota}^*})\quad \text{ as } i\rightarrow +\infty,\]
where the sequence $\{(\bm{u}_{\iota}^i,\lambda_{\iota}^i)\}$ generated by (\ref{eq:eq21}) is Fej{\'e}r monotone sequence  \cite{Bauschke2011}.

\textbf{Step 2.} Next, we have to prove that $(\overline{\bm{u}_{\iota}^*}, \overline{\lambda_{\iota}^*})$ is a saddle point of the Lagrangian functional $\mathcal{L}^{\iota}\big (\bm{u},\lambda)$. For any $(\bm{u}_{\iota},\lambda_{\iota})\in \mathcal{N}\left((\bm{u}_{\iota}^*,\lambda_{\iota}^*),\delta_{\iota}\right)$,  let us consider
\begin{equation}\label{eq_viq_24}
\left\{ {\begin{array}{rl}
\lim\limits_{i\rightarrow+\infty}\Big\{\mathcal{E}^{{\iota}}\big(\bm{u}_{\iota}\big)&-\mathcal{E}^{{\iota}}\big(\bm{u}_{\iota}^{i+1}\big)+\int_\Omega(\bm{u}_{\iota}-\bm{u}_{\iota}^{i+1})\left(\nabla\left(\lambda_{\iota}^{i+1}\right)+\frac{1}{\gamma}(\bm{u}_{\iota}^{i+1}-\bm{u}_{\iota}^{i})\right)d\bm{\omega}_{\iota}\Big\}\\
&=\mathcal{E}^{{\iota}}\big(\bm{u}_{\iota}\big)-\mathcal{E}^{{\iota}}\big(\overline{\bm{u}_{\iota}^*}\big)+\int_\Omega(\bm{u}_{\iota}-\overline{\bm{u}_{\iota}^*})\nabla\left(\overline{\lambda_{\iota}^*}\right)d\bm{\omega}_{\iota}\geq 0,\\
\lim\limits_{i\rightarrow+\infty}&\int_\Omega(\lambda_{\iota}-\lambda_{\iota}^{i+1})\left((\diver\bm{u}_{\iota}^{i+1}+\cfrac{\partial h(\bm{\omega}_{{\iota}},t)}{\partial t}\Big|_{t={\iota}})+\frac{1}{\beta}(\lambda_{\iota}^{i+1}-\lambda_{\iota}^{i})\right)d\bm{\omega}_{\iota}\\
&=\int_\Omega(\lambda_{\iota}-\overline{\lambda_{\iota}^*})\left(\diver\overline{\bm{u}_{\iota}^*}+\cfrac{\partial h(\bm{\omega}_{{\iota}},t)}{\partial t}\Big|_{t={\iota}}\right)d\bm{\omega}_{\iota}\geq 0.
  \end{array}} \right.
\end{equation}
which implies from Lemma \ref{constraint_opti_viq1} that $(\overline{\bm{u}_{\iota}^*}, \overline{\lambda_{\iota}^*})$ is the saddle point of the Lagrangian functional $\mathcal{L}^{\iota}\big (\bm{u},\lambda)$.
\end{proof}

\begin{remark}\label{rek02}
From Theorem \ref{constraint_opti_viq2} and the constraint optimization theory \cite{RGlowinski1989}, let $\mathcal{E}^{{\iota}}\big(\bm{u}_{\iota}\big)$ satisfy Assumption \ref{assump_1}, then we may observe the following:
\begin{description}
  \item[1.] $\overline{\bm{u}_{\iota}}^*:=\overline{\bm{u}^*(\bm{\omega}_{\iota})}$ is the solution of the constraint optimization (\ref{eq:eq15_subp}) with $\iota=t_k$;
  \item[2.] $\{\bm{\omega}_{t_k}:=\bm{\phi} (\bm{x},t_k)\}$ is a convergent sequence and converges to the solution $\bm{\varphi}^*:=\bm{\phi}^*(\bm{x},1)$ of the diffeomorphic image registration model (\ref{model1}) as $t_k\rightarrow 1$ ($k\rightarrow +\infty$).
\end{description}
\end{remark}

%------------------------------------------
\subsection{Numerical implementation}
It can be seen from formula (\ref{eq:eq21}) that the main work of seeking the optimal solution of the model is to solve the subproblem 
\begin{equation}\label{u_subprob}
\min\limits_{\bm{u}(\bm{\omega}_{{\iota}})}\mathcal{E}_{\bm{u}^i,\lambda^i}(\bm{u}),
\end{equation}
so the following discussion will focus on minimizing $\mathcal{E}_{\bm{u}^i,\lambda^i}(\bm{u})$.  
To obtain the optimal solution of (\ref{u_subprob}), the first-order variation of $\mathcal{E}_{\bm{u}^i,\lambda^i}(\bm{u})$ can be used to deduce the Euler-Lagrange equation as follow (\emph{see} \cite{fischer2002fast} for more details)
\begin{equation}\label{eq:eq28}
\begin{split}
\alpha({{\iota}})&\big[T\big(\bm{\omega}_{{\iota}}+\alpha({{\iota}})\bm{u}(\bm{\omega}_{{\iota}})\big)- R\big(\bm{\phi}^{-1}(\bm{\omega}_{\iota})\big) \big] \nabla T\big(\bm{\omega}_{{\iota}}+\alpha({{\iota}})\bm{u}(\bm{\omega}_{{\iota}})\big)-\tau  \Delta {\bm{u}(\bm{\omega}_{{\iota}})}\\
&+{\beta}\left(\nabla \big(-\cfrac{\partial h(\bm{\omega}_{{\iota}},t)}{\partial t}\Big|_{t={\iota}}+ \frac{{{\lambda^i}(\bm{\omega}_{{\iota}})}}{\beta}\big)- \nabla\diver{\bm{u}(\bm{\omega}_{{\iota}})}\right)+\frac{1}{\gamma}(\bm{u}(\bm{\omega}_{{\iota}})-\bm{u}^i(\bm{\omega}_{{\iota}})) = 0
\end{split}
\end{equation}
with the boundary condition $\cfrac{\partial{\bm{u}}}{\partial{\bm{n}}}=\bm{0}  \text { on } \partial \Omega $.
Finally, combining (\ref{eq:eq28}) and the boundary condition in (\ref{eq_rot}), the solving nonlinear systems can be rewritten as
\begin{equation}\label{eq:eq29}
\left\{\begin{array}{rll}
- \Delta {\bm{u}}(\bm{\omega}_{{\iota}})- \frac{\beta}{\tau}\nabla\diver{\bm{u}(\bm{\omega}_{{\iota}})}+\frac{1}{\gamma}\bm{u}(\bm{\omega}_{{\iota}})=& {\bm{r}}\big(\bm{\omega}_{{\iota}}+\alpha({{\iota}})\bm{u}(\bm{\omega}_{{\iota}})\big),& \text{ in } \Omega_{\text{in}} \\
  {\bm{u}}(\bm{\omega}_{{\iota}}) =& \bm{0} \text{ and } \cfrac{\partial{\bm{u}(\bm{\omega}_{{\iota}})}}{\partial{\bm{n}}}= \bm{0}, & \text{ on } \partial \Omega,
\end{array}\right.
\end{equation}
where
\begin{equation*}
\begin{split}
{\bm{r}}\big(\bm{\omega}_{{\iota}}+\alpha({{\iota}})\bm{u}(\bm{\omega}_{{\iota}})\big):=&-\frac{\alpha({{\iota}})}{\tau}\Big[T\big(\bm{\omega}_{{\iota}}+\alpha({{\iota}})\bm{u}(\bm{\omega}_{{\iota}})\big) - R\big(\bm{\phi}^{-1}(\bm{\omega}_{\iota})\big) \Big]\times\\ \nabla T\big(\bm{\omega}_{{\iota}}&+\alpha({{\iota}})\bm{u}(\bm{\omega}_{{\iota}})\big)
-\frac{{\beta}}{\tau }\nabla\Big(-\cfrac{\partial h(\bm{\omega}_{{\iota}},t)}{\partial t}\Big|_{t={\iota}} + \frac{{{\lambda^i}(\bm{\omega}_{{\iota}})}}{{\beta} }\Big)+\frac{1}{\gamma}\bm{u}^i(\bm{\omega}_{{\iota}}).
\end{split}
\end{equation*}
Therefore, the minimum point ${\bm{u}(\bm{\omega}_{{\iota}})}$ of $\mathcal{E}_{\bm{u}^i,\lambda^i}(\bm{u})$ can be transformed into the solution of the nonlinear differential equations \eqref{eq:eq29}.

\subsubsection{Discrete setting} 
Before we present some numerical results from applying the diffemorphic registration model, let us first introduce the discrete setting which we will use in the rest of this paper. Without loss of generality, we restrict our attention to 2D in the following part. 

To find the finite difference scheme associated with the equation (\ref{eq:eq29}), we establish a uniform Cartesian mesh of size $m\times n$ with cell centers: $\{(x_i,y_j):=(ih_x,jh_y):i=1,\dots, m,\;j=1,\dots, n\}$, where $h_x=1/m$ and $h_y=1/n$ denote the spacing step length and $(i, j)$ denotes the index of the locations $(ih_x, jh_y)$ in image domain $\Omega:=[0, 1]^2$. The grid line intersections (or grid nodes) are labeled using fractional indices. For example, the grid node in the upper left corner of the cell centered at $(i,j)$ is labeled $(i-\frac{1}{2},j-\frac{1}{2})$. Fig.\ref{fig_grid2} illustrates the labeling scheme. Here for the time step ${\iota}=t_k\in[0,1]$, the boundary conditions $\bm{u}(\bm{\omega}_{{\iota}})|_{\partial\Omega}= \bm{0}$ and $\cfrac{\partial{\bm{u}}}{\partial{\bm{n}}}=\bm{0}$ of the control increment $\bm{u}$ are added, i.e.,
 \begin{equation*}\begin{split}
 u^l_{\frac{1}{2},j}=0,\quad u^l_{m+\frac{1}{2},j}=0,&\quad u^l_{i,\frac{1}{2}}=0,\quad u^l_{i,n+\frac{1}{2}}=0,\\
 u^l_{0,j}=u^l_{1,j},\quad u^l_{m,j}=u^l_{m+1,j}, & \quad u^l_{i,0}=u^l_{i,1},\quad u^l_{i,n}=u^l_{i,n+1},\\
 & \text{ for all } i=1,\dots,m,\quad j=1,\dots,n \text{ and }l=1,2.
 \end{split}\end{equation*}
Naturally, the simplest method is to consider separately the discretization of
each second-order derivative in $x$ and $y$, which is equivalent to using the one-dimensional approximation. By doing so,
 let us define the discrete schemes applied to a scalar function $v$ and a vectorial $\bm{u}=(u^1,u^2)$ at the grid point $(i, j)$ by
\[(\nabla v)_{ij}=(\delta_x v_{ij},\delta_y v_{ij}),\quad (\diver\bm{u})_{ij}=\delta_x u^1_{ij}+\delta_y u^2_{ij},\]
\[(\Delta u^l)_{ij}=\delta_{xx} u^l_{ij}+\delta_{yy} u^l_{ij},\quad (\Delta \bm{u})_{ij}=\left((\Delta u^1)_{ij},(\Delta u^2)_{ij}\right),\]
where
  \[\delta_x v_{i,j}=(v_{i+1,j}-v_{i-1,j})/(2h_x),\;\delta_y v_{i,j}=(v_{i,j+1}-v_{i,j-1})/(2h_y),\]
  \[\delta_{xx} u^l_{i,j} = (u^l_{i-1,j}-2u^l_{i,j}+u^l_{i+1,j})/h_x^2,\;\delta_{yy} u^l_{i,j} = (u^l_{i,j-1}-2u^l_{i,j}+u^l_{i,j+1})/h_y^2.\]  

As mentioned previously, the gradient $\nabla T$ is obtained from the first variation of $\text{SSD}(\bm{u})$. To discrete $\nabla T$ at cell centers around boundary $\partial\Omega$, it would be interesting to take into account the Neumann boundary condition
$\frac{\partial T(\bm{x})}{\partial \bm{n}}=0$ on $\partial \Omega $, i.e
 \begin{equation*}\begin{split}
 T_{0,j}=T_{1,j},\quad T_{m+1,j}=T_{m,j},&\quad T_{i,0}=T_{i,1},\quad T_{i,n+1}=T_{i,n},\\
 & \text{ for all } i=1,\dots,m,\quad j=1,\dots,n.
 \end{split}\end{equation*}
Then we have 
\[(\nabla T)_{ij}=(\delta_x T_{ij},\delta_y T_{ij}),\]
\[\delta_x T_{ij}=(T_{i+1,j}-T_{i-1,j})/(2h_x);\;\;\delta_y T_{ij}=(T_{i,j+1}-T_{i,j-1})/(2h_y).
\]
Here we employ bicubic interpolator, which takes more surrounding pixels into consideration, to obtain the most image information at the non-grid point from the weighted average of the most recent 16 sampling points in the rectangular grid.

To summarize, the approximation of (\ref{eq:eq29}) is then given by
 \begin{equation}\label{sub_eq_1}
 \begin{split}
 -(\Delta u^1)_{ij}-\frac{\beta}{\tau}\delta_x\big((\diver\bm{u})_{ij}\big)+\frac{1}{\gamma}u^1_{ij}=-&\frac{\alpha(t_k)}{\tau}(T^k-R)_{ij}\delta_x\big((T^k)_{ij}\big)\\+&\frac{\beta}{\tau}\delta_x\big(\big(h^{\prime}(\cdot,t_k)- \frac{{{\lambda(\cdot)}}}{{\beta} }\big)_{ij}\big)+\frac{1}{\gamma}\bar{u}^1_{ij},\\
 -(\Delta u^2)_{ij}-\frac{\beta}{\tau}\delta_y\big((\diver\bm{u})_{ij}\big)+\frac{1}{\gamma}u^2_{ij}=-&\frac{\alpha(t_k)}{\tau}(T^k-R)_{ij}\delta_y\big((T^k)_{ij}\big)\\+&\frac{\beta}{\tau}\delta_y\big(\big(h^{\prime}(\cdot,t_k)- \frac{{{\lambda(\cdot)}}}{{\beta} }\big)_{ij}\big)+\frac{1}{\gamma}\bar{u}^2_{ij},
 \end{split}
 \end{equation}
 where $T^k=T\big(\bm{\omega}_{t_k}+\alpha({t_k})\bm{u}(\bm{\omega}_{t_k})\big)$ and $h^{\prime}(\cdot,t_k)=\cfrac{\partial h(\bm{\omega}_{t_k},t)}{\partial t}\Big|_{t=t_k}$. Now we can consider the augmented Lagrangian multiplier update (\ref{eq:eq_multiplier_update}), which can be seen as the combination of the previous discretization. So the discrete scheme is
\begin{equation}\label{eq_multiplier_discrete}
(\lambda^{\ell+ 1}_k)_{ij} =(\lambda^\ell_k)_{ij} - {\beta}\big(\diver{{\bm{u}^{\ell+ 1}}} + \cfrac{\partial h(\bm{\omega}_{t_k},t)}{\partial t}\Big|_{t=t_k}\big)_{ij}.
\end{equation}

\subsubsection{Deformation quality indicator and folding correction}
As has been discussed, the folding of deformed mesh in diffeomorphic image registration, measured often by the local quantity ${\det (\nabla{\bm{\varphi}} (\bm{x}))}$, should be reduced or avoided. Especially in infinite-dimension case, ${\det (\nabla{\bm{\varphi}} (\bm{x}))} >0$ ensures that the deformation ${\bm{\varphi}} (\bm{x})$ is an one-to-one mapping.

\textbf{Grid unfolding indicator.} We are motivated by a finite volume approach of Haber and Modersitzki \cite{haber_ImageRegistration_2007,haber_NumericalMethods_2004} to discuss the Jacobian determinant value at the cell center $(i,j)$
in 2D. For a straightforward central-difference discretization, the unfolding quantity $\det \big(\nabla \bm{\varphi}\big)\big|_{O}$ of the deformation $O:=\bm{\varphi}_{ij}$ at cell center $(i,j)$ is given by
\begin{equation}\label{det_jacobian}
\begin{split}
\det \big(\nabla \bm{\varphi}\big)\big|_{O}&= \left| {\begin{array}{*{20}{c}}
  \delta_x \varphi^1_{ij}&\delta_y \varphi^1_{ij} \\
  \delta_x \varphi^2_{ij}&\delta_y \varphi^2_{ij}
  \end{array}} \right|=\frac{1}{4h_xh_y}\left| {\begin{array}{*{20}{c}}
  \varphi^1_{i+1,j}-\varphi^1_{i-1,j}&\varphi^1_{i,j+1}-\varphi^1_{i,j-1} \\
  \varphi^2_{i+1,j}-\varphi^2_{i-1,j}&\varphi^2_{i,j+1}-\varphi^2_{i,j-1}
  \end{array}} \right|\\
  =&\frac{1}{4h_xh_y}\Bigg(\left| {\begin{array}{*{20}{c}}
  \varphi^1_{i+1,j}-\varphi^1_{ij}&\varphi^1_{i,j+1}-\varphi^1_{ij} \\
  \varphi^2_{i+1,j}-\varphi^2_{ij}&\varphi^2_{i,j+1}-\varphi^2_{ij}
  \end{array}} \right|+\left| {\begin{array}{*{20}{c}}
  \varphi^1_{i,j+1}-\varphi^1_{ij}&\varphi^1_{i-1,j}-\varphi^1_{ij} \\
  \varphi^2_{i,j+1}-\varphi^2_{ij}&\varphi^2_{i-1,j}-\varphi^2_{ij}
  \end{array}} \right|\\
&\qquad+\left| {\begin{array}{*{20}{c}}
  \varphi^1_{i-1,j}-\varphi^1_{ij}&\varphi^1_{i,j-1}-\varphi^1_{ij} \\
  \varphi^2_{i-1,j}-\varphi^2_{ij}&\varphi^2_{i,j-1}-\varphi^2_{ij}
  \end{array}} \right|+\left| {\begin{array}{*{20}{c}}
  \varphi^1_{i,j-1}-\varphi^1_{ij}&\varphi^1_{i+1,j}-\varphi^1_{ij} \\
  \varphi^2_{i,j-1}-\varphi^2_{ij}&\varphi^2_{i+1,j}-\varphi^2_{ij}
  \end{array}} \right|
  \Bigg)\\
  =&\frac{1}{2}(R^{ij}_{\Delta OBD}+R^{ij}_{\Delta ODA}+R^{ij}_{\Delta OAC}+R^{ij}_{\Delta OCB}),
  \end{split}
  \end{equation}
  where five adjacent cell centers ($O, A, B, C$ and $D$ in Fig.\ref{fig_grid2}(a)) are involved, and $R^{ij}_{\Delta OBD}$ is an area ratio of the triangle signed area $\frac{1}{2}\left| {\begin{array}{*{20}{c}}
  \varphi^1_{i+1,j}-\varphi^1_{ij}&\varphi^1_{i,j+1}-\varphi^1_{ij} \\
  \varphi^2_{i+1,j}-\varphi^2_{ij}&\varphi^2_{i,j+1}-\varphi^2_{ij}
  \end{array}} \right|$ to the area element $h_xh_y$.
 However, if we allow the deformation $\bm{\varphi}$
to shrink or enlarge area in a certain amount,
the deformation is allowed to be much more irregular
and the displacement regularity is not guaranteed even if the above value $\det \big(\nabla \bm{\varphi}\big)\big|_{O}>0$.
An intuitive example is illustrated in Fig.\ref{fig_grid_correction}(a), where the movement of the right point $D$ results in a “twist” of the
box, but this twist can not be observed by computing the quantity $\det \big(\nabla \bm{\varphi}\big)\big|_{O}=\frac{1}{2}(|R^{ij}_{\Delta OBD}|-|R^{ij}_{\Delta ODA}|+|R^{ij}_{\Delta OAC}|+|R^{ij}_{\Delta OCB}|)$. Since the ratio $R^{ij}_{\Delta ODA}$ computed by the triangle $\Delta ODA$ is negative and each one of the other ratios is positive. Especially when point $D$
 is close to $A$ in Fig.\ref{fig_grid_correction}(a), the quantity $\det \big(\nabla \bm{\varphi}\big)\big|_{O}$ may
even be large although a twist has occurred. 

\begin{figure}
\begin{center}
    \includegraphics[width=0.95\textwidth]{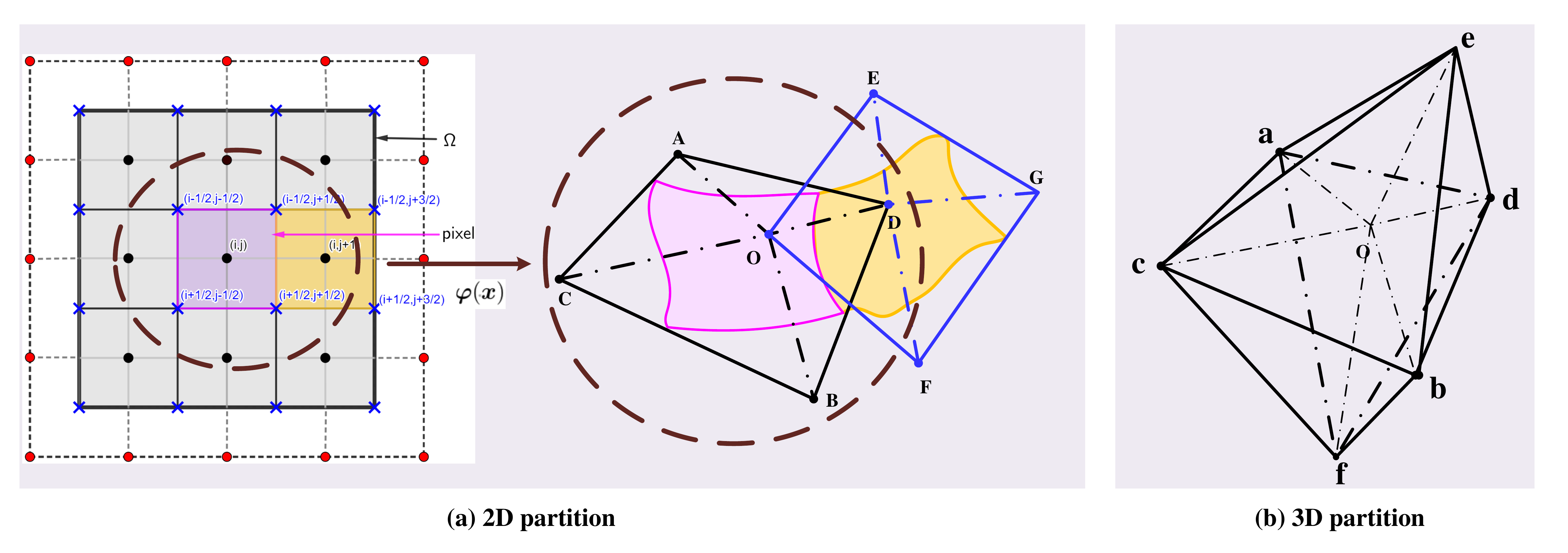}     
\end{center}
\caption{Cartesian grid, deformation $\bm{\varphi}(\bm{x})$ and triangle partition for deformation quality evaluations in 2D and 3D case, and overlapping correction, where five adjacent cell centers in 2D are denoted by $A:=\bm{\varphi}_{i-1,j}=(\varphi^1_{i-1,j},\varphi^2_{i-1,j})$, $B:=\bm{\varphi}_{i+1,j}=(\varphi^1_{i+1,j},\varphi^2_{i+1,j})$, $C:=\bm{\varphi}_{i,j-1}=(\varphi^1_{i,j-1},\varphi^2_{i,j-1})$, $D:=\bm{\varphi}_{i,j+1}=(\varphi^1_{i,j+1},\varphi^2_{i,j+1})$ and $O:=\bm{\varphi}_{ij}=(\varphi^1_{ij},\varphi^2_{ij})$.
}\label{fig_grid2}
\end{figure}

To prevent twists and singular points, Haber and
Modersitzki \cite{haber_NumericalMethods_2004} demanded that the area of every deformed box is preserved for vertex grid discretization. However, a cell center grid (\emph{see} the black points in Fig.\ref{fig_grid2}(a)) and the central-difference are used in this work, we therefore propose a signed area minimum ratio indicator that
can detect twists of the grid, i.e.,
\begin{equation}
R_O=R_{ij}:=\min(R^{ij}_{\Delta OBD},R^{ij}_{\Delta ODA},R^{ij}_{\Delta OAC},R^{ij}_{\Delta OCB}).
\end{equation}
Based on the previous considerations, a triangle cannot twist without its measure $R_{ij}$
to change sign, so we consider our grid quality
evaluation on a triangulation. If a triangle folds, its $R_{ij}$ becomes negative. Such an evaluation is consistent even
in cases of large deformations.

Similar to 2D case, we then also discuss the extension to 3D. The deformation $\bm{\varphi}$ is
discretized at the center of each voxel. Once again, a
Jacobian determinant $\det \big(\nabla \bm{\varphi}\big)\big|_{o}$ of the deformation $o:=\bm{\varphi}_{ijk}$ based on the deformed center $(i,j,k)$ cannot detect
twists. Therefore, every octahedron with six corners $a, b, c, d, e, f$ and one center $o$ is divided into eight
tetrahedrons; cf. Fig.\ref{fig_grid2}(b). The Jacobian determinant with respect to $(i,j,k)$ can be defined by
\begin{equation}
\begin{split}
\det &\big(\nabla \bm{\varphi}\big)\big|_{o}= \left| {\begin{array}{*{20}{c}}
  \delta_x \varphi^1_{ijk}&\delta_y \varphi^1_{ijk}&\delta_z \varphi^1_{ijk} \\
  \delta_x \varphi^2_{ijk}&\delta_y \varphi^2_{ijk}&\delta_z \varphi^2_{ijk}\\
  \delta_x \varphi^3_{ijk}&\delta_y \varphi^3_{ijk}&\delta_z \varphi^3_{ijk}
  \end{array}} \right|\\=&\frac{1}{8h_xh_yh_z}\left| {\begin{array}{*{20}{c}}
  \varphi^1_{i+1,j,k}-\varphi^1_{i-1,j,k}&\varphi^1_{i,j+1,k}-\varphi^1_{i,j-1,k} &\varphi^1_{i,j,k+1}-\varphi^1_{i,j,k-1} \\
  \varphi^2_{i+1,j,k}-\varphi^2_{i-1,j,k}&\varphi^2_{i,j+1,k}-\varphi^2_{i,j-1,k} &\varphi^2_{i,j,k+1}-\varphi^2_{i,j,k-1} \\
  \varphi^2_{i+1,j,k}-\varphi^2_{i-1,j,k}&\varphi^2_{i,j+1,k}-\varphi^2_{i,j-1,k} &\varphi^2_{i,j,k+1}-\varphi^2_{i,j,k-1} 
  \end{array}} \right|\\
  =&\frac{1}{8h_xh_yh_z}\sum\limits_{i',j',k'\in\{-1,+1\}}\left| {\begin{array}{*{20}{c}}
  \varphi^1_{i+i',j,k}-\varphi^1_{i,j,k}&\varphi^1_{i,j+j',k}-\varphi^1_{i,j,k} &\varphi^1_{i,j,k+k'}-\varphi^1_{i,j,k} \\
  \varphi^2_{i+i',j,k}-\varphi^2_{i,j,k}&\varphi^2_{i,j+j',k}-\varphi^2_{i,j,k} &\varphi^2_{i,j,k+k'}-\varphi^2_{i,j,k} \\
  \varphi^2_{i+i',j,k}-\varphi^2_{i,j,k}&\varphi^2_{i,j+j',k}-\varphi^2_{i,j,k} &\varphi^2_{i,j,k+k'}-\varphi^2_{i,j,k} 
  \end{array}} \right|\\
  =&\frac{6}{8h_xh_yh_z}(\mathcal{V}_{odea}+\mathcal{V}_{oaec}+\mathcal{V}_{oceb}+\mathcal{V}_{obed}+\mathcal{V}_{ocfa}+\mathcal{V}_{obfc}+\mathcal{V}_{odfb}+\mathcal{V}_{oafd}),
  \end{split}
\end{equation}
where $\mathcal{V}_{odea}$ of the above formulas is the signed volume with the tetrahedron $\mathcal{V}(o,d,e,a)$.

\begin{figure}[ht]
\begin{center}
    \includegraphics[width=0.95\textwidth]{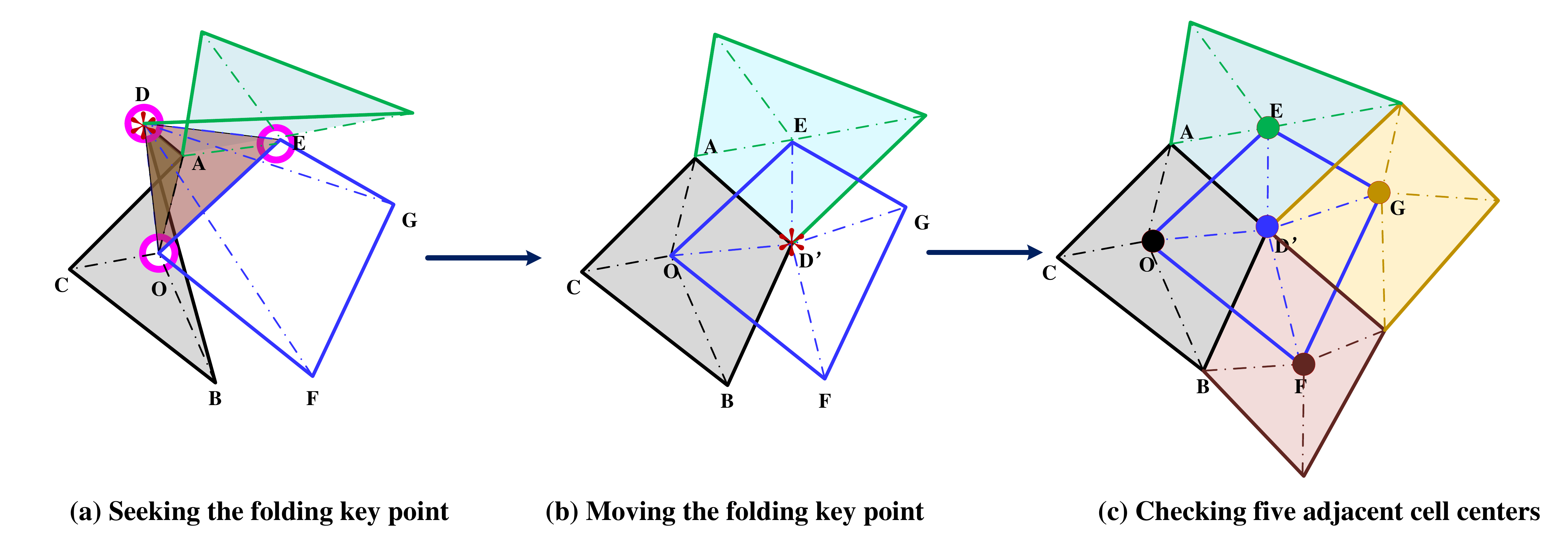}     
\end{center}
\caption{Grid folding correction.
}\label{fig_grid_correction}
\end{figure}

\textbf{Grid folding correction.} Triangulation is widely used in geometric computing, also in image processing community, the unfolding of a single triangle element is the favorable geometric properties of the triangulation that make it so useful. The fundamental property is that a triangle deformation cannot twist without their areas
to change sign, such deformation grid $\bm{\varphi}$ is regular even
in cases of large deformations.

However, while the area of a triangle changes sign, the corresponding grid unfolding indicator $R_{ij}$ is negative, and the grid degeneracy arises. To detect the grid degeneracy, we first collect such point $(i,j)$ to set $S:=\{ (i,j)~|~R_{ij}<0\}$ (\emph{see} three points marked by the magenta circles in Fig.\ref{fig_grid_correction}(a)). Further, we define the folding degree of a point $P$ by
\[\textbf{Deg}_P= \text{the triangle partition number with a vertex } P \text{ and the negative area ratio}.\]
to seek the folding key point. Using the above definition, we have that $\textbf{Deg}_O=2$, $\textbf{Deg}_E=3$ and $\textbf{Deg}_D=4$. Hence the point $D$ with large folding degree is selected as a folding key point (\emph{see} Fig.\ref{fig_grid_correction}(a)).

To correct the grid point $D$ resulting in “twist” of the
boxes, the folding key point $D$ is moved to $D^\prime$, this movement should ensure that the grid unfolding indicators ($R_O$, $R_E$, $R_G$, $R_F$ and $R_{D^\prime}$) of five adjacent cell centers ($O, E, G, F$ and $D^\prime$ respectively) are positive. We illustrate this process in Fig.\ref{fig_grid_correction}. The deformation correction algorithm for using the above technique and backtrack strategy is given in Algorithm \ref{alg:MA2}.

\begin{algorithm}[h]
  \caption{Deformation correction}\label{alg:MA2}
  \KwIn{ $\bm{\varphi}$, $\Delta t$, $\bm{v}$, $\rho$, $ \epsilon = 1\times 10^{-2}$, $m$ and $n$;}
  Compute $ \{ R_1, R_2, R_3, R_4 \}:= \{R_{\triangle OBD}, R_{\triangle ODA}, R_{\triangle OAC}, R_{\triangle OCB}\}$ and $ R_{ij}:= \min \{R_1(i,j),R_2(i,j),R_3(i,j),R_4(i,j)\} $\;
  Find $ S:=\{ (i,j)~|~R_{ij}< \epsilon\} $ and the folding point set $P$ causing the overlaps
   $P:=\{P_{ij}\in \{(i,j),(i-1,j),(i,j-1),(i+1,j),(i,j+1)|(i,j)\in S\}\}$ \;
  \uIf{$|P| = \varnothing$}{
    $\Delta t=2\Delta t$, \textbf{break}\;
  }
  \uElseIf{$|P|/n/m>\rho$}{
$\Delta t=\Delta t/2$, $ \bm{\varphi} = \textbf{RK4}(\bm{\varphi}_0, \bm{v},\Delta t) $, and go to \textbf{line 1}\;
  }
  \Else{
  \ForEach{$(i,j)\in P$}
  {optimize $\bm{\varphi}_{ij}$ such that $\min\{R_{ij},R_{i-1,j},R_{i,j-1},R_{i+1,j},R_{i,j+1}\}\geq\epsilon$
    }
  } 
  \KwOut{ $ \bm{\varphi} $, $\Delta t$; }
\end{algorithm}
\subsubsection{Algorithm}
Based on the above discussion, the algorithm based on ALMM solving the proposed diffeomorphic registration model with an optimal control relaxation (ALMM-based OCRDIR) and the deformation correction can be summarized in Algorithm \ref{alg:ALMM_OCRDI}.

\begin{algorithm}[h]
  \caption{A robust ALMM-based OCRDIR algorithm for using the deformation correction and backtrack
strategy.}\label{alg:ALMM_OCRDI}
  \KwIn{ $\mathcal{R}, \mathcal{T}, \tau, \beta,\gamma, N, \text{MaxIter}, tol$;}
  Initial: $\bm{\omega}_0=\bm{x}, g_{_0}(\bm{x})=1, \bm{u}_0^0 =\bm{0}, \lambda _0^0=1, t_0=0, k=0, \Delta t_0 = 1/ N$, $\text{flag}=0$, $\sigma_\epsilon = 0.01 $\;
  \While{$t_k\leq1$ and $\text{\emph{flag}}\leq 1$}{
    Compute: $g_k := g(\bm{\omega}_{k})\approx\frac{1}{2\pi\sigma_\epsilon^2}\int_\Omega g(\bm{y})\exp\left(-\frac{\|\bm{y}-\bm{\omega}_{k}\|^2}{2\sigma_\epsilon^2}\right)d\bm{y}$\;
    Compute: $ \alpha_k:=\alpha({k})=\frac{\Delta {t_k}}{h(\bm{\omega}_{k},t_k)}$\;
    \For{$\ell =0,\cdots, \emph{MaxIter}$}
    {
      Update $\bm{u}^{\ell+ 1}_{k}$ by solving the control increment equation \eqref{sub_eq_1}\;      
      Update $\lambda^{\ell+ 1}_{k}$ by using the scheme \eqref{eq_multiplier_discrete}\;
     Update $ \mathcal{T}(\bm{\omega}_k + \alpha_k{\bm{u}^{\ell+1}_k}) $\;
      \If{$\frac{\| {\bm{u}^{\ell+1}_k} - {\bm{u}^{\ell}_k}\|_2}{ \| {\bm{u}^{\ell+1}_k} - {\bm{u}^{0}_k}\|_2 } < tol$}{
        \textbf{break}\;
      }
    } 
      Deformation prediction: $ \bm{\omega}_{k+1} = \textbf{RK4}(\bm{\omega}_k, \bm{v}_k,\Delta t_k) $\;
      Deformation correction: $ \bm{\omega}_{k+1} $ and $\Delta t_k$ by using Algorithm \ref{alg:MA2}\;
    Compute $ R_{\min}^{k} $, $ {\det_{\min}^{k}} $, $ {\det_{\max}^{k}} $\;
    $t_{k+1} =t_{k}+\Delta t_k$\;
    \If{$t_{k+1}\geq1$}{
    $t_{k+1}=1$, $\text{flag}=\text{flag}+1$\;
    }
    $ \bm{u}_{k+1}^0 =  \bm{u}_{k}^{\ell + 1}$\;
    $k=k+1$\;
  }
  \KwOut{ $\bm{\omega}_{k+1}$ and $ \mathcal{T}(\bm{\omega}_{k+1}) $; }
\end{algorithm}
%----------------------------------------------------------------------------------------
%\newpage
\section{Experimental results}\label{sec:Section5} 
In this section, through numerical experiments, we aim to
\begin{enumerate}
  \item[\textbf{i.)}] select an appropriate composite function $h(\bm{\phi} (\bm{x},t),t)$ of the proposed OCRDIR model;  
  \item[\textbf{ii.)}] compare with the demons-type models to conclude that the new model has better performance; 
  \item[\textbf{iii.)}] compare with the other popular diffeomorphic models to show the diffeomorphism and accuracy of our OCRDIR model for large deformed images.
\end{enumerate}

To assess the performance of our OCRDIR method, we compare
\begin{itemize} 
\item[$\bm{-}$] the qualitative evaluations including the registered images, the errors between the registered image $T(\bm{\varphi}(\bm{x}))$ and \emph{reference} image $R(\bm{x})$, the deformation grids, the displacement fields and the Jacobian determinant hotmaps of the displacement fields.
\item[$\bm{-}$] the quantitative evaluations of the deformation $\bm{\varphi}$ including
\begin{itemize} 
\item[\textbf{a.)}] Grid unfolding indicator defined by $R_{\min}=\min\limits_{i,j} R_{ij}$, $R_{\min}>0$ (denoted by "+") indicates geometrically
that the grid folding is not detected, while if $R_{\min}\leq 0$ (denoted by "-") means singularity at $\bm{x}$, the grid folding has occurred;
\item[\textbf{b.)}] Jacobian determinant measures defined by
\begin{equation*}\label{eq:eq30}
\begin{array}{rl}
\det \big(\nabla \bm{\varphi}\big)\big|_{ij},&\; \overline{\det}(J(\bm{\varphi}))=\cfrac{1}{mn} \sum\limits_{i,j}\det \big(\nabla \bm{\varphi}\big)\big|_{ij},\\
  \det_\text{min}(J(\bm{\varphi}))=\min\limits_{i,j}\det\big(\nabla \bm{\varphi}\big)\big|_{ij},&\;
  \det_\text{max}(J(\bm{\varphi}))=\max\limits_{i,j}\det\big(\nabla \bm{\varphi}\big)\big|_{ij}.   
\end{array}
\end{equation*}
When the grid unfolding indicator $R_{\min}>0$, $\det \big(J(\bm{\varphi})\big)=1$ indicates geometrically that no area (volume) change has occurred, and $\det \big(J(\bm{\varphi})\big)>1$ implies area (volume) growth, and  $0<\det \big(J(\bm{\varphi})\big)<1$ indicates shrinkage.
\item[\textbf{c.)}] the $\rm{Re_{-}SSD}$ (the relative Sum of Squared Differences) defined by
\begin{equation*}
\text{Re\_SSD}(T,R,T(\bm{\varphi}))= \cfrac{\int_\Omega(T(\bm{\varphi}(\bm{x}))-R(\bm{x}))^2\;d\bm{x}}{\int_\Omega(T(\bm{x})-R(\bm{x}))^2\;d\bm{x}}.
\end{equation*}
\item[\textbf{d.)}] the $\emph{ssim}$  (Structural Similarity) defined by 
\begin{equation*}
\emph{ssim}(R, T)   = \frac{\left(2 \mu_{_{R}} \mu_{_{T}}+c_{1}\right)\left(2 \sigma_{_{R}{_{T}}}+c_{2}\right)}{\left(\mu_{_{R}}^{2}+\mu_{_{T}}^{2}+c_{1}\right)\left(\sigma_{_{R}}^{2}+\sigma_{_{T}}^{2}+c_{2}\right)},
\end{equation*}
which is used to measure luminance, contrast  and structure between image $R$ and image $T$. Where $ \mu_{_{R}} $ and $ \mu_{_{T}} $  are the mean of $R$ and $T$,  $ \sigma_{_{R}{_{T}}} $ is the covariance of $R$ and $T$, $ \sigma_{_{R}}^{2} $ and $ \sigma_{_{T}}^{2} $  are the variance of $R$ and $T$, $ c_{1} $ and $ c_{2} $ are constants.  $\emph{ssim}$ is  between $0$ and $1$, when the two images are exactly the same, $\emph{ssim}=1$.
\end{itemize} 
\end{itemize}

It should be noted that a good registration should not only generate a small $\rm{Re_{-}SSD}$, but also ensure that the grid unfolding indicator $R_{\min}>0$. In an ideal case, we should focus on seeking an approximation area-preserving deformations ($\overline{\det}\approx 1$), which can be restricted by the constraint condition.

\subsection{Comparison for parameter choices}\label{sec:Section6} 
The purpose of this part is to show how sensitive our approach is with respect to regularization parameter $\tau$, also penalty parameters $\gamma$ and $\beta$. On one hand, a regularization parameter $\tau$ balancing the trade-off between a good SSD metric and a smooth solution is difficult to be fixed: if
the value $\tau$ is too large, then it is a poor deformation $\bm{\varphi}(\bm{x})$ for matching between $R(\bm{x})$ and $T(\bm{\varphi}(\bm{x}))$, while if too small, the corresponding $\bm{\varphi}(\bm{x})$ is not one to one. On the other hand, some experiences on image inverse applications have shown that, if the fixed penalty $\gamma$ or $\beta$ is chosen too small or too large, the solution is worse and the solution time can increase significantly.

We first study how our diffeomorphic image registration model is affected
when varying $\tau$. To this end, Algorithm \ref{alg:ALMM_OCRDI}
%based on CG method solving the $u$ subproblem
was tested
for a synthetic \emph{Circle-square} image with the results shown in Tab.\ref{table-parameter1}, where a pair of $256\times 256$ synthetic \emph{circle} images is needed to register into a synthetic \emph{square} image. The regularization parameter $\tau$ is varied from $1$ to $10$ in case of fixing $\beta=0.01$ and $\gamma=0.01$. The selection of suitable $\tau$ is important because it is unknown a priori and it significantly
affects on the qualities of registered images as well as the algorithm performance \cite{JPZhang2015}. However, for the range of  tested in Tab.\ref{table-parameter1},
the proposed model still obtains the satisfactory solution in a reasonable range of values $\tau$, so
for this example, the accurate selection of $\tau$ is
not needed as any $\tau$ between $2$ and $10$ can give better results, and is reasonable and recommendable. 

Next we test how our algorithm with the penalty parameter is affected for the above example when varying
the values of $\beta$ and $\gamma$; $\beta$ is varied from $0.001$ to $0.1$ and $\gamma$ is varied from $0.01$ to $q$, Tab.\ref{table-parameter1} shows that our model solves the image registration problem. As already discussed previously, the larger $\beta$ will lead to poor matching between $R$ and $T(\bm{\varphi})$, while the smaller $\gamma$ will make solution $\bm{\varphi}$ more worse.

\begin{table}[t]%\scriptsize
  \centering
  \newsavebox{\mybox}
  \begin{lrbox}{\mybox}   
    \begin{tabular}{ccccccc}
      \toprule[1.5pt]
      \textbf{Parameter}  &   &  $\rm{Re_{-}SSD}$  & $R_{min}$& ${\det}_{\min}(J(\bm{\varphi}))$&  $\det_{\max}(J(\bm{\varphi}))$ &  $\overline{\det}(J(\bm{\varphi}))$ \\
      \hline
      \multirow{5}{*}{\makecell[c]{$\beta=0.01$\\$\gamma=0.01$\\$\tau$}}  &  0.5   &  0.1128\% &  +  &  0.42  & 2.11 &1.001  \\
%     \cline{2-8}  
      &1      &   0.0141\% &  + & 0.41   &  2.17  &  1.001\\
%     \cline{2-8} 
      & 5 &   0.0086\% &  +  &  0.40  &  2.16 &  1.001  \\
%     \cline{2-8} 
      & 10 &   0.0204\%&  +  & 0.38  &   2.75   & 1.001 \\
      & 20 &  0.0816\%&    +&  0.40     &  2.13   & 1.001 \\
      \hline
      \multirow{5}{*}{\makecell[c]{$\tau=5$\\$\gamma=0.01$\\$\beta$ }}  &  0.001   &   0.0086\% &  + &  0.40  & 2.16 &1.001  \\
      %     \cline{2-8}  
      &0.01      &0.0100\%    &  +    &  0.41  &  2.12  & 1.001 \\
      %     \cline{2-8} 
      & 0.1 &  0.1035\%  & +   &  0.60   & 2.18   &  1.001  \\
      %     \cline{2-8} 
      & 1 & 1.5663\% &  +  &   0.15  &  2.63  & 1.001\\
      & 2 & 2.5158\% &   + &    0.15&   1.63    &  1.001\\
      \hline      
            \multirow{5}{*}{\makecell[c]{$\tau=5$\\$\beta=0.01$\\$\gamma$}}  & 0.001   &  0.0705\%  &  +  &    0.32&  2.17& 1.001 \\
      %     \cline{2-8}  
      &0.01      &   0.0086\% &  +  &  0.40  &  2.16  &  1.001\\ 
      %     \cline{2-8} 
      & 0.1 &  0.0184\%  &  +   &  0.45   & 2.16    &  1.001 \\
      %     \cline{2-8} 
      & 1& 0.0221\%  &  +   &  0.51   & 2.17    &  1.001 \\
      & 10 &  0.0402\% & +  &   0.52  &  2.16 &   1.001\\
      \hline      
      \toprule[1.5pt]
    \end{tabular}
  \end{lrbox}
\caption{Quantitative assessment of different modeling parameters.}
  \scalebox{0.80}{\usebox{\mybox}}\label{table-parameter1} 
\end{table}

\begin{figure}%[H]
    \begin{center}
    \includegraphics[width=0.99\textwidth]{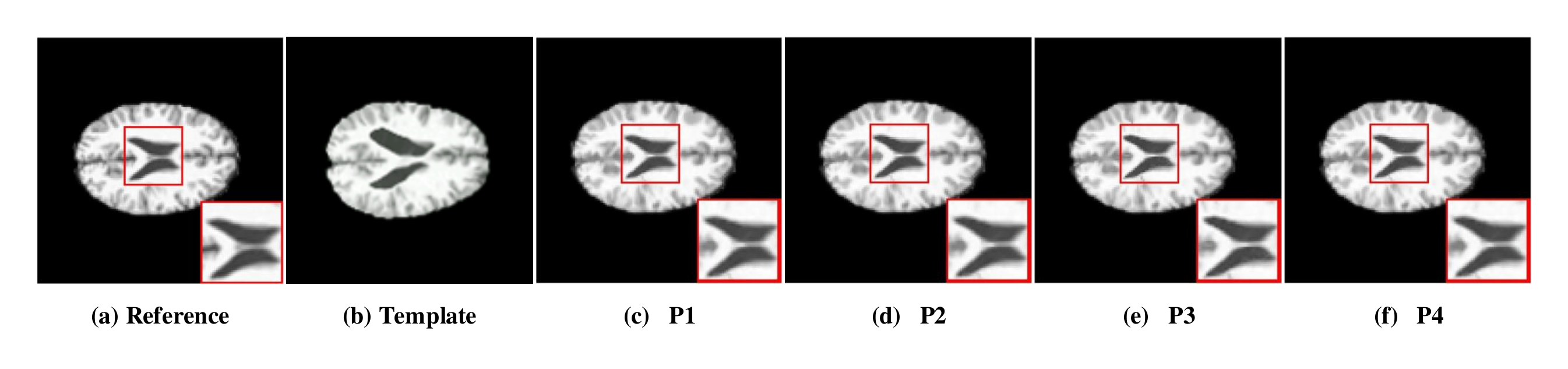}      
  \end{center}
    \caption{Comparison for different composite functions using the \emph{Whole-Brain} images. From (a) to (f): the reference image with a zoom-in region, template image, four solutions with different composite functions.
}\label{fig_test1_composite_function}
\end{figure}

\subsection{Comparison of different homotopy composite functions}

In this subsection, we briefly give the discussions on the performances of our OCRDIR model by employing
different composite functions  $\text{P}_{k}$ ($k=1,2,3,4$). Here the tested \emph{Whole-Brain} images are of size $128\times 128$ and the red boxes denote zoom-in regions, we only explain how the proposed model performs when a composite function is used,
and see more details about the composite functions in Section \ref{Sect_choice}. 

To make the comparison more fair, we set parameters that make the registration results optimal for different composite functions ($\tau=5$, $\beta=0.01$, $\gamma=0.01$ for $\mathrm{P}_{1}$ and $\mathrm{P}_{2}$, $\tau=5$, $\beta=0.1$, $\gamma=0.01$ for $\mathrm{P}_{3}$, $\tau=3$, $\beta=0.1$, $\gamma=0.01$ for $\mathrm{P}_{4}$). The solutions by using different composite functions based on OCRDIR algorithm are shown in Fig.\ref{fig_test1_composite_function} with the \emph{reference} image and \emph{template} image shown in Fig.\ref{fig_test1_composite_function} (a) and Fig.\ref{fig_test1_composite_function} (b), four registered
images from $T$ to $R$ are presented in Fig.\ref{fig_test1_composite_function} (c)-(f) for taking the time-step length $\Delta t=1/N$ ($N=40$), respectively. We can see that all four models can produce visually satisfactory results.

\begin{table}[h]%\scriptsize
  \centering
  %\newsavebox{\mybox}
  \begin{lrbox}{\mybox}          
    \begin{tabular}{cccccccc}
      \toprule[1.2pt]
      %   \hline
      \textbf{CF}& \textbf{$\overline{\det}(J(\bm{\varphi}))$}&$ R_{min} $ & $\det_\text{min}(J(\bm{\varphi}))$ & $\det_\text{max}(J(\bm{\varphi}))$    & $\textbf{\emph{ssim}}$  & $\rm{Re_{-}SSD}$ &$ \textbf{\emph{psnr}} $\\
      \hline
      \multirow{1}{*}{$\textbf{P}_{1}$}
      & 1.000  & +  & 0.13  & 3.17     &0.9900   &2.06\%  &27.94\\
      \hline
      \multirow{1}{*}{$\textbf{P}_{2}$}
      &  1.000  & +   & 0.15  & 3.17    &0.9900   &2.08\%  &27.94\\ 
      \hline
      \multirow{1}{*}{$\textbf{P}_{3}$}
      & 1.001  & +  & 0.16  & 3.92    &0.9865   &2.62\%  &26.70\\
      \hline
      \multirow{1}{*}{$\textbf{P}_{4}$}
      & 1.000  & +  & 0.16 & 3.18   &0.9897   &2.07\%  &27.94\\ 
      \toprule[1.2pt]
    \end{tabular}
  \end{lrbox}
  \caption{Quantitative assessment of different composite functions for the proposed OCRDIR model, where 'CF' means to use different homotopy composite function.} 
  \scalebox{1}{\usebox{\mybox}}\label{table1}  
\end{table}

In Tab.\ref{table1}, the grid unfolding indicator $R_{\min}>0$ shows that the four composite functions can ensure the diffeomorphism of the deformation. Further we can also see that all $\overline{\det}$ values for the four composite functions approach to 1, which indicates that the deformations are area preserving for 2D case.  The Jacobian determinants of the four composite functions all belong to a small interval $[\det_{\text{min}},\det_{\text{max}}]$, which means that the degree of area change of deformed grid is considerably small, because the new registration model explicitly controls and penalizes area change by Jacobian equation. Furthermore, $\mathrm{P}_{3}$ in this example is not as good as $\mathrm{P}_{1}$, $\mathrm{P}_{2}$ and  $\mathrm{P}_{4}$ because it gives worse \textit{ssim}, $\rm{Re_{-}SSD}$ and \textit{psnr}, which means the registered images obtained by the composite function have significant registration errors between the \emph{reference} and \emph{template} images.
So we recommend the composite function $\mathrm{P}_{1}$ as the used composite function in our model and in the following numerical experiments. 

\begin{figure}%[H]
    \begin{center}
    \includegraphics[width=0.98\textwidth]{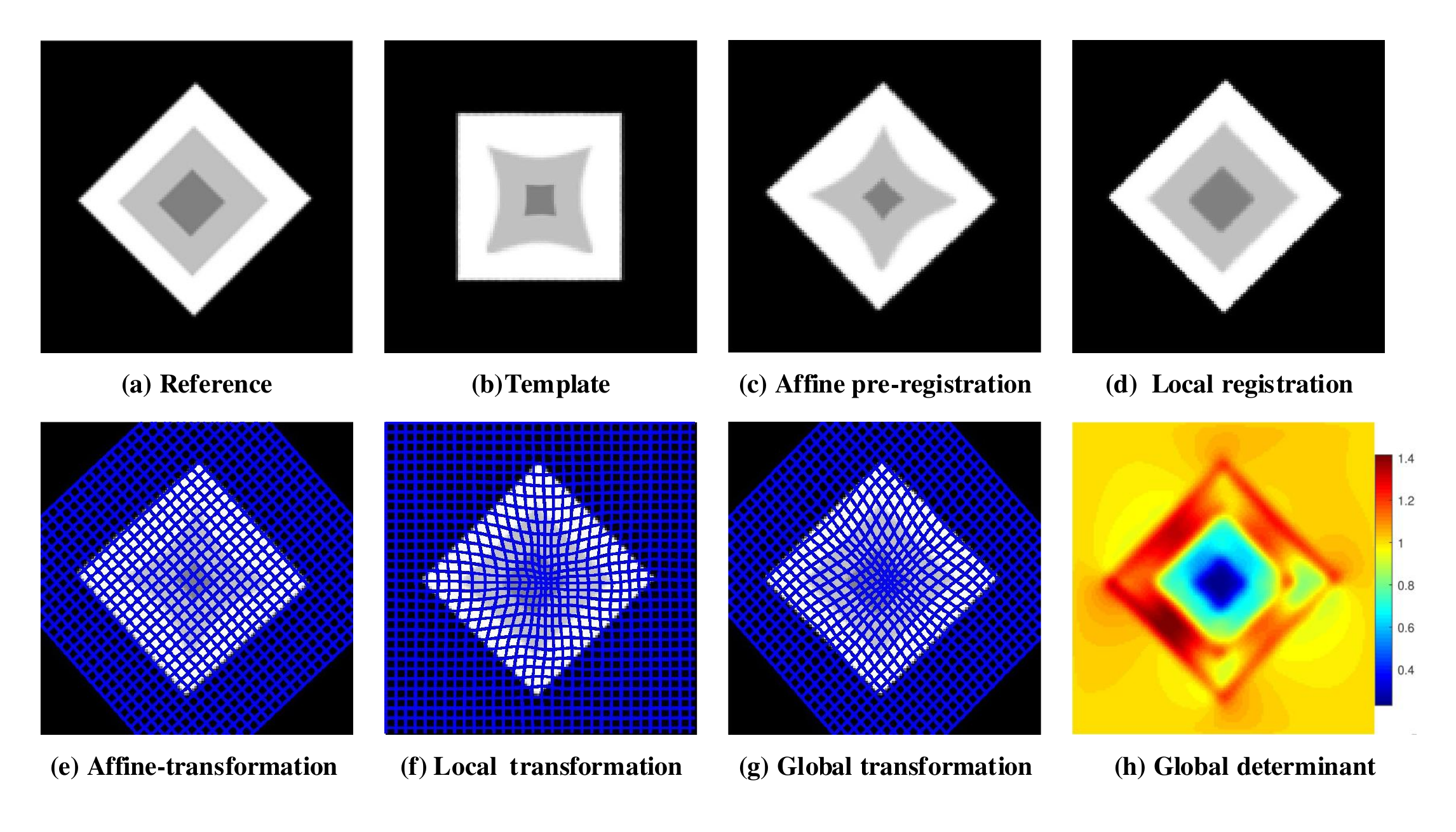}     
  \end{center}
  \caption{Registration with affine preprocessing.
  (a) the reference image; (b) the template image; (c) the pre-registration results of the affine model; (d) the proposed local registration from the template (c); (e) the affine large deformation grid; (f) the proposed local deformation grid; (g) the global deformation grid from the template (b) to the reference (a); (h) the Jacobian determinant hotmaps of global deformation $\bm{\varphi}$.}\label{fig:ALMMaffine}
\end{figure}

\subsection{Registration with affine preprocessing} Deformable registration is applicable to a large class of non-rigid registration problems, while it can suffer from worse mesh distortion for solving the large rotate deformation problems if the boundary of transforming grid is fixed, e.g., Dirichlet boundary condition $\bm{u}=\bm{0}$ on $\partial\Omega$. Affine image registration is one of the commonly-used rigid parametric models. Such affine method is always many orders of magnitude faster than a nonlinear variational method due to much less unknowns involved, which is widely used as a pre-registration step for sophistical non-rigid registration methods, such as elastic, fluid, and diffusion registration, by providing the good initial positions for the image to be registered. 

\begin{table}%[t]%\scriptsize
  \centering
% \newsavebox{\mybox}
  \setlength{\tabcolsep}{2.8pt}
  \renewcommand\arraystretch{1.3}
  \begin{lrbox}{\mybox} 
    \begin{tabular}{cccccccc}
      \toprule[1.5pt]
      %\hline
      $\overline{\det}(J(\bm{\varphi}))$&$ R_{min} $& $\det_{\min}(J(\bm{\varphi}))$ & $\det_{\max}(J(\bm{\varphi}))$ & $\textbf{\emph{ssim}}$  & $\rm{Re_{-}SSD}$ &\textbf{\emph{psnr}}&runtime (s)\\
      \hline
      0.9998  &+ & 0.21  &  1.44  & 0.9974 & 0.28\%  & 39.54& 15.17\\
      \toprule[1.5pt]
    \end{tabular}
  \end{lrbox}
  \caption{Registration with affine preprocessing.}
  \scalebox{0.86}{\usebox{\mybox}}\label{table-affine} 
\end{table}

To show the quantitative gain from using two registration steps: pre-registration step and non-rigid registration, we now present the
comparable results in Tab.\ref{table-affine} for clarity, where the grid unfolding indicator $R_{\min}>0$ and $ {\det}(J(\bm{\varphi}))\in [0.21, 1.44]$ indicate that the registration keeps the diffeomorphism. It is worth noting that the $\rm{Re_{-}SSD}$ and \emph{psnr} are 0.28\% and 39.54 respectively, and the \emph{ssim} is as high as $0.9974$. We also show the respective registration visualizations with affine preprocessing in Fig.\ref{fig:ALMMaffine}.

\begin{figure}%[H]
    \begin{center}
    \includegraphics[width=0.92\textwidth]{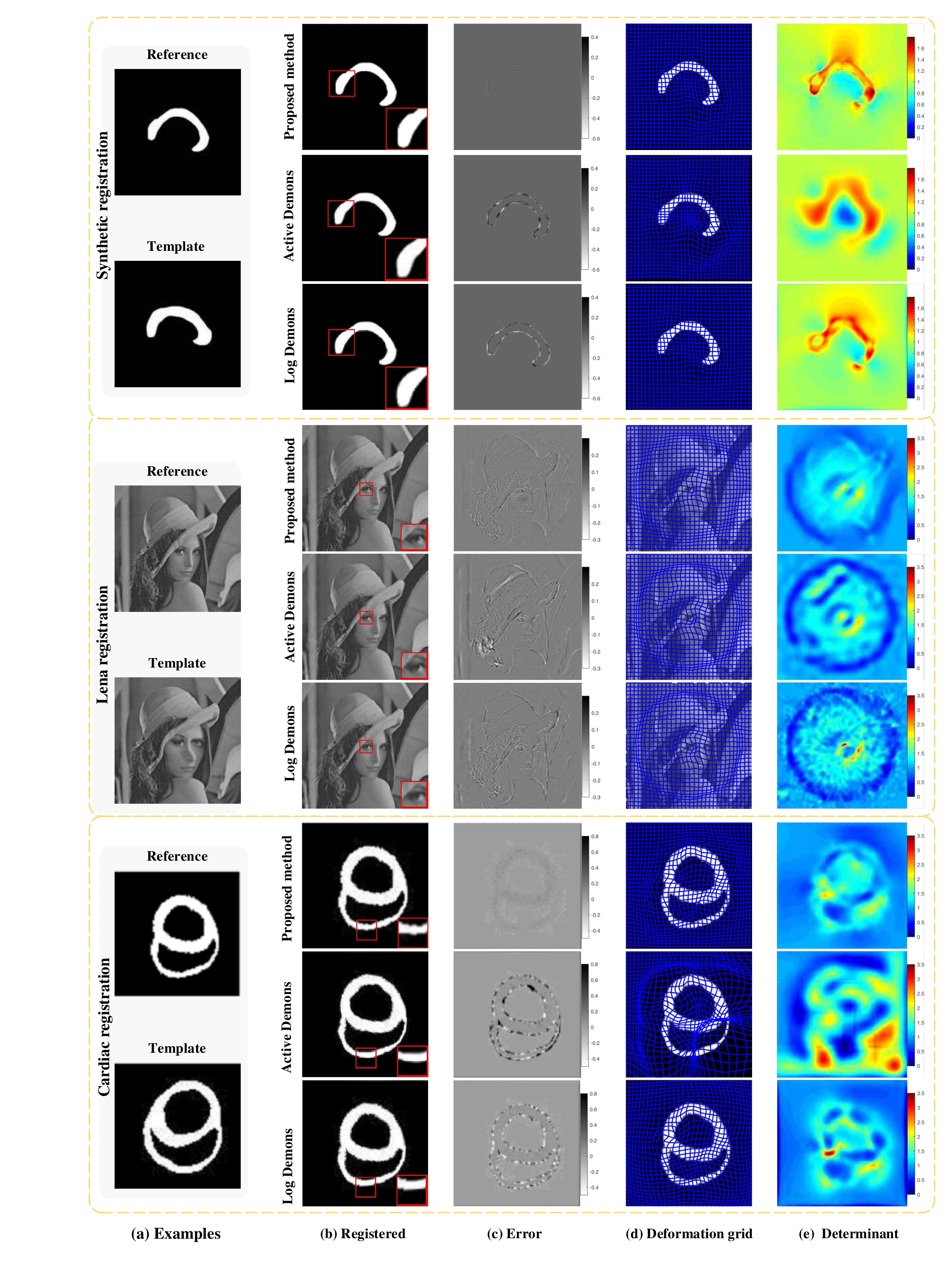}   
  \end{center}
    \caption{Comparison with active demons and diffeomorphic log-Demons. (a) the reference and template images (\emph{Synthetic}, \emph{Lena} and \emph{Cardiac} examples); (b) the registered results of the three classical models with the optimal parameters (\emph{Synthetic}: $\tau=3,\beta=0.01,\gamma=0.1$ and $N=40$ for OCRDIR model, Gaussian filter $\sigma^{2} = 10$, $\tau = 0.8$ and $ \text{MaxIter} = 200$ for active demons model, $\sigma_{\text {fluid}}=1$, $\sigma_{\text {diffusion}}=1$, $\sigma_{i}=2$, $\sigma_{x}=2$, $\text{nlevel}=5$ and $\text{niter}=40$ for Log-Demons model; \emph{Lena}: $\tau=3,\beta=0.1,\gamma=0.5$ and $N=40$ for OCRDIR model, Gaussian filter $\sigma^{2} = 6$, $\tau = 2.4$ and $ \text{MaxIter} = 200$ for active demons model, $\sigma_{\text {fluid}}=1$, $\sigma_{\text {diffusion}}=1$, $\sigma_{i}=2$, $\sigma_{x}=2$, $\text{nlevel}=5$ and $\text{niter}=40$ for Log-Demons model; \emph{Cardiac}: $\tau=2.5,\beta=0.05,\gamma=0.01$ and $N=60$ for the proposed model, Gaussian filter $\sigma^{2} = 10$, $\tau = 0.4 $ and $\text{MaxIter} = 300$ for active demons model, $\sigma_{\text {fluid}}=4$, $\sigma_{\text {diffusion}}=1$, $\sigma_{i}=2$, $\sigma_{x}=2$, $\text{nlevel}=5$ and $\text{niter}=60$ for Log-Demons model.); (c) registration errors $T(\bm{\varphi})-R(\bm{x})$; (d) the deformation grids; (e) the Jacobian determinant hotmaps of the displacement fields.}\label{fig:test2_demon-type_ex_heart}
\end{figure}

\begin{table}[t]\small
  \centering
  %\newsavebox{\mybox}
  \begin{lrbox}{\mybox}       
    \begin{tabular}{ccccccccc}
      \toprule[1.5pt]
      \setlength{\tabcolsep}{2.8pt} % µ÷Õû±í¸ñÁÐ¼äµÄ³¤¶È
      \renewcommand\arraystretch{1.1}
      \textbf{Example} & \textbf{Method}  & $\overline{\det}(J(\bm{\varphi}))$&$ R_{min} $& $\min \det(J(\bm{\varphi}))$ & $\max \det(J(\bm{\varphi}))$ & $\textbf{\emph{ssim}}$  & $\rm{Re_{-}SSD}$ &\textbf{\emph{psnr}}\\
      \hline
      \multirow{3}{*}{\textbf{Synthetic}}  & Proposed           & 1.000  &+ & 0.45   &  1.95   & \textbf{1.0000} &\textbf{0.001}\%  &\textbf{52.19}\\
      %     \cline{2-8}  
      &Active Demons   & 1.008  & + &  0.31  &  1.82   & 0.9949 &1.50\% & 19.65 \\
      %     \cline{2-8} 
      &Log Demons      & 1.015   & +  & 0.64   &  1.60   & 0.9932 &1.70\%  &19.55\\
      \hline
      \multirow{3}{*}{\textbf{Lena}}  & Proposed  &  1.001 & +  & 0.48   &  2.58   & 0.9505 & \textbf{3.78}\%  &\textbf{26.84}\\
      %     \cline{2-8}
      &Active Demons   &  1.079 & +  &  0.30  &  2.37   &0.9067& 9.50\%  &22.91\\
      %     \cline{2-8} 
      &Log Demons       &  1.072 & +  & 0.29 &  3.52   & 0.9330 & 7.12\%  &24.36\\
      \hline
      \multirow{3}{*}{\textbf{Cardiac}} & Proposed   & 1.001 & + & 0.33   & 2.96   &  0.8369  &\textbf{1.06}\%  &\textbf{21.01}\\
      %     \cline{2-8}
      & Active Demons  &  1.539  & - & \textbf{-0.02  }  &  3.09   & \textbf{0.8459}  &6.21\%  &13.48\\
      %     \cline{2-8} 
      &Log Demons   &  1.107   & + & 0.09    &  3.52  &  0.8205  &3.38\%  &16.18 \\
      \toprule[1.5pt]
    \end{tabular}
  \end{lrbox}
  \caption{Comparison of the new model with Demons} 
  \scalebox{0.86}{\usebox{\mybox}}\label{table-small}  
\end{table}

\subsection{Comparison of the new model with Demons-type models} Active demons algorithm is an automatic and efficient registration algorithm, which uses all the image information to avoid the error caused by artificial intervention. It also has certain advantages in image registration when the \emph{reference} image gradient is very small and the deformation $\bm{\varphi}$ is relatively large. It is an important algorithm in image registration. Log-Demons algorithm inherits the ability of the demons algorithm in dealing with large deformation registration, it also can keep diffeomorphism of the deformation $\bm{\varphi}$. So we compare with the demons-type models to conclude that the new model has better performance for small or large deformation. In this example, we test three pairs of images, including \emph{Synthetic} images, \emph{Lena} images and \emph{Cardiac} images of resolution $128\times 128$. We compare the proposed OCRDIR model with the active demons and diffeomorphic log-demons to demonstrate the performance of our model.
 
The proposed ALMM-based OCRDIR algorithm is to solve a \emph{time-dependent} optimization problem, where the
maximum number of outer iterations of time-step $t_j$ is set to $N=40$ (i.e., $ t_j\in[0,1], j=1,\cdots,N$),  the
maximum number of inner iterations of our algorithm is set to $\rm{MaxIter}=5$ and
the tolerance for the
relative residual is set to $10^{-6}$. To make a fair comparison in implementation, we set the iteration of active demons algorithm to N$\times$MaxIter as well. The diffeomorphic log-demons uses a multi-resolution strategy for better registration, we let the level of multi-resolution be $\text{nlevel}=5$, $\text{niter}=\text{MaxIter}$.

In \emph{Synthetic} and \emph{Lena} examples, Fig.\ref{fig:test2_demon-type_ex_heart} shows that all three methods can produce satisfactory visualizations. The proposed method can present more image details which can be seen from the zoom-in images (\emph{see} Fig.\ref{fig:test2_demon-type_ex_heart}(b)), and obtain better registration performances, which not only is qualitatively shown by visual inspection of the image mismatch errors (\emph{see} Fig.\ref{fig:test2_demon-type_ex_heart}(c)), but also is quantitatively analyzed. As has been seen, the $\rm{Re_{-}SSD}$, $\emph{ssim}$ and \emph{psnr} of the new model are better than active demons and diffeomorphic log-demons, where the $\emph{ssim}$ and \emph{psnr} reach $1.000$ and $52.19$ respectively, and the $\rm{Re_{-}SSD}$ is as low as $0.001\% $. For two examples with small deformation in Fig.\ref{fig:test2_demon-type_ex_heart}(d), all models can generate diffeomorphic deformations and obtain smooth grid. It also can be seen from Tab.\ref{table-small} that all models produce the grid unfolding indicator $R_{\min}>0$ in \emph{Synthetic} and \emph{Lena} examples, i.e. the deformations obtained by these three models are diffeomorphic.
 We can notice that the transformed grids generated by the three models are relatively smooth for the small deformation, but the transformed grids obtained by active demons are winding for the large deformation obviously. Fig.\ref{fig:test2_demon-type_ex_heart}(e) shows the hotmaps of the Jacobian determinant, where the right color bars illustrate the range of the Jacobian determinant values at each pixel $\bm{x}$ (\emph{see} Tab.\ref{table-small} for more details).

However, for large deformation example in \emph{Cardiac} in Fig.\ref{fig:test2_demon-type_ex_heart}, it is known from Tab.\ref{table-small} that the deformation obtained by the active demons is non-diffeomorphic due to the grid unfolding indicator $R_{\min}<0$, i.e. mesh folding. From Fig.\ref{fig:test2_demon-type_ex_heart}(d), we can also see that the diffeomorphic log-demons generates diffeomorphic transformation, but the $\emph{ssim}$, $\rm{Re_{-}SSD}$ and $\emph{psnr}$ are worse in Tab.\ref{table-small}.

Therefore, those examples demonstrate that our OCRDIR model can get an accurate and diffeomorphic transformation for both small and large deformation registrations.
 
\begin{figure}
    \begin{center}
    \includegraphics[width=0.99\textwidth]{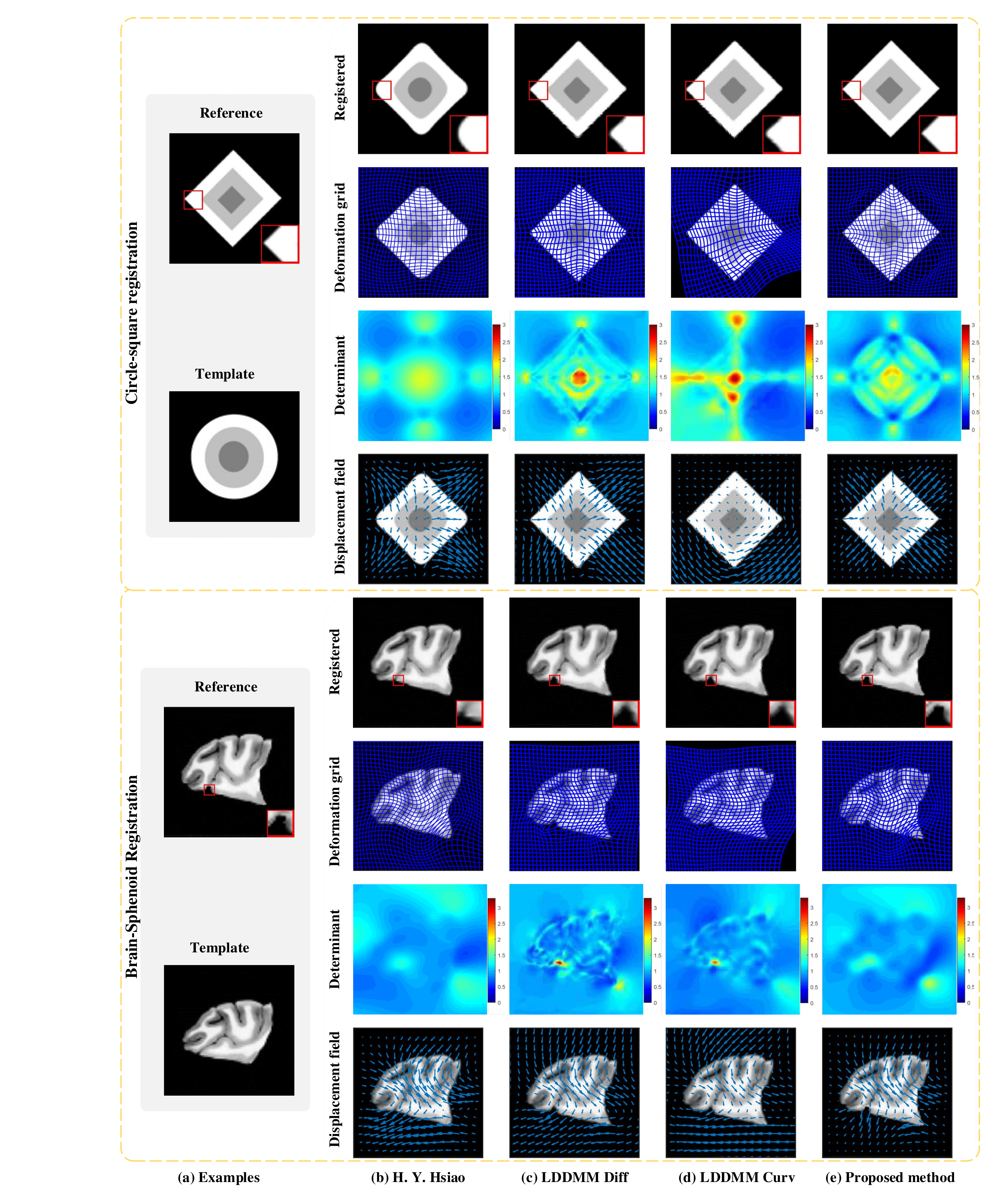}   
  \end{center}
    \caption{Comparison with three diffeomorphic variational models for \emph{Circle-square} and \emph{Brain-sphenoid} examples. (a) shows reference and template images; (b)-(e) show the results of H. Y. Hsiao'model, LDDMM diffusion model, LDDMM curvature model and the proposed OCRDIR model; from top to bottom: the registered images, the deformation grids, the Jacobian determinant hotmaps of the deformation grids and the displacement fields, respectively. The parameters are chose as follows: $\alpha= 80$,$\alpha= 0.1$ for LDDMM-Diff and LDDMM-Curv respectively, $\tau=4,\beta=0.01,\gamma=0.01$ and $N=40$
 for the proposed model in \emph{Circle-square} test; $\alpha= 100$,$\alpha= 0.1$ for LDDMM-Diff and LDDMM-Curv respectively, $\tau=4,\beta=0.05,\gamma=0.01$ and $N=40$ for the proposed model in \emph{Brain-sphenoid} example.}\label{fig:test3_diffeo-type_ex_brain}
\end{figure}

\begin{figure}
    \begin{center}
    \includegraphics[width=0.99\textwidth]{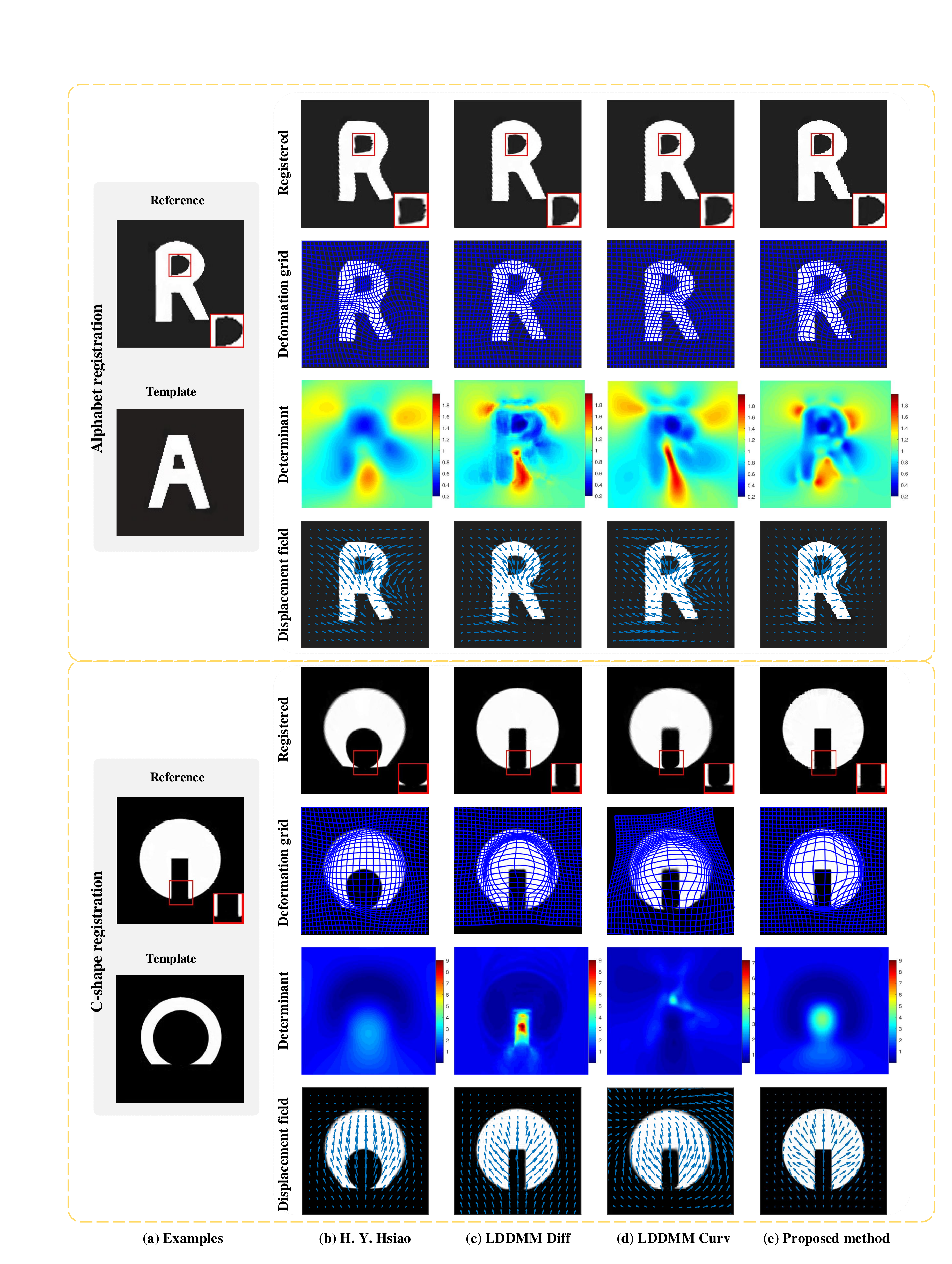}    
  \end{center}
  \caption{Comparison with three diffeomorphic variational models for large deformation image pairs: \emph{Alphabet} and \emph{C-shape}  examples. (a) shows reference and template images; (b)-(e) show the diffeomorphic registration results of H. Y. Hsiao's model, LDDMM diffusion model, LDDMM curvature model and the proposed OCRDIR model; from top to bottom: the registered images, the deformation grids, the Jacobian determinant hotmaps of the deformation grids and the displacement fields. The parameters are chose as follows: $\alpha= 300$, $\alpha= 1$ for LDDMM-Diff and LDDMM-Curv respectively, $\tau=4$, $\beta=0.01$, $\gamma=1$ and $ N= 40$ for the proposed model in \emph{Alphabet} example; $\alpha= 200$, $\alpha= 10$ for LDDMM-Diff and LDDMM-Curv respectively, $\tau=3,\beta=0.01,\gamma=0.05$ and $N=60$ for the proposed model in \emph{C-shape}  test.}\label{fig:test3_diffeo-type_ex_AR}    
\end{figure}

\subsection{Comparison of the new model with other diffeomorphic models}
In this experiment, we compare
our OCRDIR model with three popular diffeomorphic image
registration models. 
The first compared approach is naturally the Hsiao's model proposed by H. Y. Hsiao \cite{hsiao2014new}, which uses div-curl system as constraints to generate smooth and invertible deformation. The second and third compared works are the efficient LDDMM frameworks (with diffusion regularizer and curvature regularizer, respectively) \cite{MangA},
they can efficiently solve large deformation diffeomorphic registration problems.
Their approaches are different from ours since they
are focused on grid generation and a multi-resolution technique. For the LDDMM framework we use a five-level strategy in this experiment, the values for the regularization parameters are optimally chosen, and other parameters refer to default values provided in LDDMM framework to solve the problem.

 We start by giving illustrative registered reconstructions in Fig.\ref{fig:test3_diffeo-type_ex_brain}, which clearly shows the advantage of using dynamical system over the LDDMM frameworks (with diffusion regularizer and curvature regularizer, respectively) and the Hsiao's model. 
As is shown in the first row of Fig.\ref{fig:test3_diffeo-type_ex_brain}, our method perfectly registers the sharp edges of the large deformation images, the registered images are almost the ground truth \emph{reference} images. On the other hand, the Hsiao's model just registers the smooth faces, while the LDDMM can register almost the ground truth references except some place "under-enhancement".
The second and third rows of Fig.\ref{fig:test3_diffeo-type_ex_brain} shows the transformations and hotmaps of Jacobian determinant $\det(J(\bm{\varphi}))$ by four models.
In the fourth row of Fig.\ref{fig:test3_diffeo-type_ex_brain},
we see that our method produces the best displacement fields or optical flows, which shows a major advantage (or
better performance) of our proposed model when the displacement between the \emph{template} and \emph{reference} images is of large deformation,
and even when the contrast between meaningful objects and the background is low.
But in the LDDMM framework with curvature regularization, the deformation grids are obviously distorted. Besides, the Jacobian hotmap has a large range of changes.

To present more quantitative differences, we compare all four models by computing the grid unfolding indicator $R_{\min}$, the average value $\overline{\det}(J(\bm{\varphi}))$, minimizer $\det_{\min}(J(\bm{\varphi}))$ and maximizer $\det_{\max}(J(\bm{\varphi}))$ of Jacobian determinant $\det(J(\bm{\varphi}))$, also compare the $\emph{ssim}$, $\text{Re\_SSD}$ and $\emph{psnr}$ values between the registered image and \emph{reference} image, respectively. We show the results in Tab.\ref{table-large}. Here all models can generate diffeomorphic solutions as judged by the values of the grid unfolding indicator $R_{\min}$.
While the range of the relative area change is considerably smallest of the Hsiao's model ($\det(J(\bm{\varphi}))\in[0.61,1.43]$ and $\det(J(\bm{\varphi}))\in[0.65,1.74]$) among the four methods, which leads to over-preservation of area, hence it give worse $\rm{Re_{-}SSD}$, $\emph{ssim}$ and $\emph{psnr}$. The LDDMM model with the diffusion regulariser leads to further improvements over
the Hsiao's model but exhibits
the bigger mesh change near large deformation region (e.g.,  $\det(J(\bm{\varphi}))\in[0.51,3.22]$ for \emph{Brain-Sphenoid} example). The proposed method leads to significantly
better results, the range of the relative area change is considerably smaller than whose of the LDDMM, e.g.,  $ {\det}(J(\bm{\varphi}))\in [0.24, 2.77]$, $ {\det}(J(\bm{\varphi}))\in [0.34, 2.15]$ for the \emph{Brain-sphenoid}, \emph{Circle-square}, respectively. In addition, in the \emph{Brain-sphenoid} and \emph{Circle-square} examples, our model generates better $\rm{Re_{-}SSD}$, \emph{ssim} and \emph{psnr} than the LDDMM. Especially in the \emph{Circle-square} experiment, our model obtains the best  $\rm{Re_{-}SSD}$ of $\text{0.008\%}$, $\emph{ssim}$ of \text{0.9994} and $\emph{psnr}$ of \text{45.58}. This means that the deformed image is exactly the same as the \emph{reference} image.

\begin{table} \footnotesize % \small
  \centering
  \setlength{\tabcolsep}{2pt} % µ÷Õû±í¸ñÁÐ¼äµÄ³¤¶È
  \renewcommand\arraystretch{1.1}
  %\newsavebox{\mybox}
  \begin{lrbox}{\mybox}   
    \begin{tabular}{ccccccccccc}
      \toprule[1.5pt]
      \textbf{Example} & \textbf{Method} & $\overline{\det}(J(\bm{\varphi}))$  & $ R_{min} $ & $\det_{\min}(J(\bm{\varphi}))$ & $\det_{\max}(J(\bm{\varphi}))$ & $\textbf{\emph{ssim}}$  & $\rm{Re_{-}SSD}$ &\textbf{\emph{psnr}}&\textbf{\emph{Runtime(s)}}& \textbf{\emph{Memory(Mb)}}\\
      \hline
      \multirow{4}{*}{\makecell{\textbf{Brain-}\\ \textbf{sphenoid}}}  & Proposed            &  \textbf{1.000 } & + &0.24  &  2.77   & \textbf{0.9622 }  & \textbf{0.44}\%  & \textbf{28.31}&10.15 &16.06 \\
      %     \cline{2-8}  
      & Hsiao       &  1.030  & +&  0.61  &  1.43   &  0.9195 & 5.10\%  & 17.72&30.44&9.25\\
      %     \cline{2-8} 
      & LDDMM-Diff &  0.971  & +&  0.51  &   3.22  &  0.9444  & 0.94\%  &24.99&132.03&20.29\\
      %     \cline{2-8} 
      & LDDMM-Cur &  0.920  &  +& 0.56  &  2.66   &  0.9459 & 1.02\% & 24.67&205.55&40.68\\
      \hline
      \multirow{4}{*}{\makecell{\textbf{Circle-}\\\textbf{square}}}  & Proposed            &  1.001  & +&  0.34  &  2.15   & \textbf{0.9994}  & \textbf{0.008}\%  &\textbf{45.58}&15.21&16.23\\
      %     \cline{2-8}
      & Hsiao       &  1.061  & +&  0.65  &  1.74   &  0.9372 & 3.60\%  &19.02&30.30&11.82\\
      %     \cline{2-8} 
      & LDDMM-Diff &  1.079  & +&  0.56  &  2.54   &  0.9892  & 0.76\%  &25.83&196.49&20.44\\
      %     \cline{2-8} 
      & LDDMM-Curv &  1.085  & +&  0.49  &  2.85   &  0.9872  & 0.87\%  &25.26&226.01&42.71\\
      \hline
      %-------------
      \multirow{4}{*}{\textbf{Cardiac}}  & Proposed            & 1.001 & +& 0.33   &  2.96   &  \textbf{0.8369} &\textbf{1.04}\%  &\textbf{21.01}&15.97 &12.44\\
      %     \cline{2-8}
      & Hsiao       &  1.073 & + & 0.60  &  2.45   &  0.7854  & 5.28\%  &14.20 &45.58 &10.83\\
      %     \cline{2-8} 
      & LDDMM-Diff &  1.084  & + & 0.26  &  3.85   &  0.8347  & 1.69\%  &19.13 &221.80 &21.23\\
      %     \cline{2-8} 
      & LDDMM-Curv &  1.115  & + & 0.52  &  2.92   &  0.8094  & 3.26\%  &16.29 &414.74 &39.61\\
      \hline
      \multirow{4}{*}{\textbf{A-R}}  & Proposed            &  \textbf{ 0.999} &+ & 0.20  & 1.87 & \textbf{0.9949 }  & \textbf{0.05}\%  & \textbf{36.82}& \textbf{11.16}&11.29\\
      %     \cline{2-10}  
      & Hsiao      &   1.047  & + & 0.40  &  1.53    &  0.9392  & 8.40\%  &14.93  & 59.49 &10.23\\
      %     \cline{2-10} 
      & LDDMM-Diff &  1.010   & + & 0.20  &  1.87    &  0.9820    & 1.58\%  & 22.11 & 375.95 &24.50\\
      & LDDMM-Curv &  1.013   & + & 0.26  &  2.00    &  0.9778  & 2.10\%  &20.89  & 406.23 &45.04\\
      \toprule[1.5pt]
    \end{tabular}
  \end{lrbox}
  \caption{Comparison of the new OCRDIR model with Hsiao's model and LDDMM frameworks (with the diffusion regularizer and the curvature regularizer).}
  \scalebox{0.86}{\usebox{\mybox}}\label{table-large} 
\end{table}

This is due to the fact that our model explicitly controls and penalizes area change by the control increment constraint with time evolution and its regularization (\ref{eq:eq13}).
The reason is that the new OCRDIR model tries to approximate Jacobian determinant $\det(J(\bm{\phi}(\bm{x},t)))\approx 1$ of $\bm{\phi}(\bm{x},t)$ based on high order smooth regulariser, which is clearly better in this case; in
other words, our approach is more effective in eliminating the error $T(\bm{\phi}(\bm{x},t))-R(\bm{x})$ for smooth deformation and is
competitive with high order methods. Hence, those examples demonstrate that our model has advantages over other models for large deformation registrations and can all help to get an accurate deformed image and diffeomorphic transformation.

At the end of this part, we perform the registration by \emph{Alphabet} pairs and \emph{C-shape} pairs with larger deformation by the popular diffeomorphic models and our OCRDIR model. The registered images, the deformation
grids, the Jacobian determinant hotmaps of the deformation grids and the displacement fields are shown in Fig.\ref{fig:test3_diffeo-type_ex_AR}. The both experiments show that our algorithm can deal with the details better.

\subsection{3D registration experiment}
In \cite{burger_HyperelasticRegularization_2013}, the cofactor
matrix is employed to quantify surface changes, where the solution is defined by
\begin{equation}
\begin{split}
\mathcal{A}_{0}:=\{y& \in W^{1,2}\left(\Omega, \mathbb{R}^{3}\right): \\
&\left.\operatorname{cof} \nabla y \in L_{4}\left(\Omega, \mathbb{R}^{3 \times 3}\right), \operatorname{det} \nabla y \in L_{2}(\Omega, \mathbb{R}), \operatorname{det} \nabla y>0 \text { a.e. }\right\} .
\end{split}
\end{equation}
The existence theory of diffeomorphic solutions of  registration problem using standard arguments from the theory of the hyperelastic regularization is build in their work \cite{burger_HyperelasticRegularization_2013}. The solution properties of our proposed model has been discussed in the previous Section, here we focus on a most challenging large deformation registration.

In this part, we will conduct two registration experiments with $128\times 128\times 128$ brain images and $64\times 64\times 64$ synthetic images to verify the effectiveness of the proposed approach for 3D image registration. The relevant parameters in the experiment are fixed as follows: $N=20$, $\text{MaxIter}=5$, $\tau=5$, $\beta=0.01$, $\gamma=1$. The experimental results are shown in Fig.\ref{fig:test3_diffeo-type_ex_brain3D}, and the performance indicators are list in Tab.\ref{table-3D}. 
\begin{figure}
    \begin{center}
    \includegraphics[width=0.99\textwidth]{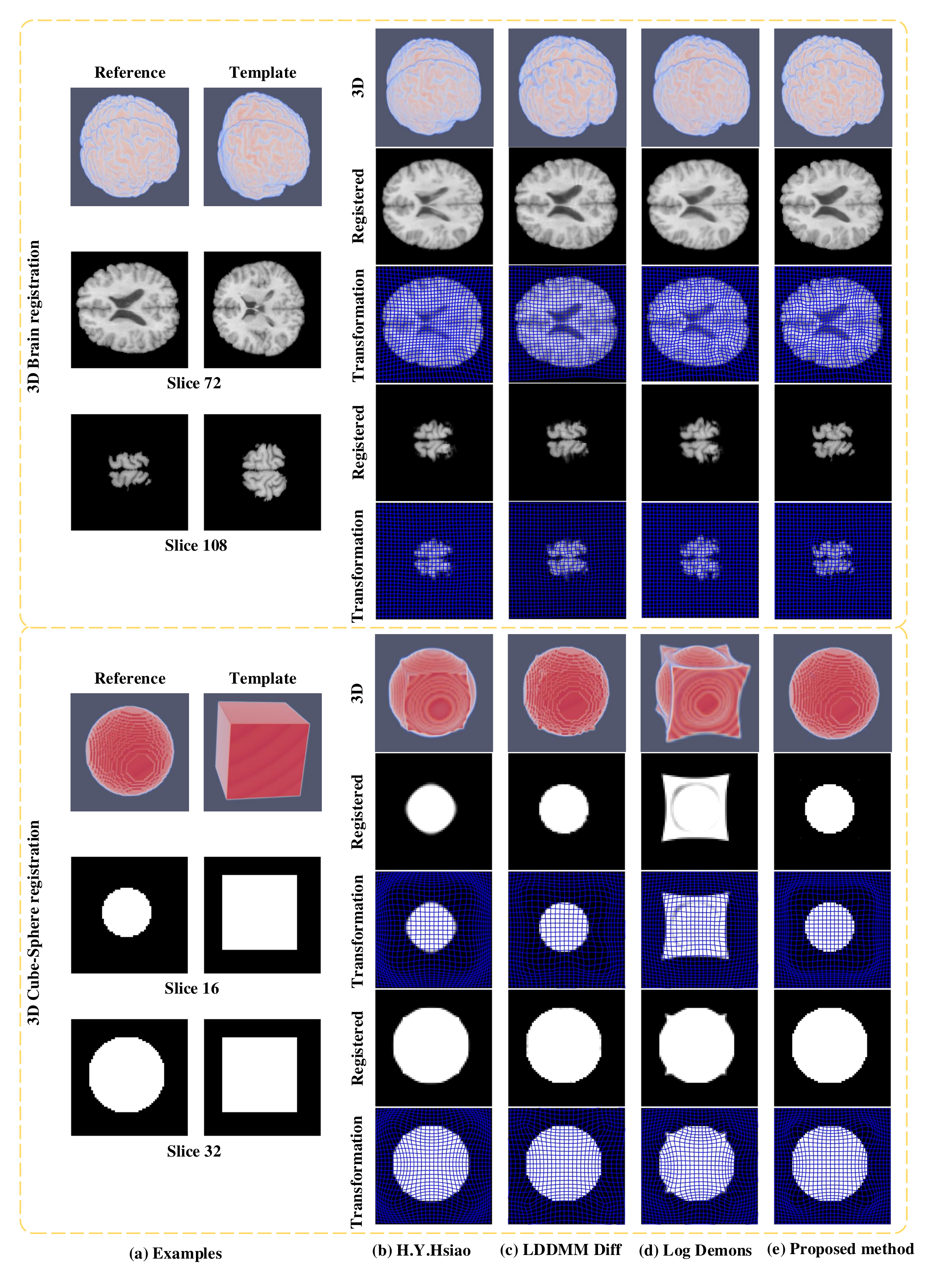}     
  \end{center}
  \caption{Comparison of 3D registration. (a) the reference images and template images of 3D Brain and Cube-Sphere examples, as well as different slices; (b)-(e) the experimental results of H. Y. Hsiao, LDDMM Diff, Log Demons and the proposed method,  respectively.}\label{fig:test3_diffeo-type_ex_brain3D}    
\end{figure}

\begin{table}%[t]%\scriptsize
  \centering
  \begin{lrbox}{\mybox} 
    \begin{tabular}{cccccccc}
      \toprule[1.5pt]
      %\hline
      Example &$\overline{\det}(J(\bm{\varphi}))$&$ R_{\min} $& $\det_{\min}(J(\bm{\varphi}))$ & $\det_{\max}(J(\bm{\varphi}))$ & $\textbf{\emph{ssim}}$  & $\rm{Re_{-}SSD}$ &\textbf{\emph{psnr}}\\
      \hline
      3D-Brain &0.9992  &+ &  0.06  &  4.23   & 0.9215 & 0.34\%  & 18.71\\
      %     \cline{2-10}  
      \toprule[1.5pt]
    \end{tabular}
  \end{lrbox}
  \caption{Comparison for the 3D experiments.}
  \scalebox{0.86}{\usebox{\mybox}}\label{table-3D} 
\end{table}

3D results shows the significance and the effect of using the proposed approach. Carefully comparing Figs.\ref{fig:test3_diffeo-type_ex_brain3D}(b)–(e), we notice that our new model deals with the brain and synthetic images better than the other three models, especially in reducing the dissimilarity.

\section{Conclusion}\label{sec:conclusions}
Maintaining the continuity and reversibility of the deformation field is a key issue in image registration modeling, and the deformation with diffeomorphism just meets this requirement. In this paper, the proposed OCRDIR model is based on the Jacobian equation with diffeomorphic solution as a constraint condition. Combined with the theoretical proof, it is known that the new OCRDIR model can always ensure the diffeomorphism of the deformation field during the registration. We analyzed an efficient ALMM  scheme to efficiently solve the relaxation optimization problem arising in our OCRDIR model. We gave a grid unfolding indicator for cell central-difference discretization, then a robust ALMM-based OCRDIR algorithm for using the deformation correction and backtrack
strategy is proposed to guarantee that the solution is diffeomorphic. The numerical results demonstrate that our ALMM-based OCRDIR algorithm can efficiently register both small and large deformation image registration. Our future work will focus on the development of faster optimization algorithm and the registration model in multi-modality diffeomorphic image registration.

\bibliographystyle{siam}
\bibliography{diffeo_reg_ref}
\end{document}